\documentclass[10pt]{amsart}

\usepackage{amsmath, amsfonts, amssymb,mathrsfs,color}
\usepackage{graphicx}
\usepackage{xypic}

\newtheorem{theorem}{Theorem}[section]
\newtheorem{lemma}[theorem]{Lemma}
\newtheorem{corollary}[theorem]{Corollary}
\newtheorem{proposition}[theorem]{Proposition}

\theoremstyle{definition}
\newtheorem{definition}[theorem]{Definition}
\newtheorem{example}[theorem]{Example}

\theoremstyle{remark}
\newtheorem{remark}[theorem]{Remark}

\numberwithin{equation}{section}

\begin{document}

\title{Equivariant Seiberg--Witten--Floer cohomology}

\author{David Baraglia}
\address{School of Mathematical Sciences, The University of Adelaide, Adelaide SA 5005, Australia}
\email{david.baraglia@adelaide.edu.au}

\author{Pedram Hekmati}
\address{Department of Mathematics, The University of Auckland, Auckland, 1010, New Zealand}
\email{p.hekmati@auckland.ac.nz}

\begin{abstract}

We develop an equivariant version of Seiberg--Witten--Floer cohomology for finite group actions on rational homology $3$-spheres. Our construction is based on an equivariant version of the Seiberg--Witten--Floer stable homotopy type, as constructed by Manolescu. We use these equivariant cohomology groups to define a series of $d$-invariants $d_{G,c}(Y,\mathfrak{s})$ which are indexed by the group cohomology of $G$. These invariants satisfy a Fr\o yshov-type inequality under equivariant cobordisms. Lastly we consider a variety of applications of these $d$-invariants: concordance invariants of knots via branched covers, obstructions to extending group actions over bounding $4$-manifolds, Nielsen realisation problems for $4$-manifolds with boundary and obstructions to equivariant embeddings of $3$-manifolds in $4$-manifolds.

\end{abstract}



\date{\today}


\maketitle


\section{Introduction}

In this paper we develop an equivariant version of Seiberg--Witten--Floer cohomology for rational homology $3$-spheres equipped with the action of a finite group. Our approach is modelled on the construction of a Seiberg--Witten--Floer stable homotopy type due to Manolescu \cite{man} which we now briefly recall. Let $Y$ be a rational homology $3$-sphere and $\mathfrak{s}$ a spin$^c$-structure on $Y$. Given a metric $g$ on $Y$, the construction of \cite{man} yields an $S^1$-equivariant stable homotopy type $SWF(Y , \mathfrak{s} , g)$. The Seiberg--Witten--Floer cohomology of $(Y , \mathfrak{s})$ is then given (up to a degree shift) by the $S^1$-equivariant cohomology of $SWF(Y,\mathfrak{s},g)$:
\[
HSW^*(Y , \mathfrak{s}) = \widetilde{H}^{*+2n(Y,\mathfrak{s},g)}_{S^1}( SWF(Y , \mathfrak{s} , g) ),
\]
where $n(Y,\mathfrak{s},g)$ is a rational number given by a certain combination of eta invariants.

The stable homotopy type $SWF(Y,\mathfrak{s},g)$ depends on the choice of metric, but only up to a suspension. Given two metrics $g_0,g_1$ one obtains a canonical homotopy equivalence
\begin{equation}\label{eq:swfsf}
SWF(Y,\mathfrak{s},g_1) \cong \Sigma^{SF( \{ D_s \} ) \mathbb{C}} SWF(Y , \mathfrak{s},g_0),
\end{equation}
where $SF( \{D_s \})$ denotes the spectral flow for the family of Dirac operators $\{ D_s \}$ determined by a path of metrics $\{ g_s \}$ from $g_0$ to $g_1$. The rational numbers $n(Y,\mathfrak{s},g)$ are defined in such a way that they split the spectral flow in the sense that
\begin{equation}\label{eq:sfsplit}
SF( \{ D_s \} ) = n(Y , \mathfrak{s} , g_1) - n(Y , \mathfrak{s} , g_0).
\end{equation}
Hence we obtain a canonical isomorphism
\[
\widetilde{H}^{*+2n(Y,\mathfrak{s},g_1)}_{S^1}( SWF(Y , \mathfrak{s} , g_1) ) \cong \widetilde{H}^{*+2n(Y,\mathfrak{s},g_0)}_{S^1}( SWF(Y , \mathfrak{s} , g_0) ).
\]
This shows that the Seiberg--Witten--Floer cohomology $HSW^*(Y , \mathfrak{s})$ does not depend on the choice of metric $g$.

By working in an appropriately defined $S^1$-equivariant Spanier--Whitehead category in which suspension by fractional amounts of $\mathbb{C}$ is allowed, Manolescu defined the Seiberg--Witten--Floer homotopy type of $(Y,\mathfrak{s})$ to be
\[
SW(Y,\mathfrak{s}) = \Sigma^{-n(Y,\mathfrak{s},g)\mathbb{C}} SWF(Y , \mathfrak{s} , g).
\]

This is independent of the choice of $g$ by (\ref{eq:swfsf}) and (\ref{eq:sfsplit}).

Now suppose that a finite group $G$ acts on $Y$ by orientation preserving diffeomorphisms which preserve the isomorphism class of $\mathfrak{s}$. Let $g$ be a $G$-invariant metric on $Y$. Lifting the action of $G$ to the associated spinor bundle determines an $S^1$ extension
\[
1 \to S^1 \to G_{\mathfrak{s}} \to G \to 1.
\]

Manolescu's construction of the stable homotopy type $SWF(Y,\mathfrak{s},g)$ can be carried out $G_{\mathfrak{s}}$-equivariantly, so that $SWF(Y , \mathfrak{s} , g)$ may be promoted to a $G_{\mathfrak{s}}$-equivariant stable homotopy type. This is analogous to the construction in \cite{man2} of the $Pin(2)$-equivariant Seiberg--Witten--Floer stable homotopy type of $(Y , \mathfrak{s})$ where $\mathfrak{s}$ is a spin-structure on $Y$. The main difference is that in our construction, the additional symmetries that comprise the group $G_{\mathfrak{s}}$ come from symmetries of $Y$ rather than internal symmetries of the Seiberg--Witten equations.

We define the {\em $G$-equivariant Seiberg--Witten--Floer cohomology} of $(Y,\mathfrak{s})$ to be
\[
HSW^*_G(Y , \mathfrak{s}) = \widetilde{H}^{*+2n(Y,\mathfrak{s},g)}_{G_{\mathfrak{s}}}( SWF(Y , \mathfrak{s} , g) ).
\]
The right hand side is independent of the choice of metric $g$ by much the same argument as in the $S^1$-equivariant case.

We make some remarks concerning this construction.
\begin{itemize}
\item[(1)]{In this paper we have chosen to work throughout with cohomology instead of homology. This is simply a matter of preference and we could just as well work with Seiberg--Witten--Floer homology groups.}
\item[(2)]{Instead of Borel equivariant cohomology, we could take co-Borel cohomology or Tate cohomology, which correspond to the different versions of Heegaard--Floer cohomology, \cite[Corollary 1.2.4]{lima}.}
\item[(3)]{In a similar fashion we can also define the $G$-equivariant Seiberg--Witten--Floer $K$-theory
\[
KSW^*_G(Y , \mathfrak{s}) = \widetilde{K}^{*+2n(Y,\mathfrak{s},g)}_{G_{\mathfrak{s}}}( SWF(Y , \mathfrak{s} , g) ).
\]
More generally we could use any generalised equivariant cohomology theory in which the Thom isomorphism holds.}
\item[(4)]{We have not attempted to construct a metric independent $G_{\mathfrak{s}}$-equivariant stable homotopy type. To do this one would need to split the equivariant spectral flow $SF_{G_{\mathfrak{s}}}( \{ D_s \} )$ in the same way that $n(Y,\mathfrak{s},g)$ splits the non-equivariant spectral flow, as in Equation (\ref{eq:sfsplit}). We suspect that in general there are obstructions to carrying this out.}
\end{itemize}

\subsection{Main results}

Throughout we work with cohomology with coefficients in a field $\mathbb{F}$. To avoid the necessity of local systems we assume that either $\mathbb{F} = \mathbb{Z}/2\mathbb{Z}$, or that the order of $G$ is odd (see Section \ref{sec:assumption}). We now outline the main properties of equivariant Seiberg--Witten--Floer cohomology.

{\bf Module structure.} $HSW^*_G(Y , \mathfrak{s})$ is a graded module over $H^*_{G_{\mathfrak{s}}}$ (where for a group $K$ we write $H^*_K$ for $H^*_K(pt)$). In particular if $G_{\mathfrak{s}}$ is the trivial extension then  $HSW^*_G(Y , \mathfrak{s})$ is a graded module over $H^*_G[U]$, where $deg(U)=2$.

\begin{theorem}[Spectral sequence]\label{thm:ss1}
There is a spectral sequence $E_r^{p,q}$ abutting to $HSW^*_G(Y , \mathfrak{s})$ whose second page is given by
\[
E_2^{p,q} = H^p( BG ; HSW^q(Y , \mathfrak{s}) ).
\]
\end{theorem}

\begin{theorem}[Localisation]
Suppose that the extension $G_{\mathfrak{s}}$ is trivial and choose a trivialisation $G_{\mathfrak{s}} \cong S^1 \times G$. Then $H^*_{G_\mathfrak{s}} \cong H^*_G[U]$ and the localisation $U^{-1}HSW^*_G(Y,\mathfrak{s})$ of $HSW^*_G(Y,\mathfrak{s})$ with respect to $U$ is a free $H^*_G[U,U^{-1}]$-module of rank $1$.
\end{theorem}

{\bf $L$-spaces.} We say that $Y$ is an {\em $L$-space} with respect to $\mathfrak{s}$ and $\mathbb{F}$ if $HSW^*(Y,\mathfrak{s})$ is isomorphic to a free $\mathbb{F}[U]$-module of rank $1$.

\begin{theorem}
Suppose that $G_\mathfrak{s}$ is a split extension. If $Y$ is an $L$-space with respect to $\mathfrak{s}$ and $\mathbb{F}$, then the spectral sequence given in Theorem \ref{thm:ss1} degenerates at $E_2$. Moreover we have
\[
HSW^*_G(Y , \mathfrak{s} ) \cong HSW^*(Y , \mathfrak{s}) \otimes_{\mathbb{F}} H^*_G.
\]
\end{theorem}

{\bf Correction terms.} Suppose that $G_\mathfrak{s}$ is a split extension. For each non-zero $c \in H^*_G$ we obtain an invariant
\[
d_{G,c}(Y,\mathfrak{s}) \in \mathbb{Q}
\]
which may be thought of as a generalisation to the equivariant setting of the $d$-invariant $d(Y,\mathfrak{s})$. We also set $d_{G,0}(Y,\mathfrak{s}) = -\infty$. 

\begin{theorem}
The equivariant $d$-invariants satisfy the following properties:

\begin{itemize}
\item[(1)]{$d_{G,1}(Y,\mathfrak{s}) \ge d(Y,\mathfrak{s})$, where $1$ is the generator of $H^0_G(pt)$.}
\item[(2)]{$d_{G,c_1+c_2}(Y,\mathfrak{s}) \le \max\{ d_{G,c_1}(Y,\mathfrak{s}), d_{G,c_2}(Y,\mathfrak{s}) \}$.}
\item[(3)]{$d_{G,c_1c_2}(Y,\mathfrak{s}) \le \min\{ d_{G,c_1}(Y,\mathfrak{s}), d_{G,c_2}(Y,\mathfrak{s}) \}$.}
\item[(4)]{$d_{G,c_1}(Y , \mathfrak{s}) + d_{G,c_2}( \overline{Y} , \mathfrak{s}) \ge 0$ whenever $c_1c_2 \neq 0$.}
\item[(5)]{If $Y$ is an $L$-space with respect to $\mathfrak{s}$ and $\mathbb{F}$, then $d_{G,c}(Y , \mathfrak{s}) = d(Y,\mathfrak{s})$ for all $c \neq 0$.}
\item[(6)]{$d_{G,c}(Y,\mathfrak{s})$ is invariant under equivariant rational homology cobordism.}
\end{itemize}
\end{theorem}

We find it convenient to also define corresponding equivariant $\delta$-invariants by setting 
\[
\delta_{G,c}(Y,\mathfrak{s}) = d_{G,c}(Y,\mathfrak{s})/2.
\]

Our primary motivation for considering the equivariant $d$-invariants is that they are necessary for the formulation of our equivariant generalisation of Fr\o yshov's inequality described below.

{\bf Cobordism maps.} Suppose that $(W,\mathfrak{s})$ is a $G$-equivariant cobordism from $(Y_1,\mathfrak{s}_1)$ to $(Y_2,\mathfrak{s}_2)$ (see Section \ref{sec:indcobord} for the precise statement). Then $W$ induces a morphism of graded $H^*_{G_{\mathfrak{s}}}$-modules
\[
SW_G(W,\mathfrak{s}) \colon HSW^*_G(Y_2 , \mathfrak{s}_2) \to HSW^{*+b_+(W)-2\delta(W,\mathfrak{s})}_G(Y_1, \mathfrak{s}_1)
\]
where $\delta(W,\mathfrak{s}) = (c_1(\mathfrak{s})^2 - \sigma(W))/8$.

\begin{theorem}[Equivariant Fr\o yshov inequality]
Let $W$ be a smooth, compact, oriented $4$-manifold with boundary and with $b_1(W) = 0$. Suppose that $G$ acts smoothly on $W$ preserving the orientation and a spin$^c$-structure $\mathfrak{s}$. Suppose that the extension $G_{\mathfrak{s}}$ is trivial. Suppose each component of $\partial W$ is a rational homology $3$-sphere and that $G$ sends each component of $\partial W$ to itself. Let $e \in H^{b_+(W)}_G$ be the image in $H^*_G(pt ; \mathbb{F})$ of the Euler class of any $G$-invariant maximal positive definite subspace of $H^2(W ; \mathbb{R})$. Let $c \in H^*_G$ and suppose that $c e \neq 0$.
\begin{itemize}
\item[(1)]{If $\partial W = Y$ is connected, then 
\[
\delta(W , \mathfrak{s}) \le \delta_{G,c}(Y , \mathfrak{s}|_Y) \; \; \text{ and } \; \; \delta_{G,c e}(\overline{Y} , \mathfrak{s}|_Y ) \le \delta( \overline{W} , \mathfrak{s} ).
\]
}
\item[(2)]{If $\partial W = \overline{Y}_1 \cup Y_2$ has two connected components, then 
\[
\delta_{G,c e}(Y_1 , \mathfrak{s}|_{Y_1}) + \delta(W , \mathfrak{s}) \le \delta_{G,c}(Y_2 , \mathfrak{s}|_{Y_2}).
\]
}
\end{itemize}

\end{theorem}

{\bf Knot concordance invariants.}
Let $K$ be a knot in $S^3$ and let $Y = \Sigma_2(K)$ be the double cover of $S^3$ branched over $K$. Then $Y$ has an action of $G = \mathbb{Z}_2$ generated by the covering involution. Further, $Y$ has a spin$^c$-structure $\mathfrak{t}_0$ uniquely determined by the condition that it arises from a spin structure. Set $\mathbb{F} = \mathbb{Z}_2$. Then $H^*_G \cong \mathbb{F}[Q]$, where $deg(Q)=1$. For each $j \ge 0$, we define an invariant $\delta_j(K) \in \mathbb{Z}$ by setting
\[
\delta_j(K) = 4\delta_{\mathbb{Z}_2 , Q^j}(\Sigma_2(K) , \mathfrak{t}_0).
\]

Let $\sigma(K)$ and $g_4(K)$ denote the signature and smooth $4$-genus of $K$.

\begin{theorem}
The invariants $\delta_j(K)$ have the following properties:
\begin{itemize}
\item[(1)]{$\delta_j(K)$ is a knot concordance invariant.}
\item[(2)]{$\delta_0(K) \ge \delta(K)$, where $\delta(K)$ is the Manolescu--Owens invariant \cite{mo}.}
\item[(3)]{$\delta_{j+1}(K) \le \delta_j(K)$ for all $j \ge 0$.}
\item[(4)]{$\delta_j(K) \ge -\sigma(K)/2$ for all $j \ge 0$ and $\delta_j(K) = -\sigma(K)/2$ for $j \ge g_4(K)-\sigma(K)/2$.}
\item[(5)]{$\delta_j(-K) \ge \sigma(K)/2$ for all $j \ge 0$ and $\delta_j(-K) = \sigma(K)/2$ for $j \ge g_4(K)+\sigma(K)/2$.} 
\item[(6)]{If $\Sigma_2(K)$ is an $L$-space, then $\delta_j(K) = \delta(K)=-\sigma(K)/2$ and $\delta_j(-K) = \delta(-K) = \sigma(K)/2$ for all $j \ge 0$.}
\end{itemize}

\end{theorem}

In particular, if $K$ is quasi-alternating, then $\Sigma_2(K)$ is an $L$-space \cite{os3}. So we recover the main result of \cite{liow} that $\delta(K) = -\sigma(K)/2$ for quasi-alternating knots.

The concordance invariants $\delta_j(K)$ can also be used to obtain a strengthening of the inequality $g_4(K) \ge |\sigma(K)|/2$ \cite{mur}.

\begin{theorem} 
For a knot $K$, let $j_+(K)$ be the smallest positive integer such that $\delta_j(K) =  -\sigma(K)/2$ and $j_-(K)$ the smallest positive integer such that $\delta_j(-K) =  \sigma(K)/2$. Then 
\[
g_4(K) \ge \max\{ -\sigma(K)/2+j_-(K) , \sigma(K)/2+j_+(K)\}.
\]
\end{theorem}

\begin{corollary}
If $\delta(K) > -\sigma(K)/2$ and $\sigma(K) \ge 0$, then
\[
g_4(K) \ge |\sigma(K)|/2 + 1
\]
\end{corollary}
\begin{proof}
If $\delta(K) > -\sigma(K)/2$, then $\delta_0(K) \ge \delta(K) > -\sigma(K)/2$, hence $j_+(K) \ge 1$. Hence $g_4(K) \ge \sigma(K)/2+1 =|\sigma(K)|/2+1$.
\end{proof}

One can obtain even more knot concordance invariants by considering higher order cyclic branched covers, see Remark \ref{rem:higher}.

\subsection{Applications}

We outline here some of the applications of equivariant Seiberg--Witten--Floer cohomology. These are considered in more detail in Section \ref{sec:caa}. 

\subsubsection{Non-extendable actions {\em (Section \ref{sec:app1})}} Let $Y$ be a rational homology $3$-sphere equipped with an orientation preserving action of $G$ and let $W$ be a smooth $4$-manifold which bounds $Y$. The equivariant $d$-invariants give obstructions to extending the action of $G$ over $W$. 

\begin{example}
The Brieskorn homology sphere $Y = \Sigma(p,q,r)$ where $p,q,r$ are pairwise coprime is the branched cyclic $p$-fold cover of the torus knot $T_{q,r}$. Let $\tau \colon Y \to Y$ be a generator of the $\mathbb{Z}_p$-action determined by this covering. For certain values of $p,q,r$ it can be shown that $Y$ bounds a contractible $4$-manifold, for example $\Sigma(2,3,13)$ bounds a contractible $4$-manifold \cite{ak}. It can be shown that $\tau$ is smoothly isotopic to the identity, hence it follows that $\tau$ can be extended as a diffeomorphism over any $4$-manifold bounded by $Y$. On the other hand we show in Proposition \ref{prop:casson} that if $p$ is prime then $\delta_{\mathbb{Z}_p , 1}( Y , \mathfrak{s}) = -\lambda( Y )$ is minus the Casson invariant of $Y$ (where $\mathfrak{s}$ is the unique spin$^c$-structure on $Y$), which is non-zero. We futher show that the non-vanishing of $\delta_{\mathbb{Z}_p , 1}(Y , \mathfrak{s})$ implies that $\tau$ can not be extended as a smooth $\mathbb{Z}_p$-action to any contractible $4$-manifold bounded by $Y$. This partially recovers the non-extendability results of Anvari--Hambleton \cite{ah1,ah2} for Brieskorn homology $3$-spheres bounded by contractible $4$-manifolds.

On the other hand, our non-extendability result also holds in situations not covered by Anvari--Hambleton. Suppose now that $Y = \Sigma(p,q,r)$ bounds a {\em rational} homology $4$-ball $W$. For example, Fintushel--Stern showed that $\Sigma(2,3,7)$ bounds a rational homology $4$-ball, although it does not bound an integral homology $4$-ball \cite{fs0}. More examples can be found in \cite{al,savk}. We show in Section \ref{sec:app1} that the non-vanishing of $\delta_{\mathbb{Z}_p , 1}( \Sigma(p,q,r) )$, where $p$ is prime implies that the $\mathbb{Z}_p$-action can not be extended to any rational homology $4$-ball $W$ bounded by $Y$, provided that $p$ does not divide the order of $H^2( W ; \mathbb{Z})$. 
\end{example}

\subsubsection{Realisation problems {\em (Section \ref{sec:app2})}} Let $W$ be a smooth $4$-manifold with boundary an integral homology sphere $Y$. Suppose that a finite group $G$ acts on $H^2(W ; \mathbb{Z})$ preserving the intersection form. We say that the action of $G$ on $H^2(W ; \mathbb{Z})$ can be realised by diffeomorphisms if there is a smooth orientation preserving action of $G$ on $W$ inducing the given action on $H^2(W ; \mathbb{Z})$. The equivariant $d$-invariants give obstructions to realising such actions by diffeomorphism. This extends the non-realisation results of \cite{bar1}, \cite{bar} for closed $4$-manifolds to the case of $4$-manifolds with non-empty boundary.

\begin{example} Suppose that $b_1(W)=0$, that $H^2(W ; \mathbb{Z})$ has no $2$-torsion and even intersection form. Suppose that $Y$ is an $L$-space. Suppose that an action of $G = \mathbb{Z}_p$ on $H^2(W ; \mathbb{Z})$ is given, where $p$ is prime and that the subspace of $H^2(W ; \mathbb{R})$ fixed by $G$ is negative definite. If $\sigma(W)/8 < -\delta(Y,\mathfrak{s})$ (where $\mathfrak{s}$ is the unique spin$^c$ structure on $Y$) then the action of $\mathbb{Z}_p$ on $H^2(W ; \mathbb{Z})$ is not realisable by a smooth $\mathbb{Z}_p$-action on $W$. Note that we are not making any assumptions about the action of $\mathbb{Z}_p$ on the boundary.
\end{example}

\subsubsection{Equivariant embeddings of $3$-manifolds in $4$-manifolds {\em (Section \ref{sec:app3})}} Let $Y$ be a rational homology $3$-sphere equipped with an orientation preserving action of $G$. By an equivariant embedding of $Y$ into a $4$-manifold $X$, we mean an embedding $Y \to X$ such that the action of $G$ on $Y$ extends over $X$.

\begin{example}
Let $Y = \Sigma(2,2s-1,2s+1)$ where $s$ is odd, equipped with the involution $\tau$ obtained from viewing $Y$ as the branched double cover $\Sigma_2( T_{2s-1,2s+1})$. Then $Y$ embeds in $S^4$ \cite[Theorem 2.13]{bb}. On the other hand, $\delta_j( Y , \mathfrak{s}) \neq 0$ for some $j$. We will show that the non-vanishing of this invariant implies that $Y$ can not be equivariantly embedded in $S^4$.
\end{example}

It is known that every $3$-manifold $Y$ embeds in the connected sum $\#^n (S^2 \times S^2)$ of $n$ copies of $S^2 \times S^2$ for some sufficiently large $n$ \cite[Theorem 2.1]{agl}. Aceto--Golla-Larson define the {\em embedding number} $\varepsilon(Y)$ of $Y$ to be the smallest $n$ for which $Y$ embeds in $\#^n( S^2 \times S^2)$. Here we consider an equivariant version of the embedding number. To obtain interesting results we need to make an assumption on the kinds of group actions allowed.

\begin{definition}
Let $G = \mathbb{Z}_p = \langle \tau \rangle$ where $p$ is a prime number. We say that a smooth, orientation preserving action of $G$ on $X = \#^n(S^2 \times S^2)$ is {\em admissible} if $H^2( X ; \mathbb{Z})^{\tau} = 0$, where $H^2(X ; \mathbb{Z})^{\tau} = \{ x \in H^2(X ; \mathbb{Z}) \; | \; \tau(x) = x \}$. We define the {\em equivariant embedding number $\varepsilon(Y,\tau)$} of $(Y,\tau)$ to be the smallest $n$ for which $Y$ embeds equivariantly in $\#^n (S^2 \times S^2)$ for some admissible $\mathbb{Z}_p$-action on $\#^n( S^2 \times S^2)$, if such an embedding exists. We set $\varepsilon(Y,\tau) = \infty$ if there is no such embedding.
\end{definition}

\begin{example}
Let $Y = \Sigma(2,3,6n+1) = \Sigma_2(T_{3,6n+1})$ and equip $Y$ with the covering involution $\tau$. We show that
\[
2n \le \varepsilon( \Sigma(2,3,6n+1) , \tau ) \le 12n.
\]
Suppose that $n$ is odd. Then from \cite[Proposition 3.5]{agl}, the (non-equivariant) embedding number of $\Sigma(2,3,6n+1)$ is $10$. In particular, we see that $\varepsilon( \Sigma(2,3,6n+1) , \tau ) > \varepsilon( \Sigma(2,3,6n+1) )$ for all odd $n > 5$. We also show that
\[
\varepsilon( \Sigma( 2 , 3 , 7) , \tau ) = 12,
\]
whereas $\varepsilon( \Sigma( 2 , 3 , 7) ) = 10$. 
\end{example}

\subsection{Comparison with other works}

In \cite{hls}, the authors introduce equivariant versions of several types of Floer homology, mostly focusing on the case that the group is $\mathbb{Z}_2$. In particular they define a $\mathbb{Z}_2$-equivariant version of $HF^-$, which is a module over $H^*_{S^1 \times \mathbb{Z}_2}(pt ; \mathbb{Z}_2) = \mathbb{Z}_2[U,Q]$. This construction shares many similarities with the equivariant Seiberg--Witten--Floer cohomology constructed in this paper, such as a localisation isomorphism and a spectral sequence relating the equivariant and ordinary Floer homologies. In fact, it seems reasonable to conjecture that our constructions are isomorphic.

In \cite{aks}, the authors consider a $\mathbb{Z}_2$-equivariant Heegaard--Floer homology $HFB^-(K)$ for a branched double cover $Y = \Sigma_2(K)$ of a knot $K$, constructed in a manner similar to involutive Heegaard--Floer homology \cite{hema} except that the involution arises from the covering involution on $Y$. These groups are modules over the ring $\mathbb{Z}_2[U,Q]/(Q^2)$. From this group they obtain knot concordance invariants $\overline{\delta}(K), \underline{\delta}(K)$. A similar approach was taken in \cite{dhm}, where the authors obtain $\iota$-complexes \cite[Definition 8.1]{hmz} associated to involutions on $Y$. Since $\mathbb{Z}_2[U,Q]/(Q^2) = H^*_{S^1 \times \mathbb{Z}}(pt ; \mathbb{Z}_2)$, we suspect that the group $HFB^-(K)$ may be isomorphic to the $\mathbb{Z}$-equivariant Seiberg--Witten--Floer homology of $\Sigma_2(K)$.

In \cite[Remark 3.1]{lima2}, the authors define equivariant Seiberg--Witten--Floer homology in the special case that $G$ acts freely on $Y$. Their construction coincides with ours in such cases.

\subsection{Structure of the paper}

In Section \ref{sec:swfs} we recall the construction of Seiberg--Witten--Foer spectra using finite dimensional approximation and the Conley index. In Section \ref{sec:eswfc} we extend this construction to the $G$-equivariant setting, arriving at the construction of the $G$-equivariant Seiberg--Witten--Floer cohomology in \textsection \ref{sec:gswfc}. In the remainder of Section \ref{sec:eswfc} we introduce the equivariant $d$-invariants and establish their basic properties. Section \ref{sec:cobord} is concerned with the behaviour of equivariant Seiberg--Witten--Floer cohomology and the $d$-invariants under equivariant cobordism. In Section \ref{sec:zp} we specialise to the case that $G$ is a cyclic group of prime order. In Section \ref{sec:bdc} we consider the case of branched double covers of knots with their natural involution to obtain knot concordance invariants. Finally in Section \ref{sec:caa} we carry out some explicit computations of $d$-invariants and consider various applications.

\section{Seiberg--Witten--Floer spectra}\label{sec:swfs}

\subsection{Seiberg--Witten trajectories}

Throughout we let $Y$ be a rational homology $3$-sphere, i.e. $Y$ is a compact, oriented, smooth $3$-manifold with $b_1(Y)=0$. References for the material in this section are \cite{man}, \cite{lima}.

Let $g$ be a Riemannian metric on $Y$ and let $\mathfrak{s}$ be a spin$^c$-structure with associated spinor bundle $S$. Let $\rho \colon TY \to End(S)$ denote Clifford multiplication, satisfying $\rho(v)\rho(w) + \rho(w)\rho(v) = -2g(v,w)$. The spinor bundle $S$ is equipped with a Hermitian metric $\langle \; , \; \rangle$ which we take to be anti-linear in the first variable. Let $\mathfrak{su}(S)$ be the Lie algebra bundle of trace-free skew-adjoint endomorphisms of $S$ and $\mathfrak{sl}(S)$ the Lie algebra bundle of trace-free endomorphisms of $S$. Then $\rho$ induces an isomorphism $\rho \colon TY \to \mathfrak{su}(S)$ which extends by complexification to an isomorphism $\rho \colon TY_\mathbb{C} \to \mathfrak{sl}(S)$ satisfying $\rho( \bar{v} ) = - \rho(v)^*$. Using the metric $g$ to identify $TY$ and $T^*Y$ we will also view $\rho$ as a map $\rho \colon T^*Y \to \mathfrak{su}(S)$. We extend $\rho$ to $2$-forms by the rule $\rho( v \wedge w) = \frac{1}{2} [ \rho(v) , \rho(w) ]$. It follows that $\rho(\lambda) = -\rho( *\lambda)$ for any $2$-form $\lambda$. Define a Hermitian inner product on $\mathfrak{su}(S)$ by $\langle a , b \rangle = \frac{1}{2} tr( a^*b )$. Then for any tangent vectors $u,v$, we have $\langle \rho(u) , \rho(v) \rangle = g(\overline{u},v)$. Define a map
\[
\tau \colon S \times S \to T^*Y_{\mathbb{C}}
\]
by setting
\[
\tau( \phi , \psi ) = \rho^{-1}( \phi \otimes \psi^* )_0,
\]
where $(\phi \otimes \psi^*)_0$ is the trace-free part of $\phi \otimes \psi^*$. That is, if $\xi$ is any spinor, then
\[
(\phi \otimes \psi^*)(\xi) = \phi \langle \psi , \xi \rangle - \frac{1}{2} \langle \psi , \phi \rangle \xi.
\]
Then it follows that
\[
\tau(\psi , \phi) = -\overline{ \tau(\phi,\psi) }, \; \; \tau( a \psi , b \phi ) = a \overline{b} \tau(\psi , \phi).
\]
In particular, $\tau( \phi , \phi)$ is imaginary and $\tau( c\phi , c\phi) = |c|^2 \tau(\phi , \phi)$. We also have the identity
\[
\langle \tau(\phi , \phi) , v \rangle = \frac{1}{2} \langle \phi , \rho(v)\phi \rangle,
\]
for all spinors $\phi$ and vectors $v$.

Let $L = det(S)$ be the determinant line bundle of $S$. Let $\Gamma(L)$ denote the space of $U(1)$-connections on $L$, which is an affine space over $i \Omega^1(Y)$. We will write such a connection as $2A$. Then $2A$ determines a $Spin^c(3)$-connection on $S$ whose $\mathfrak{u}(1)$ part is $A$ and whose $\mathfrak{spin}(3)$ part is the Levi-Civita connection. Abusing terminology, we will refer to $A$ as a spin$^c$-connection.

Given a spin$^c$-connection $A$, we let $D_A$ denote the associated Dirac operator on $S$. Fix a reference spin$^c$-connection $A_0$. Then we may write $A = A_0 + a$ for some $a \in i \Omega^1(Y)$. It follows that
\[
D_A(\psi) = D_{A_0 + a}(\psi) = D_{A_0}(\psi) + \rho(a)\psi.
\]
Since $b_1(Y) = 0$, it follows that $L$ admits a flat connection. We will assume that $A_0$ defines a flat connection on $L$.

We define the configuration space of $Y$ to be
\[
C(Y) = \Gamma(L) \times \Gamma(S).
\]
$C(Y)$ depends on $g$ and $\mathfrak{s}$ but we omit this from the notation. $C(Y)$ is an affine space modelled on $i \Omega^1(Y) \oplus \Gamma(S)$. In particular, the tangent space $T_{(A,\phi)} C(Y)$ to any point $(A , \phi) \in C(Y)$ can naturally be identified with $i \Omega^1(Y) \oplus \Gamma(S)$. There is a natural metric on $i\Omega^1(Y) \oplus \Gamma(S)$, the $L^2$-metric
\[
\langle ( a_1 , \phi_1 ) , (a_2 , \phi_2) \rangle_{L^2} = - \int_Y a_1 \wedge * a_2 + \int_Y Re \langle \phi_1 , \phi_2 \rangle dvol_Y.
\]
This defines a (constant) Riemannian metric on $C(Y)$. We will need to work with Sobolev completions. Given a flat reference spin$^c$-connection $A_0$, Sobolev norms are defined using $A_0$ and $g$. Fix an integer $k  \ge 4$. Later we will work with the $L^2_{k+1}$-completion of $C(Y)$ and $L^2_{k+2}$-gauge transformations.

Having fixed a reference connection $A_0$, we identify $C(Y)$ with $i \Omega^1(Y) \oplus \Gamma(S)$. Thus an element $(A , \phi) \in C(Y)$ will be identified with $(a , \phi) \in i\Omega^1(Y) \oplus \Gamma(S)$, where $A = A_0 + a$. To simplify notation, we will write $D_a$ in place of $D_{A_0 + a}$.

The {\em Chern--Simons--Dirac functional} $\mathcal{L} \colon C(Y) \to \mathbb{R}$ (with respect to $A_0$) is defined as
\[
\mathcal{L}( a , \phi ) = \frac{1}{2} \left( \int_Y \langle \phi , D_{a} \phi \rangle dvol_Y - \int_Y a \wedge da \right).
\]

The gauge group $\mathcal{G} = \mathcal{C}^\infty( Y , S^1 )$ acts on $C(Y)$ by
\[
u \cdot ( a , \phi ) = (a - u^{-1} du , u \cdot \phi ).
\]
Observe that
\[
D_{a - u^{-1}du}( u \phi ) = u D_a \phi.
\]
So
\[
\langle u \phi , D_{a - u^{-1}du} (u \phi ) \rangle = \langle u \phi , u D_a \phi \rangle = \langle \phi , D_a \phi \rangle.
\]
It follows that $\mathcal{L}$ is gauge invariant and we can regard $\mathcal{L}$ as a function on the quotient space $C(Y)/\mathcal{G}$. The goal of Seiberg--Witten--Floer theory is to construct some sensible notion of Morse homology of $\mathcal{L}$ on $C(Y)/\mathcal{G}$. Consider the formal $L^2$-gradient of $\mathcal{L}$. By this, we mean the function $grad( \mathcal{L}) \colon C(Y) \to i\Omega^1(Y) \oplus \Gamma(S)$ such that
\[
\langle grad( \mathcal{L})(a , \phi ) , (a' , \phi' ) \rangle_{L^2} = \left. \frac{d}{dt} \right|_{t=0} \left( \mathcal{L}( (a,\phi) + t(a',\phi') \right)
\]
for all $(a,\phi) \in C(Y)$ and all $(a',\phi') \in i\Omega^1(Y) \oplus \Gamma(S)$. A short calculation gives
\[
grad( \mathcal{L} ) = ( *da + \tau(\phi , \phi) , D_a \phi ).
\]

A critical point of $\mathcal{L}$ is a point where $grad( \mathcal{L})$ vanishes. So $(a , \phi)$ is a critical point if and only if it satisfies
\[
*da + \tau(\phi , \phi) = 0, \quad D_a \phi = 0.
\]
These are the $3$-dimensional Seiberg--Witten equations.

A trajectory for the downwards gradient flow is a differentiable map $x : \mathbb{R} \to L^2_{k+1}( C(Y) )$ such that
\[
\frac{d}{dt} x(t) = - grad(\mathcal{L})( x(t) ).
\]
If  $x(t) = (a(t) , \phi(t) ) \in L^2_{k+1}( Y , i \, T^*Y \oplus S)$, then
\[
\frac{d}{dt} a(t) = -*da(t) - \tau( \phi(t) , \phi(t) ), \quad \frac{d}{dt} \phi(t) = -D_{a(t)} \phi(t).
\]

A key observation is that such trajectories can be re-interpreted as solutions of the $4$-dimensional Seiberg--Witten equations on the cylinder $X = \mathbb{R} \times Y$. 

\begin{definition}
A Seiberg--Witten trajectory $x(t) = (a(t) , \phi(t))$ is said to be of {\em finite type} if both $\mathcal{L}(x(t))$ and $|| \phi(t) ||_{\mathcal{C}^0}$ are bounded functions of $t$.
\end{definition}

\subsection{Restriction to the global Coulomb slice}

Define the {\em global Coulomb slice} (with respect to $A_0$) to be the subspace
\[
V = Ker(d^*) \oplus \Gamma(S) \subset C(Y).
\]
Given $(a,\phi) \in C(Y)$, there exists an element of $V$ which is gauge equivalent to $(a,\phi)$, namely
\[
(a - df , e^f \phi )
\]
where $d^*(a-df) = 0$, so $\Delta f = d^*(a)$. If we impose the condition $\int_Y f dvol_Y = 0$, then there is a unique solution to these equations given by $f = G d^*a$, where $G$ is the Green's operator for the Laplacian $\Delta = d d^*$ on functions.

We have a globally defined map $\Pi \colon C(Y) \to V$, called the {\em global Coulomb projection}
\[
\Pi(a,\phi) = (a -df , e^f \phi )
\]
where $\Delta f = d^*(a)$ and $\int_Y f dvol = 0$.

Restricting to the global Coulomb slice $V$ uses up all of the gauge symmetry except for the $S^1$ subgroup of constant gauge transformations. Instead of working on $C(Y)$ with full gauge symmetry, we work on $V$ with $S^1$ symmetry.

As $b_1(Y) = 0$, every map $u \colon Y \to S^1$ can be written as $u = e^f$ for some $f \colon Y \to i\mathbb{R}$. Moreover, $f$ is unique up to addition of an integer multiple of $2\pi i$. We define $\mathcal{G}_0$ to be the subgroup of gauge transformations of the form $u = e^f$ for some $f \colon Y \to i\mathbb{R}$ with $\int_Y f dvol = 0$. It is easy to see that $\mathcal{G} = \mathcal{G}_0 \times S^1$.

We have that $\mathcal{G}_0$ acts freely on $C(Y)$ and the quotient space can be identified with $V$. This determines a metric $\tilde{g}$ on $V$ as follows. Take the restriction of the $L^2$-metric on $C(Y)$ to the subbundle of the tangent bundle orthogonal to the gauge orbits. This construction is $\mathcal{G}_0$-invariant and   descends to a metric $\tilde{g}$ on $V$.

The Chern--Simons--Dirac functional $\mathcal{L}$ is gauge invariant, hence the gradient $grad(\mathcal{L})$ is orthogonal to the gauge orbits. It follows that the projection of $grad(\mathcal{L})$ to $V$ coincides with taking the gradient of $\mathcal{L}|_{V}$ with respect to $\tilde{g}$. So the trajectories of $grad(\mathcal{L})$ on $C(Y)$ project to the trajectories of $\mathcal{L}|_V$, where the gradient of $\mathcal{L}|_V$ is taken using the metric $\tilde{g}$. Thus the trajectories on $V$ have the form
\[
\frac{d}{dt} ( a(t) , \phi(t) ) = ( -*da - \tau(\phi , \phi) , -D_a \phi ) - ( -df , f \phi )
\]
for a function $f \colon Y \to i\mathbb{R}$. The function $f$ is uniquely determined by the conditions that $\int_Y f dvol_Y = 0$ and that $*da + \tau(\phi , \phi) - df$ is in the kernel of $d^*$. Hence $df = (1-\pi)\tau(\phi , \phi)$, where $\pi$ denotes the $L^2$ orthogonal projection to $Ker(d^*)$. We have that
\[
\frac{d}{dt} ( a(t) , \phi(t) ) = ( -*da - \pi \tau(\phi , \phi) , -D_a \phi-f\phi ) = -(l+c)(a,\phi),
\]
where
\[
l(a , \phi) = ( *da , D \phi)
\]
is the linear part and
\[
c(a,\phi) = ( \pi \tau(\phi , \phi ) , \rho(a)\phi + f\phi)
\]
is given by the non-linear terms.

Let $\chi$ denote the gradient of $\mathcal{L}|_V$ with respect to $\tilde{g}$. Then $\chi = l + c$ extends to a map
\[
\chi = l + c : V_{k+1} \to V_{k}
\]
where $V_k$ denotes the $L^2_k$-Sobolev completion of $V$. The map $l$ is a linear Fredholm operator. Using Sobolev multiplication and an estimate on the unique solution to $df = (1-\pi)\tau(\phi , \phi)$, $\int_Y f dvol_Y = 0$, it follows that $c$ viewed as a map $V_{k+1} \to V_{k+1}$ is continuous. Hence $c : V_{k+1} \to V_k$ is compact. The flow lines of $\chi$ on $V$ will be called Seiberg--Witten trajectories in the Coulomb gauge. We say that such a trajectory $x(t) = (a(t) , \phi(t))$ is of finite type if $\mathcal{L}(x(t))$ and $|| \phi(t) ||_{\mathcal{C}^0}$ are bounded independent of $t$. Clearly the finite type Seiberg--Witten trajectories in the Coulomb gauge are precisely the projection to $V$ of the Seiberg--Witten trajectories in $C(Y)$ of finite type.

\subsection{Finite dimensional approximation}

Let $V_\lambda^\mu$ denote the direct sum of all eigenspaces of $l$ in the range $(\lambda , \mu]$ and let $\tilde{p}^\mu_\lambda$ be the $L^2$-orthogonal projection from $V$ to $V_\lambda^\mu$. Note that $V_\lambda^\mu$ is a finite dimensional subspace of $V$. For technical reasons we replace the projections $\tilde{p}^\mu_\lambda$ with smoothed out versions
\[
p^\mu_\lambda = \int_0^1 \rho(\theta) \tilde{p}^{\mu-\theta}_{\lambda+\theta} d\theta
\]
where $\rho \colon \mathbb{R} \to \mathbb{R}$ is smooth, non-negative, non-zero precisely on $(0,1)$ and $\int_{\mathbb{R}} \rho(\theta) d\theta  = 1$. This is to make $p^\mu_\lambda$ vary continuously with $\mu$ and $\lambda$. The reason for doing this is  to show that the Conley index is independent of the choices of $\mu, \lambda$, up to a suspension. This is achieved by continuously increasing or decreasing $\mu, \lambda$ to get a continuous family of flows and using homotopy invariance of the Conley index under continuous deformation of the flow.

Consider the gradient flow equation
\[
\frac{d}{dt} x(t) = -(l + p^\mu_\lambda c) x(t)
\]
where $x \colon \mathbb{R} \to V^\mu_\lambda$. We call this an {\em approximate Seiberg--Witten trajectory}.

Let $B(R)$ denote the open ball of radius $R$ in $L^2_{k+1}(V)$. Using the a priori estimates for the Seiberg--Witten equations, it can be shown that there exists an $R > 0$ such that all the finite type trajectories of $l+c$ are in $B(R)$ \cite[Proposition 1]{man}. This boundedness property does not necessarily hold for approximate trajectories, since the crucial estimates that hold for the Seiberg--Witten equations do not apply to the approximate trajectories. However, we have the following result which acts as a kind of substitute:

\begin{proposition}[\cite{man}, Proposition 3]\label{prop:R}
For any $-\lambda, \mu$ sufficiently large, if an approximate trajectory $x \colon \mathbb{R} \to L^2_{k+1}( V^\mu_\lambda)$ satisfies $x(t) \in \overline{B(2R)}$ for all $t$, then in fact $x(t) \in B(R)$ for all $t$.
\end{proposition}

This result will allow us to construct the Seiberg--Witten--Floer homotopy type of $(Y , \mathfrak{s})$ using Conley indices.

\subsection{The Conley index}

Suppose we have a $1$-parameter group $\{ \varphi_t \}$ of diffeomorphisms of an $n$-dimensional manifold $M$ (not necessarily compact). The example to keep in mind is the gradient flow of a Morse function. Given a compact subset $N \subseteq M$, the invariant set of $N$ is
\[
Inv(N,\varphi) = \{ x \in N \; | \; \varphi_t(x) \in N \; \text{for all } t \in \mathbb{R} \}.
\]
A compact subset $N \subseteq M$ is called an {\em isolating neighbourhood} if $Inv(N , \varphi) \subseteq {\rm int} N$. An {\em isolated invariant set} is a subset $S \subseteq M$ such that $S = Inv(N , \varphi)$ for some isolating neighbourhood. Note that $S$ must be compact since it is a closed subset of $N$ and $N$ is required to be compact.

\begin{definition}
Let $S$ be an isolated invariant set. An {\em index pair} $(N,L)$ for $S$ is a pair of compact sets $L \subseteq N \subseteq M$ such that
\begin{itemize}
\item{$Inv(N-L,\varphi) = S \subseteq {\rm int}(N-L)$}
\item{$L$ is an exit set for $N$, that is, for all $x \in N$, if there exists $t > 0$ such that $\varphi_t(x)$ is not in $N$, then there exists $\tau$ with $0 \le \tau < t$ with $\varphi_\tau(x) \in L$.}
\item{$L$ is positively invariant in $N$, that is, if $x \in L$ and $t>0$ and such that $\varphi_s(x) \in N$ for all $0 \le s \le t$, then $\varphi_s(x) \in L$ for all $0 \le s \le t$.}
\end{itemize}
Any isolated invariant set $S$ admits an index pair $(N,L)$. The {\em Conley index} of $S$ is the based homotopy type
\[
I(S) = (N/L , [L]).
\]
\end{definition}

The Conley index is independent of the choice of index pair $(N , L)$ in a strong way. Namely for any two pairs $(N_1,L_1), (N_2,L_2)$, there is a canonical homotopy equivalence $N_1/L_1 \cong N_2/L_2$. The composition of two such canonical homotopy equivalences $N_1/L_1 \cong N_2/L_2$ and $N_2/L_2 \cong N_3/L_3$ coincides up to homotopy with the canonical homotopy equivalence $N_1/L_1 \cong N_3/L_3$ (one says that the collection of Conley indices $N/L$ forms a {\em connected simple system}). By abuse of terminology, if $(N,L)$ is an index pair for $S$ we say that $I = N/L$ is ``the" Conley index of $S$.

\begin{example}
Consider a Morse function with critical point of index $p$, say
\[
f(x_1, \dots , x_n) = \frac{1}{2} \left( -x_1^2 - \cdots -x_p^2 + x_{p+1}^2 + \cdots + x_n^2 \right).
\]
The negative gradient of $f$ using the Euclidean metric is
\[
-grad(f)(x) = (x_1 , \dots , x_p , -x_{p+1} , \dots , -x_{n}).
\]
It follows that the downwards gradient flow is given by
\[
\varphi_t( x ) = ( e^t x_1 , \dots , e^t x_p , e^{-t}x_{p+1} , \dots , e^{-t}x_n).
\]
Let $S = \{ 0 \}$ be the critical point. This is an isolated invariant set. In fact, the only invariant point of $\varphi$ is the origin, so we could take $N = D^p \times D^{n-p}$ as an isolating neighbourhood (where $D^j$ is the closed $j$-dimensional unit disc). Then $L = S^{p-1} \times D^{n-p}$ is an exit set for $N$. It is easy to see that $(N,L)$ satisfies the condition for an index pair for $S$. The Conley index is $I(S) = D^p \times D^{n-p} / ( S^{p-1} \times D^{n-p})$, which is homotopy equivalent to $S^p$, a $p$-dimensional sphere.
\end{example}

\begin{example}
If $M$ is a compact manifold and $\varphi$ is a Morse--Smale gradient flow on $M$, then the set $S$ of all critical points and all flow lines between them is an isolated invariant set. The reduced homology of $I(S)$ is known to be isomorphic to the homology of $M$.

On the other hand, if $M$ is non-compact, then we can not take $S$ to be all critical points of $M$ and all flow lines starting or terminating at a critical point, because there could be flow lines going off to $-\infty$ or coming in from $+\infty$ and then $S$ would not be compact.
\end{example}

We also need the equivariant Conley index. Let $G$ be a compact Lie group acting smoothly on $M$, preserving a flow $\varphi$ and an isolated invariant set $S$. It turns out that one can find a $G$-invariant index pair $(N,L)$ for $S$ and one can define the $G$-equivariant Conley index to be the pointed $G$-equivariant homotopy type
\[
I_G(S) = (N/L , [L]).
\]
It can be shown that this is well-defined, up to $G$-equivariant homotopy equivalence. Moreover $I_G(S)$ has the based homotopy type of a finite $G$-CW complex.

\begin{example}\label{ex:linear}
Consider again the example of the Morse function on $\mathbb{R}^n$ given by
\[
f(x_1, \dots , x_n) = \frac{1}{2} \left( -x_1^2 - \cdots -x_p^2 + x_{p+1}^2 + \cdots + x_n^2 \right).
\]
Now suppose that $G$ is a compact Lie group which acts linearly on $\mathbb{R}^n$ preserving $f$. Note that $f$ defines an $O(n-p,p)$-structure on $\mathbb{R}^n$ and the fact that $G$ preserves $f$ just means that the action of $G$ on $\mathbb{R}^n$ factors through a homomorphism $G \to O(n-p,p)$. As $G$ is compact, we may as well assume (after a linear change of coordinates) that $G$ maps to the maximal compact subgroup $O(n-p) \times O(p)$. So we can decompose $\mathbb{R}^n$ as
\[
\mathbb{R}^n = V_+ \oplus V_-,
\]
where $V_+, V_-$ are real orthogonal representations of $G$ of dimensions $n-p$ and $p$ respectively. Once again, take $S = \{0\}$ as our isolated invariant set. As our Conley index, we can take $N = D(V_-) \times D(V_+)$ and $L = S(V_-) \times D(V_+)$, hence
\[
I_G(S) = D(V_-) \times D(V_+) / S(V_-) \times D(V_+) \cong D(V_-)/S(V_-) \cong (V_-)^+,
\]
where $(V_-)^+$ is the one-point compactification of $V_-$. We see that the action of $G$ on the Conley index is determined by the representation of $G$ on the subspace of the tangent space at the critical point in the direction of the negative eigenvalues of the Hessian of $f$.
\end{example}

Let $G$ and $H$ be compact Lie groups and suppose that $G \times H$ acts smoothly on $M$. Suppose that $\{ \varphi_t \}$ is a $G \times H$-invariant flow. Then $G$ acts smoothly on the submanifold $M^H$ and the restriction of $\{ \varphi_t\}$ defines a $G$-invariant flow on $M^H$. In such a situation we can consider the relation between $G \times H$-equivariant Conley indices for the flow on $M$ and $G$-equivariant Conley indices for the restriction of the flow to $M^H$.

\begin{proposition}\label{prop:restrict}
Let $G \times H$ act smoothly on $M$ preserving a flow $\{ \varphi_t \}$. Let $(N,L)$ be a $G \times H$-equivariant index pair for an isolated invariant set $S = Inv(A)$. Then $(N^H , L^H)$ is a $G$-equivariant index pair for the isolated invariant set $S^H = Inv( A^H)$. Moreover, $(N/L)^H \cong N^H/L^H$.
\end{proposition}
\begin{proof}
First note that $A^H$ is compact because $A$ is compact. Moreover 
\[
Inv( A^H ) = \{ a \in A^H \; | \varphi_t(a) \in A^H \text{ for all } t \} = Inv(A) \cap A^H = S \cap A^H = S^H.
\]
Let $M$ be a topological space and let $P, Q \subseteq M$ be subspaces. Give $Q$ the induced topology. Then
\[
int_M(P) \cap Q = \left( \bigcup_{U \subseteq P \text{ open in } M} U \right) \cap Q = \bigcup_{U \subseteq P \text{ open in } M } (U \cap Q) \subseteq int_Q(P \cap Q).
\]
Applying this to $P = A$, $Q = M^H$, we get
\[
int_M(A) \cap M^H \subseteq int_{M^H}( A^H ).
\]
Then since $S \subseteq int_M(A)$ by the assumption that $A$ is an isolating neighbourhood, it follows that
\[
S^H \subseteq int_M(A) \cap M^H \subseteq int_{M^H}(A^H).
\]
So $A^H$ is an isolating neighbourhood in $M^H$ for $S^H$.

Now let $(N , L)$ be an index pair for $S$. So $N,L$ are compact and $L \subseteq N$. This implies that $N^H, L^H$ are compact and $L^H \subseteq N^H$. Next, since
\[
Inv( N - L ) = S \subseteq int_M(N - L),
\]
it follows that
\begin{align*}
Inv( N^H - L^H) = S^H = S \cap M^H & \subseteq int_M( N - L) \cap M^H \\
& \subseteq int_{M^H}( (N-L) \cap M^H ) = int_{M^H}( N^H - L^H).
\end{align*}
We verify that $L^H$ is an exit set for $N^H$. Let $x \in N^H$ and suppose $\varphi_t(x) \notin N^H$ for some $t > 0$. Then it follows that $\varphi_t(x) \notin N$, for if $\varphi_t(x) \in N$, then it would imply that $\varphi_t(x) \in N \cap M^H = N^H$, since $\varphi_t$ preserves $M^H$. But $L$ is an exit set for $N$, so there exists $\tau \in [0,t)$ with $\varphi_\tau(x) \in L$. It follows that $\varphi_\tau(x) \in L \cap M^H = L^H$. Hence $L^H$ is an exit set for $N^H$.

We check that $L^H$ is positively invariant in $N^H$. Suppose $x \in L^H$ and there exists a $t > 0$ for which $\varphi_s(x) \in N^H$ for all $s \in [0,t]$. Then since $L$ is positively invariant in $N$, it follows that $\varphi_s(x) \in L$ for all $s \in [0,t]$. Hence $\varphi_s(x) \in L \cap M^H = L^H$ for all $s \in [0,t]$.

We have verified that $(N^H , L^H)$ is an index pair for $S^H$. Moreover it is straightforward to check that $(N/L)^H = N^H/L^H$.
\end{proof}

\subsection{Equivariant Spanier--Whitehead category}\label{ss:swc}

In this section we recall the construction of the category $\mathfrak{C}$ from \cite{man}, which is an $S^1$-equivariant version of the Spanier--Whitehead category. In Section \ref{ss:equivswc} we will modify this construction to accommodate a finite group action on $Y$.

We work with pointed topological spaces with a basepoint preserving action of $S^1$. The objects of $\mathfrak{C}$ are triples $(X , m , n)$, where $X$ is a pointed topological space with $S^1$-action, and $m,n \in \mathbb{Z}$.\footnote{In \cite{man} $n$ is allowed to take on rational values. This is needed to construct a Seiberg--Witten--Floer spectrum which does not depend on the choice of metric. For our purposes it suffices to consider only integral values of $n$.} We further require that $X$ has the $S^1$-homotopy type of an $S^1$-CW complex, which holds for Conley indices on manifolds. The set of morphisms between two objects  $(X,m,n)$ to $(X',m',n')$ will be denoted by $\{ (X,m,n) , (X',m',n') \}_{S^1}$ and is defined to be
\[
\{ (X,m,n) , (X',m',n') \}_{S^1} = colim [ (\mathbb{R}^k \oplus \mathbb{C}^l)^+ \wedge X , (\mathbb{R}^{k+m-m'} \oplus \mathbb{C}^{l+n-n'})^+ \wedge X' ]_{S^1},
\]
where $[ \; , \; ]_{S^1}$ denotes the set of $S^1$-equivariant homotopy classes and the colimit is taken over all $k,l$ such that $k \ge m'-m$ and $l \ge n'-n$. The maps that define the colimit are given by suspensions where we smash on the left and for any topological space $Z$, we let $Z^+$ denote the one-point compactification with its obvious basepoint.

Any pointed space $X$ with $S^1$-action defines an object of $\mathfrak{C}$, namely $(X,0,0)$. We often simply write this as $X$. For any finite dimensional representation $E$ of $S^1$, we let $\Sigma^E$ denote the reduced suspension operation
\[
\Sigma^E X = E^+ \wedge X.
\]
This operation extends to $\mathfrak{C}$ by taking $\Sigma^E (X , m , n ) = (\Sigma^E X , m , n)$. We are mainly interested in the case that $E$ is a real vector space with trivial $S^1$-action, or $E$ is a complex vector space with $S^1$ acting by scalar multiplication. If $E$ is a real vector space with trivial action, then one finds that
\[
\Sigma^{E} (X,m,n) \cong (X , m - dim_{\mathbb{R}}(E) , n).
\]
The isomorphism depends on a choice of isomorphism $E \cong \mathbb{R}^{dim_{\mathbb{R}}(E)}$. Up to homotopy there are two choices since $GL(E,\mathbb{R})$ has two components. If $E$ is a complex vector space and $S^1$ acts by scalar multiplication, then
\[
\Sigma^{E}(X,m,n) \cong (X , m , n - dim_{\mathbb{C}}(E) ).
\]
The isomorphism is unique up to homotopy as $GL(E,\mathbb{C})$ is connected. We can define desuspension by a real vector space $E$ with trivial $S^1$-action as follows:
\[
\Sigma^{-E} (X,m,n) = ( (E)^+ \wedge X,m+2dim_{\mathbb{R}}(E) , n).
\]
Then $\Sigma^{-E} \Sigma^{E} Z \cong Z$ by an isomorphism which is canonical up to homotopy. We can define desuspension by a complex vector space $E$ with $S^1$ acting by scalar multiplication by:
\[
\Sigma^{-E} (X,m,n) = (X , m , n + dim_{\mathbb{C}}(E)).
\]
Then $\Sigma^{-E} \Sigma^{E} Z \cong Z$ by an isomorphism which is canonical up to homotopy.

For $Z = (X,m,n) \in \mathfrak{C}$, we define the reduced equivariant cohomology of $Z$ to be 
\[
\widetilde{H}^j_{S^1}(Z ) = \widetilde{H}^{j+m+2n}_{S^1}( X ).
\]
The cohomology is well defined as a consequence of the Thom isomorphism.

\subsection{Seiberg--Witten--Floer cohomology}\label{ss:swfh}

Consider as before a rational homology $3$-sphere $Y$ and a spin$^c$-structure $\mathfrak{s}$. Let $R > 0$, $\lambda, \mu$ be as in Proposition \ref{prop:R}. We want to take the Conley index of the set of all critical points in $B(R)$ and flow lines between them which lie in $B(R)$ for all time for the approximate Seiberg--Witten flow $l + p^\mu_\lambda c$. The problem is that there could be trajectories that go to infinity in a finite amount of time. Hence we do not have a flow $\{ \varphi_t \}$ in the sense of a $1$-parameter group of diffeomorphisms. To get around this issue, let $u^\mu_\lambda$ be a compactly supported smooth cutoff function which is identically $1$ on $B(3R)$. For consistency purposes we assume that $u^\mu_\lambda = u^{\mu'}_{\lambda'}|_{V^\mu_\lambda}$ for $\lambda' \le \lambda$, $\mu' \ge \mu$. One way of doing this is to take $u^\mu_\lambda( v ) = \rho(||v||)$, where $\rho$ is smooth, compactly supported and $\rho( t ) = 1$ for $t < 3$.

For each $\mu, \lambda$, the vector field $u^\mu_\lambda( l + p^\mu_\lambda c)$ is compactly supported, so it generates a well-defined flow $\varphi^\mu_{\lambda,t}$ on $V^\mu_\lambda$. Since $u^\mu_\lambda = 1$ on $\overline{B(2R)}$, Proposition \ref{prop:R} still applies to the trajectories of $u^\mu_\lambda( l + p^\mu_\lambda c)$. It follows that
\[
Inv( V^\mu_\lambda \cap \overline{B(2R)} ) = S^\mu_\lambda
\]
where $S^\mu_\lambda$ is the set of critical points and flow lines between critical points for the approximate Seiberg--Witten flow $l + p^\mu_\lambda c$ which lie in $B(R)$. Therefore $S^\mu_\lambda$ is an isolated invariant set. Moreover $S^1$ preserves the approximate flow, hence we may take the $S^1$-equivariant Conley index
\[
I^\mu_\lambda = I_{S^1}(S^\mu_\lambda).
\]
This is an $S^1$-equivariant homotopy type. However it is not quite an invariant of $(Y , \mathfrak{s})$ because it depends on the choice of metric $g$ as well as the values of $\lambda, \mu, R$. Note that it is independent of the choice of $u^\mu_\lambda$ because of the assumption that $u^\mu_\lambda = 1$ on $B(3R)$. To get a genuine invariant we must understand how $I^\mu_\lambda$ changes as we vary these parameters.

Let $\lambda^*, \mu^*$ satisfy Proposition \ref{prop:R}. Suppose that $\lambda' \le \lambda \le \lambda^*$ and $\mu' \ge \mu \ge \mu^*$. We wish to compare the Conley indices $I^\mu_\lambda, I^{\mu'}_\lambda$ and $I^\mu_{\lambda'}$. In other words, what happens if we increase either $\mu$ or $-\lambda$, staying in the range where $\mu$ and $-\lambda$ are sufficiently large.

We use the following invariance property of the Conley index: suppose we have a family $\{ \varphi_t(s)\}$ of flows depending continuously on $s \in [0,1]$. Suppose that a fixed compact set $A$ is an isolating neighbourhood for all $s \in [0,1]$ and let $S(s) = Inv( A , \varphi_t(s) )$. Then $I( S_0 , \varphi_t(0)) \cong I(S_1 , \varphi_t(1))$ by a canonical homotopy equivalence.

Consider increasing $\mu$ to $\mu'$. The finite energy trajectories of $l + p^\mu_\lambda c p^\mu_\lambda$ in $V^{\mu'}_\lambda$ must actually lie in $V^\mu_\lambda$. Therefore $A = \overline{B}(2R) \cap V^{\mu'}_\lambda$ is an isolating neighbourhood for $S^\mu_\lambda$ in $V^{\mu'}_\lambda$. Let $\mu(s) = (1-s)\mu + s\mu'$, $s \in [0,1]$, and let $\widetilde{\varphi}^{\mu(s)}_\lambda$ denote the flow of $u^{\mu(s)}_\lambda( l + p^{\mu(s)}_\lambda c p^{\mu(s)}_\lambda)$ on $V^{\mu'}_\lambda$. Then for each $s \in [0,1]$, $A$ is an isolating neighbourhood for $S^{\mu(s)}_\lambda$ in $V^{\mu'}_\lambda$ with respect to the flow $\widetilde{\varphi}^{\mu(s)}_\lambda$. Hence
\[
Inv( \widetilde{\varphi}^\mu_\lambda , A) \cong Inv( \widetilde{\varphi}^{\mu'}_\lambda , A) = I^{\mu'}_\lambda.
\]
But $\widetilde{\varphi}^\mu$ is easily seen to be homotopic to the product of the flow of $u^{\mu}_\lambda (l + p^\mu_\lambda c)$ on $V^\mu_\lambda$ with a linear flow on $W$ generated by $l|_W$, where $W$ is the orthogonal complement of $V^\mu_\lambda$ in $V^{\mu'}_{\lambda}$. The Conley index of a product of flows is just the smash product of Conley indices. Combined with Example \ref{ex:linear}, we see that
\[
Inv( \widetilde{\varphi}^\mu_\lambda , A) \cong I^\mu_\lambda \wedge W_-^+,
\]
where $W_-$ is the part of $W$ spanned by negative eigenvalues of $l$. But $W$ is contained in the positive eigenvalues of $l$, so $W_- = 0$ and hence
\[
I^{\mu'}_\lambda \cong I^\mu_\lambda.
\]

Now consider decreasing $\lambda$ to $\lambda'$. An identical argument to the one above gives
\[
I^\mu_{\lambda'} \cong I^\mu_\lambda \wedge W^+_-
\]
where $W$ is the orthogonal complement of $V^\mu_\lambda$ in $V^\mu_{\lambda'}$. In this case $W$ is spanned by negative eigenspaces of $l$, so $W_- = W = V^\lambda_{\lambda'}$ and
\[
I^\mu_{\lambda'} \cong I^\mu_\lambda \wedge (V^\lambda_{\lambda'})^+.
\]
This implies that
\[
\Sigma^{-V^0_\lambda} I^\mu_\lambda
\]
does not depend on the values of $\mu, \lambda$ (provided $\mu, -\lambda$ are sufficiently large).

\begin{definition}
Given $(Y,\mathfrak{s})$ and a metric $g$, we set
\[
SWF(Y , \mathfrak{s} , g) = \Sigma^{-V^0_\lambda(g)} I^\mu_\lambda(g)
\]
for suitably chosen $\mu,\lambda$ and $R$. 
\end{definition}

We have established that the homotopy type of $SWF(Y,\mathfrak{s},g)$ does not depend on the choices of $\mu$ and $\lambda$, or more precisely, any two choices of $\mu,\lambda$ are related by a canonical homotopy equivalence. One also checks that it does not depend on the choice of $R$. So up to homotopy, $SWF(Y , \mathfrak{s} , g)$ depends only on $Y , \mathfrak{s}$ and $g$.

Next we consider varying the metric $g$. Consider a smooth homotopy $g_s$, $s \in [0,1]$ joining two metrics $g_0, g_1$, which is constant near $s=0$. Assuming that the $g_s$ are all sufficiently close to each other in a suitable topology, we can arrange that there exists $R, \mu^*, \lambda^*$ such that Proposition \ref{prop:R} is true for all $s \in [0,1]$ and all $\mu,\lambda$ with $\mu \ge \mu^*$, $\lambda \le \lambda^*$. This suffices, as compactness of $[0,1]$ implies that any smooth path $g_s$ can be broken up into finitely many sub-paths over which this assumption holds.

We assume that there exists some $\lambda < \lambda^*$ and $\mu > \mu^*$ such that $\lambda, \mu$ are not eigenvalues of $l_s$ for any $s \in [0,1]$. This property will hold for all sufficiently small paths. The spaces $(V^\mu_\lambda)_s$ then form a smooth vector bundle over $[0,1]$. We can trivialise this vector bundle and identify all these spaces with a single $V^\mu_\lambda$. Further, we assume that  $B(R)_{s_1} \subset B(2R)_{s_2}$ for each $s_1,s_2 \in [0,1]$. Here we think of the balls as subsets of the same space $V^\mu_\lambda$. Once again, this property will hold for all small enough paths. Then
\[
\cap_{s \in [0,1] } \overline{B(2R)_s}
\]
is a compact isolating neighbourhood for $S^\mu_\lambda$ in any metric $g_s$ with the flow $(\varphi^\mu_\lambda)_s$. The Conley index will be independent of $s$ and hence
\[
(I^\mu_\lambda)_0 \cong (I^\mu_\lambda)_1.
\]
However, we do not have that $\Sigma^{-(V^0_\lambda)_0} (I^\mu_\lambda)_0$ equals $\Sigma^{-(V^0_\lambda)_1} (I^\mu_\lambda)_1$. The reason is that some eigenvalues in $(\lambda , \mu)$ may change sign. On the other hand, any eigenvalue greater than $\mu$ or less than $\lambda$ can not change sign, by our assumption that $\mu, \lambda$ are not eigenvalues of $l_s$ for any $s \in [0,1]$. Hence the difference between $(V^0_\lambda)_0$ and $(V^0_\lambda)_1$ is given in terms of the spectral flow of the family of operators $\{ l_s \}$, $s \in [0,1]$.

The operator $l$ can be split into real and complex components. The real part has no spectral flow, so we only need to consider the complex part, which is the Dirac operator $D_s$. The spectral flow can be expressed using the Atiyah--Patodi--Singer (APS) index theorem on the cylinder $X = [0,1] \times Y$, see \cite{aps1,aps3}. Let $\widehat{g}$ be the metric on $X$ given by $g_s$ in the vertical direction and $(ds)^2$ in the horizontal direction. Let $S_s$ denote the spinor bundle associated to $(\mathfrak{s} , g_s)$. The bundles $S_s$ can all be identified with $S = S_0$, but with varying Clifford multiplication. The spin$^c$ structure $\mathfrak{s}$ lifts to a spin$^c$ structure on $X$. Let $S^\pm$ denote the spinor bundles of this spin$^c$ structure. Then $S^\pm$ can be identified with the pullback of $S$ to $X$. Suppose for each $s$ we have chosen a flat reference connection $A_s$. Since we have identified $S_s$ with $S$ for all $s$, we get an induced identification of $L_s = det(S_s)$ with $L = L_0$. Then $A_s = A_0 + i\alpha_s$ for some closed real $1$-form $\alpha_s$. The path of spin$^c$ connections $\{ A_s\}$ fit together to form a spin$^c$ connection $\widehat{A}$ on the determinant line $L$ pulled back to $X$. Let $\widehat{D}$ be the Dirac operator determined by $\widehat{g}$ and $\widehat{A}$. Then $\widehat{D}(\psi) = \partial_s \psi + D_s \psi$. After a possible reparametrisation we can assume that $(g_s , A_s )$ is constant near the boundary. Applying the APS index theorem to the Dirac operator $\widehat{D}$ on the cylinder $[0,1] \times Y$, one can write the spectral flow $SF(\{D_s \})$ as
\[
SF(\{D_s\}) = \frac{1}{2}( \eta(D_1) - k(D_1) ) - \frac{1}{2}( \eta(D_0) - k(D_0) ) + \int_{[0,1] \times Y} \left( -\frac{1}{24} p_1( \hat{g} ) + \frac{ c_1(\widehat{A})^2 }{8} \right),
\]
where $\eta(D)$ is the eta invariant of $D$, $k(D) = dim_{\mathbb{C}}(Ker(D))$, $p_1(\hat{g})$ is the first Pontryagin form of $\hat{g}$, $c_1(\widehat{A})$ is the Chern form $(i/2\pi) F_{2\widehat{A}}$, where $F_{2\widehat{A}}$ is the curvature of the induced connection $2\widehat{A}$ on $L$. Now since $\alpha_s$ is closed for each $s$, we get $F_{2 \widehat{A}} = ds \wedge 2i \partial_s \alpha_s$ and hence $c_1(\widehat{A})^2 = 0$. So
\begin{equation}\label{equ:sf1}
SF(\{D_s\}) = \frac{1}{2}( \eta(D_1) - k(D_1) ) - \frac{1}{2}( \eta(D_0) - k(D_0) ) - \frac{1}{24} \int_{[0,1] \times Y} p_1( \hat{g} ).
\end{equation}
Let $\eta_{sign}(g_s)$ denote the eta invariant of the signature operator on $Y$ defined by $g_s$. Then from the APS index theorem for the signature operator together with the fact that the signature operator has no spectral flow, we find
\begin{equation}\label{equ:sf2}
\eta_{sign}(g_1) - \eta_{sig}(g_0) = \frac{1}{3} \int_{[0,1] \times Y} p_1( \hat{g} ).
\end{equation}\label{equ:sf2}
Combining Equations (\ref{equ:sf1}) and (\ref{equ:sf2}), we see that
\[
SF( \{ D_s \} ) = \frac{1}{2}( \eta(D_1) - k(D_1) ) - \frac{1}{2}( \eta(D_0) - k(D_0) ) - \frac{1}{8}( \eta_{sign}(g_1) - \eta_{sign}(g_0) ),
\]
and hence
\[
SF( \{D_s \} ) = n(Y , \mathfrak{s} , g_0) - n(Y , \mathfrak{s} , g_1)
\]
where we have defined
\[
n(Y , \mathfrak{s} , g) = \frac{1}{2} \eta(D) - \frac{1}{2} k(D) - \frac{1}{8}\eta_{sign}(g)
\]
We will show that $n(Y , \mathfrak{s} , g)$ is a rational number. Let $(W , \mathfrak{s}_W)$ be a spin$^c$ $4$-manifold bounding $(Y,\mathfrak{s})$. This always exists because $\Omega_3^{{\rm spin}^c} = 0$. Extend the Dirac operator $D$ on $Y$ to a Dirac operator $\widehat{D}$ in the same way as we did for the cylinder $[0,1] \times Y$. The APS index theorem for the Dirac operator and signature operator on $W$ combined give
\[
ind_{APS}( \widehat{D}) = \frac{ c_1(\mathfrak{s}_W)^2 - \sigma(W) }{8} + \frac{\eta_{dir} - k}{2} - \frac{\eta_{sign}}{8}
\]
and thus
\begin{equation}\label{equ:aps}
n(Y , \mathfrak{s} , g) = ind_{APS}(\widehat{D}) - \delta(W , \mathfrak{s})
\end{equation}
where we set
\[
\delta(W , \mathfrak{s}) = \frac{ c_1(\mathfrak{s}_W)^2 - \sigma(W) }{8}.
\]
This shows that $n(Y , \mathfrak{s} , g)$ is a rational number since $ind_{APS}(\widehat{D})$ is an integer and $\delta(W , \mathfrak{s})$ is a rational number.

\begin{definition}
The {\em Seiberg--Witten--Floer cohomology} of $(Y , \mathfrak{s},g)$ is defined as
\[
HSW^j(Y , \mathfrak{s}) = \widetilde{H}^{j+2n(Y , \mathfrak{s} , g)}_{S^1}( SWF(Y , \mathfrak{s} , g) ),
\]
where $j \in \mathbb{Q}$ and as usual the coefficient group $\mathbb{F}$ has been omitted from the notation.
\end{definition}

Below we will show that $HSW^*(Y , \mathfrak{s})$ is independent of the choice of metric $g$ (and other auxiliary choices), hence it is a well defined topological invariant of the pair $(Y , \mathfrak{s})$.

Notice that because of the grading shift by $2n(Y , \mathfrak{s} , g)$ the cohomology groups $HSW^*(Y , \mathfrak{s})$ are concentrated in rational degrees. It was shown by Lidman--Manolescu \cite{lima} that $HSW^*(Y , \mathfrak{s})$ is isomorphic to the Seiberg--Witten monopole Floer cohomology as defined by Kronheimer--Mrowka \cite{km}. Together with the equivalence of monopole Floer homology and Heegaard Floer homology due to the work of Kutluhan--Lee--Taubes \cite{klt1,klt2,klt3,klt4,klt5}, Colin--Ghiggini--Honda \cite{cgh1,cgh2,cgh3} and Taubes \cite{tau}, we have isomorphisms
\[
HSW^*(Y , \mathfrak{s}) \cong HF^-_{-*}( \overline{Y} , \mathfrak{s}) \cong HF_+^*(Y , \mathfrak{s})
\]
where $HF^-_*$ denotes the minus version of Heegaard Floer homology and $HF_+^*$ denotes the plus version of Heegaard Floer {\em cohomology} (taken with respect to the same coefficient group $\mathbb{F}$). Here we use a grading convention for $HF^-$ such that $HF^-(S^3)$ starts in degree $0$. Through the work of \cite{cg,hr,ra}, the isomorphism is known to preserve the absolute gradings. Using co-Borel, Tate or non-equivariant cohomologies gives similar isomorphisms to the other versions of Heegaard--Floer homology, see \cite[Corollary 1.2.4]{lima} for the precise statement.

 If we have two metrics $g_0,g_1$, then the spectral flow of a path joining them satisfies
\[
SF( \{D_s \} ) = n(Y , \mathfrak{s} , g_1) - n(Y , \mathfrak{s} , g_0).
\]
On the other hand, from the definition of spectral flow we   have
\[
SF( \{D_s \} ) = dim(V^0_\lambda(g_0)) - dim(V^0_\lambda(g_1)).
\]
It follows that
\[
SWF(Y , \mathfrak{s} , g_1) \cong \Sigma^{SF(\{D_s \})} SWF(Y , \mathfrak{s} , g_0)
\]
and hence
\[
\widetilde{H}^{j+2SF(\{D_s\})}_{S^1}( SWF(Y , \mathfrak{s} , g_1) ) \cong \widetilde{H}^{j}_{S^1}( SWF(Y , \mathfrak{s} , g_0) )
\]
by the Thom isomorphism. Replacing $j$ by $j + 2n(Y,\mathfrak{s},g_0)$, we have
\[
\widetilde{H}^{j+2 n(Y,\mathfrak{s},g_1)}_{S^1}( SWF(Y , \mathfrak{s} , g_1) ) \cong \widetilde{H}^{j+2n(Y,\mathfrak{s},g_0)}_{S^1}( SWF(Y , \mathfrak{s} , g_0) ).
\]

Hence the Seiberg--Witten--Floer cohomology $\widetilde{H}^{j+2 n(Y,\mathfrak{s},g)}_{S^1}( SWF(Y , \mathfrak{s} , g) )$ is independent of the metric. The above isomorphism is canonical in the sense that it does not depend on the choice of path from $g_0$ to $g_1$. This follows from the fact that the space of all metrics on $Y$ is contractible, so   any two paths with the same endpoints are homotopic.

\subsection{Duality}\label{sec:duality}

\begin{definition}
Let $V$ be a finite dimensional representation of a compact Lie group $G$. Two pointed, finite $G$-CW complexes $X,X'$ are equivariantly {\em $V$-dual} if there exists a $G$-map
\[
\varepsilon \colon  X \wedge X' \to V^+
\]
such that for any subgroup $H \subseteq G$, the fixed point map
\[
\varepsilon^H \colon X^H \wedge (X')^H \to (V^H)^+
\]
induces a non-equivariant duality between $X^H$ and $(X')^H$, in the sense of non-equivariant Spanier--Whitehead duality.
\end{definition}

Consider the Conley index $I^\mu_\lambda$ associated to $(Y , \mathfrak{s} , g)$ for suitably chosen $R , \mu , \lambda$. One finds that reversing orientation of $Y$ has the effect of reversing the Chern--Simons--Dirac flow. From \cite{cor}, it follows that $I^\mu_\lambda(Y)$ and $I_\mu^\lambda( \overline{Y})$ are $V^\mu_\lambda$-dual, so there exists a duality map
\[
\varepsilon \colon I^\mu_\lambda(Y) \wedge I^\lambda_\mu( \overline{Y}) \to (V^\mu_\lambda)^+.
\]
Notice that $dim( V^0_\lambda(Y) ) + dim( V^0_{-\mu}( \overline{Y}) ) - 2k(D) = dim( V^\mu_\lambda(Y))$, where $k(D)$ is the dimension of the kernel of $D$. Desuspending, we obtain a duality map
\[
\varepsilon \colon SWF(Y , \mathfrak{s} , g) \wedge SWF( \overline{Y} , \mathfrak{s} , g) \to S^{-k(D) \mathbb{C}}.
\]
We also have
\begin{equation}\label{equ:nnbar}
n(Y , \mathfrak{s} , g ) + n( \overline{Y} , \mathfrak{s} , g) = -k(D).
\end{equation}

\subsection{Fixed points}

\begin{definition}
Let $s \ge 0$ be an integer. We say that a finite pointed $S^1$-$CW$ complex $X$ is of {\em type SWF at level $s$} if
\begin{itemize}
\item{The $S^1$-fixed point set $X^{S^1}$ is homotopy equivalent to the sphere $(\mathbb{R}^s)^+$.}
\item{The action of $S^1$ is free on the complement $X - X^{S^1}$.}
\end{itemize}
\end{definition}

\begin{proposition}
Given $(Y, \mathfrak{s}, g)$, let $R, \mu, \lambda$ be as in Proposition \ref{prop:R}. Then $\overline{B(2R)} \cap V^\mu_\lambda$ is an isolating neighbourhood for $S^\mu_\lambda = Inv( \overline{B(2R)} \cap V^\mu_\lambda )$. Let $I^\mu_\lambda = I_{S^1}( S^\mu_\lambda)$ be the Conley index. Then $I^\mu_\lambda$ is of type SWF at level $s = dim( V^0_\lambda(\mathbb{R}))$, where $V^0_\lambda(\mathbb{R})$ denotes the $S^1$-invariant part of $V^0_\lambda$.
\end{proposition}
\begin{proof}
Let $(N , L)$ be an index pair for $S^\mu_\lambda$ so that $I^\mu_\lambda = N/L$. Then by Proposition \ref{prop:restrict}, $(I^\mu_\lambda)^{S^1} = N^{S^1}/L^{S^1}$ is the Conley index of $(S^\mu_\lambda)^{S^1}$. Further, we have that $\overline{B(2R)} \cap V^\mu_\lambda(\mathbb{R})$ is an isolating neighbourhood for $(S^\mu_\lambda)^{S^1}$, where $V^\mu_\lambda(\mathbb{R})$ denotes the $S^1$-invariant part of $V^\mu_\lambda$. It is easy to see that $c = 0$ on $V^\mu_\lambda(\mathbb{R})$, where $c$ is the non-linear part of the Seiberg--Witten flow. Thus the restriction of the approximate Seiberg--Witten flow $u^\mu_\lambda( l + p^\mu_\lambda c)$ to $V^\mu_\lambda(\mathbb{R})$ is the flow $u^\mu_\lambda l$. Restricted to $B(3R)$ this is just the linear flow associated to $l$. The real part of $l$ has zero kernel, because $b_1(Y) = 0$. It follows that the Conley index of $(S^\mu_\lambda)^{S^1}$ is the Conley index of $\{0\}$ in $V^\mu_\lambda(\mathbb{R})$ with respect to the linear flow of $l$. This is $( V^0_\lambda(\mathbb{R}) )^+$. Thus we have shown that the $S^1$-fixed point set of $I^\mu_\lambda$ is homotopy equivalent to $( V^0_\lambda(\mathbb{R}) )^+$. Furthermore, $S^1$ acts freely on $V^\mu_\lambda - V^\mu_\lambda(\mathbb{R})$, hence $S^1$ acts freely on $N - N^{S^1}$ and therefore also on $(I^\mu_\lambda) - (I^\mu_\lambda)^{S^1}$. This proves the result.
\end{proof}

Using the identities
\begin{equation}\label{equ:fix}
(\mathbb{R}^+ \wedge X)^{S^1} = \mathbb{R}^+ \wedge X^{S^1}, \quad (\mathbb{C}^+ \wedge X)^{S^1} = X^{S^1}
\end{equation}
we see that:
\begin{itemize}
\item{If $X$ is of type SWF at level $s$, then $\mathbb{R}^+ \wedge X$ is of type SWF at level $s+1$,}
\item{If $X$ is of type SWF at level $s$, then $\mathbb{C}^+ \wedge X$ is of type SWF at level $s$.}
\end{itemize}

Now let $Z = (X , m , n )$ belong to the equivariant Spanier-Whitehead category $\mathfrak{C}$. We say that $Z$ is of type SWF at level $s$ if $X$ is of type SWF of level $s+m$. The above remarks shows that this is a well-defined notion.

We have shown that the Conley index $I^\mu_{\lambda}$ is of type SWF at level $s = dim(V^0_\lambda(\mathbb{R}))$, where $V^0_\lambda = V^0_\lambda(\mathbb{R}) \oplus V^0_\lambda(\mathbb{C})$ denotes the decomposition of $V^0_\lambda$ into copies of $\mathbb{R}$ and $\mathbb{C}$. Now we recall that
\[
SWF(Y , \mathfrak{s} , g) = \Sigma^{-V^0_\lambda} I^\mu_\lambda.
\]
It follows that $SWF(Y , \mathfrak{s} , g)$ is of type SWF at level $0$.

Let $X$ be a space of type SWF at level $s$. Let $\iota \colon X^{S^1} \to X$ denote the inclusion of the fixed point set. Using the localisation theorem in equivariant cohomology \cite[III (3.8)]{die}, it follows that the pullback map $\iota^* \colon \widetilde{H}^*_{S^1}(X) \to \widetilde{H}^*_{S^1}(X^{S^1})$ is not identically zero. Therefore, we may define the $d$-invariant $d(X)$ of $X$ by
\[
d(X) = {\rm min}\{ j \; | \; \exists x \in \widetilde{H}^j_{S^1}( X^{S^1}), x \neq 0, x \in Im(\iota^*) \}.
\]
Note that $d(X)$ could potentially depend on the choice of coefficient group, so we may write the invariant as $d(X ; \mathbb{F})$ if we wish to indicate the dependence on $\mathbb{F}$. 

We also define the $\delta$-invariant of $X$ by $\delta(X) = d(X)/2$. Using Equation (\ref{equ:fix}) and the Thom isomorphism, one finds:
\[
d(\mathbb{R}^+ \wedge X) = d(X)+1, \quad d( \mathbb{C}^+ \wedge X) = d(X)+2.
\]

Now if $Z = (X,m,n)$ is of type SWF, we define the $d$-invariant $d(Z)$ of $Z$ to be
\[
d(Z) = d(X) - m - 2n \in \mathbb{Z}.
\]

From \cite[Corollary 1.2.3]{lima}, it follows that the $d$-invariant $d(Y,\mathfrak{s})$ as defined by Heegaard Floer homology (with coefficient group $\mathbb{F}$) is given in terms of $SWF(Y , \mathfrak{s},g)$ by
\[
d(Y , \mathfrak{s}) = d( SWF(Y , \mathfrak{s},g)) -2n(Y,\mathfrak{s},g).
\]
For notational convenience we also define $\delta(Y , \mathfrak{s}) = d(Y , \mathfrak{s})/2$. 

\section{Equivariant Seiberg--Witten--Floer cohomology}\label{sec:eswfc}

\subsection{Assumption on $G$ and $\mathbb{F}$}\label{sec:assumption}
Throughout this paper we will assume that one of the two following conditions hold:
\begin{itemize}
\item[(1)]{$G$ is an arbitrary finite group and $\mathbb{F} = \mathbb{Z}/2\mathbb{Z}$, or}
\item[(2)]{$\mathbb{F}$ is an arbitrary field and the order of $G$ is odd.}
\end{itemize}

Condition (1) ensures that we do not need to concern ourselves with questions of orientability. Condition (2) ensures that any $S^1$-central extension $\widetilde{G}$ acts orientation preservingly on all of its finite dimensional representations. Hence under either condition, the Thom isomorphism holds without requiring local coefficients:
\[
\widetilde{H}^*_{\widetilde{G}}( X ) \cong \widetilde{H}^{* + dim_{\mathbb{R}}(V)}_{\widetilde{G}}( V^+ \wedge X ).
\]
Here $X$ is any $\widetilde{G}$-space and $V$ is any finite dimensional representation of $\widetilde{G}$.

\subsection{Lifting $G$-actions}\label{sec:lifting}

Recall that $Y$ denotes a rational homology $3$-sphere. Suppose that a finite group $G$ acts on $Y$ by orientation preserving diffeomorphisms and suppose $G$ preserves the isomorphism class of a spin$^c$-structure $\mathfrak{s}$. We will construct a $G$-equivariant version of the Seiberg--Witten--Floer cohomology of $(Y , \mathfrak{s})$.

Choose a $G$-invariant metric $g$ on $Y$ and a reference spin$^c$-connection $A_0$ such that the connection on the determinant line $L$ is flat. Let $g \in G$ and choose a lift $\hat{g} \colon S \to S$ of $g$ to the spinor bundle $S$, which is possible since $G$ preserves the isomorphism class of $\mathfrak{s}$. Then $\hat{g}^{-1}A_0 \hat{g} = A_0 + a$ for some $a \in i \Omega^1(Y)$. Since $A_0$ and $\hat{g}^{-1}A_0 \hat{g}$ are flat, we must have $da = 0$. Moreover,  $b_1(Y)=0$ implies that $a = df$ for some $f \colon Y \to i\mathbb{R}$. Setting $\widetilde{g} = e^{-f}\hat{g}$, it follows that $\widetilde{g}$ is a lift of $g$ which preserves $A_0$. Any other lift of $g$ that preserves $A_0$ is of the form $c\widetilde{g}$ with $c \in U(1)$ a constant. Let $G_{\mathfrak{s}}$ denote the set of all possible lifts of elements of $G$ which preserve $A_0$. Then $G_{\mathfrak{s}}$ is a group and we have a central extension
\[
1 \to S^1 \to G_{\mathfrak{s}} \to G \to 1.
\]

Now we carry out the construction of the Conley index of a finite-dimensional approximation of the Chern-Simons-Dirac flow $G_{\mathfrak{s}}$-equivariantly, instead of just $S^1$-equivariantly. 

\subsection{$G_{\mathfrak{s}}$-equivariant Spanier--Whitehead category}\label{ss:equivswc} 
In this section $\widetilde{G}$ denotes any $S^1$ central extension of $G$. We will construct a category $\mathfrak{C}(\widetilde{G})$, the $\widetilde{G}$-equivariant version of $\mathfrak{C}$.

Recall from Section \ref{ss:swc} that the category $\mathfrak{C}$ was constructed   so that there exists a desuspension functor $\Sigma^{-V}$ for any real vector space $V$ with trivial $S^1$-action or any complex vector space where $S^1$ acts by scalar multiplication. We now construct a category $\mathfrak{C}(\widetilde{G})$ in which we can desuspend by real representations of $\widetilde{G}$, where $S^1$ acts trivially, and by complex representations, where $S^1$ acts by scalar multiplication. We are lead to consider the following two types of finite dimensional representations of $\widetilde{G}$:
\begin{itemize}
\item{Type (1): $V$ is a real representation of $\widetilde{G}$ and $S^1$ acts trivially.}
\item{Type (2): $V$ is a complex representation $\widetilde{G}$ and $S^1$ acts on $V$ by scalar multiplication.}
\end{itemize}

Type (1) representations correspond canonically to real representations of $G$.

Type (2) representations correspond to projective unitary representations of $G$ such that the pullback to $G$ of the central extension $S^1 \to U(n) \to PU(n)$ gives an extension isomorphic to $\widetilde{G}$. If $\widetilde{G}$ is split, then type (2) representations are in bijection with complex representations of $G$. However, the bijection depends on a choice of splitting of $\widetilde{G}$.

To define stable homotopy groups we need to consider suspensions with explicitly chosen representations. In other words, we need to work at the level of representations and not just isomorphism classes. Let $V_1, \dots , V_p$ be a complete set of irreducible representations of type (1), and $W_1, \dots, W_q$ a complete set of irreducible representations of type (2). Any representation of type (1) is isomorphic to a direct sum of copies of $V_1, \dots , V_p$ and likewise any representation of type (2) is a direct sum of copies of $W_1, \dots , W_q$.

If $m = (m_1 , \dots , m_p), m' = (m'_1 , \dots , m'_p) \in \mathbb{Z}^p$, we say $m \ge m'$ if $m_i \ge m'_i$ for each $i$. If $m \in \mathbb{Z}^p$ satisfies $m \ge 0$, then we set
\[
V(m) = V_1^{\oplus m_1} \oplus \cdots \oplus V_p^{\oplus m_p}.
\]
Similarly, if $n = (n_1 , \dots , n_q) \in \mathbb{Z}^q$ satisfies $n \ge 0$, then we set
\[
W(n) = W_1^{\oplus n_1} \oplus \cdots \oplus W_q^{\oplus n_q}.
\]
The category $\mathfrak{C}(\widetilde{G})$ has as objects  triples $(X , m , n )$ where
\begin{itemize}
\item{$X$ is a pointed topological space with a basepoint preserving $\widetilde{G}$-action and the homotopy type of a $\widetilde{G}$-CW complex.}
\item{$m \in \mathbb{Z}^p$}
\item{$n \in \mathbb{Z}^q$}.
\end{itemize}

Let $(X,m,n), (X',m',n')$ be two objects of $\mathfrak{C}(\widetilde{G})$. The set of morphisms from $(X,m,n)$ to $(X',m',n')$, denoted $\{ (X,m,n) , (X' , m' , n') \}^{\widetilde{G}}$, is defined to be:
\[
\underset{k,l}{{\rm colim}} \left[ ( V(k) )^+ \wedge (W(l))^+ \wedge X , (V(k+m-m'))^+ \wedge (W(l+n-n'))^+ \wedge X' \right]^{\widetilde{G}}.
\]
The colimit is taken over all $k \in \mathbb{N}^p$, $l \in \mathbb{N}^q$ such that $k \ge m'-m$ and $l \ge n'-n$. The maps that define the colimit are given by suspensions where we smash on the left. 

Let $Y$ be any pointed $\widetilde{G}$-space. We obtain a functor $Y \wedge : \mathfrak{C}(\widetilde{G}) \to \mathfrak{C}(\widetilde{G})$ which is defined on objects by $Y \wedge (X , m , n) = (Y \wedge X , m , n)$ and on morphisms in the evident way. In particular, if $V$ is any finite dimensional representation of $\widetilde{G}$, we define the reduced suspension
\[
\Sigma^{V} Z = V^+ \wedge Z.
\]

We define desuspension by a representation $V$ of type (1) as follows:
\[
\Sigma^{-V} (X,m,n) = ( (V)^+ \wedge X,m+2[V] , n),
\]
where $[V] = (v_1 , \dots , v_p)$ and $v_i$ is the multiplicity of $V_i$ in $V$. Then $\Sigma^{-V} \Sigma^{V} Z \cong Z$, where the isomorphism  is canonical up to homotopy. For any representation $W$ of type (2) we define
\[
\Sigma^{-W} (X,m,n) = (X , m , n + [W])
\]
where $[W] = (w_1, \dots , w_q)$ and $w_i$ is the multiplicity of $W_i$ in $W$. We have that $\Sigma^{-W} \Sigma^{W} Z \cong Z$ by an isomorphism which is canonical up to homotopy. In fact, such an isomorphism is induced by a choice of isomorphism $W \cong W([W])$. But for any pair of isomorphic {\em complex} representations, the space of isomorphism is connected (by Schur's lemma it is a torsor for a product of complex general linear groups). Therefore the isomorphism $W \cong W([W])$ is unique up to homotopy.

For $Z = (X,m,n) \in \mathfrak{C}(\widetilde{G})$, we define the reduced equivariant cohomology of $Z$ to be \[
\widetilde{H}^j_{\widetilde{G}}(Z ) = \widetilde{H}^{j+|m|+2|n|}_{\widetilde{G}}( X ),
\]
where $|m|$ and $|n|$ are defined as follows:
\[
| m | = \sum_{i=1}^p m_i \, {\rm dim}_{\mathbb{R}}(V_i) \text{ for } m = (m_1, \dots , m_p),
\]
\[
| n | = \sum_{i=1}^q n_i \, {\rm dim}_{\mathbb{C}}(W_i) \text{ for } n = (n_1 , \dots , n_q).
\]

The cohomology is well defined as a consequence of the Thom isomorphism. 

\subsection{$G$-equivariant Seiberg--Witten--Floer cohomology}\label{sec:gswfc}

Let $Y$ be a rational homology $3$-sphere and $G$ a finite group acting on $Y$ preserving the isomorphism class of a spin$^c$-structure $\mathfrak{s}$. Let $G_{\mathfrak{s}}$ be the $S^1$-central extension of $G$ obtained by lifting $G$ to the spinor bundle corresponding to $\mathfrak{s}$. We repeat the construction of the Conley index $I^\mu_\lambda(g)$ from Section \ref{ss:swfh}, except that now we carry out the construction $G_{\mathfrak{s}}$-equivariantly. Restricting to the subgroup $S^1 \subseteq G_{\mathfrak{s}}$, $I^\mu_\lambda(g)$ agrees with the $S^1$-equivariant Conley index as previously constructed.

We need to understand how $I^\mu_\lambda(g)$ depends on $\mu,\lambda$, the choice of $G$-invariant metric $g$ and the constant $R$. As in the $S^1$ case, first consider variations of $\mu,\lambda$. Carrying out a similar argument but $G_{\mathfrak{s}}$-equivariantly, we see that $I^\mu_\lambda(g)$ simply changes by suspension. Analogous to the non-equivariant case we define
\[
SWF(Y , \mathfrak{s} , g) = \Sigma^{-V^0_\lambda(g)} I^\mu_\lambda(g) \in \mathfrak{C}(G_{\mathfrak{s}}),
\]
where $V^0_\lambda(g)$ is defined as before, but now carries a $G_{\mathfrak{s}}$-action. Note that $V^0_\lambda(g)$ is the sum of a representation of type (1) and a representation of type (2), so the desuspension $\Sigma^{-V^0_\lambda(g)}$ is defined. Then up to canonical isomorphisms $SWF(Y , \mathfrak{s} , g)$ depends only on the triple $(Y,\mathfrak{s},g)$.

We consider the dependence of $SWF(Y,\mathfrak{s},g)$ on the metric $g$. The argument is much the same as before except done $G_{\mathfrak{s}}$-equivariantly. Let $g_0,g_1$ be two $G$-invariant metrics. The space of such metrics is contractible, so we may choose a path $\{ g_s \}$ from $g_0$ to $g_1$. Then as in the non-equivariant case, the signature operator has no spectral flow and   we have
\[
SWF(Y , \mathfrak{s} , g_1) = \Sigma^{SF_{G_{\mathfrak{s}}}(\{D_s \})} SWF(Y , \mathfrak{s} , g_0),
\]
where now $SF_{G_\mathfrak{s}}( \{ D_s \} )$ is the equivariant spectral flow of $\{ D_s \}$. Thus $SF_{G_\mathfrak{s}}(\{D_s\})$ is to be understood as a virtual representation of $G_{\mathfrak{s}}$ \cite[\textsection 2]{fang}. Since the $S^1$ subgroup of $G_{\mathfrak{s}}$ acts by scalar multiplication on spinors, it follows that $SF_{G_\mathfrak{s}}(\{D_s\})$ is a type (2) virtual representation. From the Thom isomorphism and the fact that the underlying rank of $SF_{G_{\mathfrak{s}}}(\{D_s\})$ is $SF( \{ D_s \} ) = n(Y , \mathfrak{s} , g_1) - n(Y , \mathfrak{s} , g_0)$, we obtain a canonical isomorphism
\[
\widetilde{H}^{j+2 n(Y,\mathfrak{s},g_1)}_{G_{\mathfrak{s}}}( SWF(Y , \mathfrak{s} , g_1) ) \cong \widetilde{H}^{j+2n(Y,\mathfrak{s},g_0)}_{G_{\mathfrak{s}}}( SWF(Y , \mathfrak{s} , g_0) ).
\]

This motivates the following definition:

\begin{definition}
The {\em $G$-equivariant Seiberg--Witten--Floer cohomology} of $(Y , \mathfrak{s},g)$ is defined as
\[
HSW^j_G(Y , \mathfrak{s}) = \widetilde{H}^{j+2n(Y , \mathfrak{s} , g)}_{G_{\mathfrak{s}}}( SWF(Y , \mathfrak{s} , g) ).
\]
\end{definition}

By the argument above, the $HSW^*_G(Y , \mathfrak{s})$ depends only on $(Y,\mathfrak{s})$ and the $G$-action.

For a group $K$ we write $H^*_K$ for $H^*_K(pt)$. Since $HSW^*(Y , \mathfrak{s})$ is defined using equivariant cohomology, it is a graded module over the ring $H^*_{S^1} = \mathbb{F}[U]$, where $deg(U) = 2$. Similarly $HSW^*_G(Y , \mathfrak{s})$ is a graded module over $H^*_{G_{\mathfrak{s}}}$. Restricting from $G_\mathfrak{s}$ to $S^1$, we obtain forgetful maps
\[
HSW^*_G(Y , \mathfrak{s}) \to HSW^*(Y , \mathfrak{s}), \quad H^*_{G_\mathfrak{s}} \to H^*_{S^1}
\]
compatible with the module structures.

Observe that since $S^1$ is the identity component of $G_\mathfrak{s}$, the action of $G_\mathfrak{s}$ on $HSW^*(Y , \mathfrak{s})$ descends to an action of $G$. So we may regard $HSW^*(Y , \mathfrak{s})$ as a $G$-module.

\begin{theorem}\label{thm:ss}
There is a spectral sequence $E_r^{p,q}$ abutting to $HSW^*_G(Y , \mathfrak{s})$ whose second page is given by
\[
E_2^{p,q} = H^p( BG ; HSW^q(Y , \mathfrak{s}) ).
\]

\end{theorem}
\begin{proof}
For a $G_\mathfrak{s}$-space $M$, let $M_{G_\mathfrak{s}}$ denote the Borel model for the $G_\mathfrak{s}$-action and $M_{S^1}$ the Borel model for the $S^1$-action obtained by restriction. The composition $M_{G_\mathfrak{s}} \to BG_\mathfrak{s} \to BG$ is a fibration with fibre $M_{S^1}$. Applying the Leray--Serre spectral sequence, we get a spectral sequence which abuts to $\widetilde{H}^*_{G_\mathfrak{s}}( M )$ and has $E_2^{p,q} = H^p( BG ; \widetilde{H}^q_{S^1}( M ) )$. More generally if $M$ is the formal desuspension of a $G_\mathfrak{s}$-space, then via an application of the Thom isomorphism a similar spectral sequence exists. Applying this to $HSW_G^*(Y , \mathfrak{s})$ gives the theorem.
\end{proof}

\begin{definition}
Let $Y$ be a rational homology $3$-sphere and $\mathfrak{s}$ a spin$^c$-structure. We say that $Y$ is an {\em $L$-space} (with respect to $\mathfrak{s}$ and $\mathbb{F}$) if the action of $U$ on $HSW^*( Y , \mathfrak{s})$ is injective. Equivalently $HSW^*(Y , \mathfrak{s})$ is a free $\mathbb{F}[U]$-module of rank $1$.
\end{definition}

\begin{remark}
The usual definition of an $L$-space is that $HF^+_{red}(Y , \mathfrak{s}) = 0$ for all spin$^c$-structures and where the coefficient group is $\mathbb{Z}$. From the universal coefficient theorem it follows that an $L$-space in this sense is an $L$-space with respect to any spin$^c$-structure $\mathfrak{s}$ and any coefficient group $\mathbb{F}$.
\end{remark}

Suppose that the extension $G_\mathfrak{s}$ is split. A choice of splitting induces an isomorphism $G_\mathfrak{s} \cong S^1 \times G$ and an isomorphism $H^*_{G_{\mathfrak{s}}} \cong H^*_G[U]$. We stress that these isomorphisms depend on the choice of splitting.

\begin{theorem}\label{thm:degen}
Suppose that $G_\mathfrak{s}$ is a split extension. If $Y$ is an $L$-space (with respect to $\mathfrak{s}$ and $\mathbb{F}$), then the spectral sequence given in Theorem \ref{thm:ss} degenerates at $E_2$. Moreover we have
\[
HSW^*_G(Y , \mathfrak{s} ) \cong HSW^*(Y , \mathfrak{s}) \otimes_{\mathbb{F}} H^*_G \cong H^*_G[U] \theta,
\]
where $\theta$ has degree $d(Y,\mathfrak{s})$.
\end{theorem}
\begin{proof}
If $Y$ is an $L$-space (with respect to $\mathfrak{s}$ and $\mathbb{F}$) then
\[
HSW^*(Y , \mathfrak{s} ) \cong \mathbb{F}[U]\theta,
\]
where $\theta$ has degree $d(Y , \mathfrak{s})$. We claim that $G$ acts trivially on $HSW^*(Y , \mathfrak{s})$. This can be seen as follows. First, since $HSW^*( Y , \mathfrak{s})$ is up to a degree shift the $S^1$-equivariant cohomology of the Conley index $I = I^\mu_\lambda$, it suffices to prove the result for $I$. Let $\iota \colon I^{S^1} \to I$ be the inclusion of the $S^1$ fixed point set. Since $Y$ is an $L$-space, $U$ acts injectively on $HSW^*(Y , \mathfrak{s})$. Together with the localisation theorem in equivariant cohomology this implies that $\iota^*$ is injective. Hence it suffices to show that $G$ acts trivially on $\widetilde{H}^*_{S^1}( I^{S^1} )$. But $I^{S^1}$ has the homotopy type of a sphere, so if $\nu$ is a generator of $\widetilde{H}^*_{S^1}( I^{S^1})$ and $g \in G$, then $g^*(\nu) = \pm \nu$ according to whether or not $g$ acts orientation preservingly. Our assumptions on $G$ and $\mathbb{F}$ (see Section \ref{sec:assumption}) ensures that $g^*(\nu) = \nu$ for all $g \in G$. This proves the claim.

Letting $E_r^{p,q}$ denote the spectral sequence for $HSW^*_G(Y , \mathfrak{s})$, it follows easily that
\[
E_2^{p,q} \cong H^*(BG ; \mathbb{F}[U] \theta ) \cong H^*_G[U] \theta \cong HSW^*(Y , \mathfrak{s}) \otimes_{\mathbb{F}} H^*_G.
\]
It remains to show that the differentials $d_2,d_3, \dots $ are all zero. In fact since $\theta$ has the lowest $q$-degree of any term in $E_2^{p,q}$, it follows that $d_j( \theta ) = 0$ for all $j \ge 2$. Then since the differentials commute with the $H^*_{G_\mathfrak{s}} \cong H^*_G[U]$-module structure, it follows that $d_2,d_3, \dots $ all vanish.
\end{proof}

\subsection{Spaces of type $G$-SWF}\label{sec:gswf}

We introduce a $G$-equivariant analogue of spaces of type SWF. We then define a $G$-equivariant analogue of the $d$-invariant.

Let $\widetilde{G}$ be an extension of $G$ by $S^1$. If $\widetilde{G}$ acts on a space $X$, then we get an induced action of $G = \widetilde{G}/S^1$ on the fixed point set $X^{S^1}$. We write $\overline{G} = S^1 \times G$ for the trivial extension of $G$. 

\begin{definition}
Let $s \ge 0$ be an integer. We say that a finite pointed $\widetilde{G}$-$CW$ complex $X$ is of {\em type $G$-SWF at level $s$} if
\begin{itemize}
\item{The $S^1$-fixed point set $X^{S^1}$ is $G$-homotopy equivalent to a sphere $(V)^+$, where $V$ is a real representation of $G$ of dimension $s$.}
\item{The action of $S^1$ is free on the complement $X - X^{S^1}$.}
\end{itemize}
More generally, let $V$ be a finite dimensional representation which is the direct sum of representations of type (1) and (2). An equivariant spectrum $Z = \Sigma^{-V} X \in \mathfrak{C}(\widetilde{G})$ is said to be of {\em type $G$-SWF at level $s$} if $X$ is $G$-SWF at level $s + dim( V^{S^1})$.
\end{definition}

Assume that $\widetilde{G}$ is split and choose a splitting $\widetilde{G} \cong \overline{G}$. Let $X$ be a space of type $G$-SWF at level $s$. Let $\iota \colon X^{S^1} \to X$ denote the inclusion of the fixed point set. Recall that $H^*_{S^1} \cong \mathbb{F}[U]$, where $deg(U) = 2$. Similarly $H^*_{\overline{G}} \cong H^*_G[U]$. The localisation theorem in equivariant cohomology implies that 
\[
\iota^* \colon U^{-1} \widetilde{H}^*_{\overline{G}}(X) \to U^{-1} \widetilde{H}^*_{\overline{G}}( X^{S^1} )
\]
is an isomorphism. Note $X^{S^1} \cong (V)^+$ where $V$ is $s$-dimensional, so
\[
\widetilde{H}^*_{\overline{G}}( X^{S^1} ) \cong H^*_G[U] \tau
\]
where $deg(\tau) = s$. Therefore it also follows that
\[
U^{-1}\widetilde{H}^*_{\overline{G}}( X^{S^1} ) \cong H^*_G[U,U^{-1}] \tau.
\]
Then for each $c \in H^*_G$, it follows that there exists an $x \in \widetilde{H}^*_{\overline{G}}( X)$ for which $\iota^*(x) = c\, U^k \, \tau$, for some $k \ge 0$. Set $\Lambda_G(X) = \widetilde{H}^*_{\overline{G}}(X^{S^1})$. Then $\Lambda_G(X)$ is a free $H^*_G[U]$-module of rank $1$ and $\iota \colon X^{S^1} \to X$ induces a map 
\[
\iota^*  \colon \widetilde{H}^*_{\overline{G}}( X ) \to \Lambda_G(X)
\]
of $H^*_G[U]$-modules. Introduce a filtration
\[
\Lambda_G(X) = F_0 \supseteq F_1 \supseteq F_2 \supseteq \cdots
\]
on $\Lambda_G(X)$ by setting 
\[
F_j = H^{* \ge j}_G \Lambda_G(X),
\]
where $H^{* \ge j}_G = \bigoplus_{k \ge j} H^k_G$. This is the filtration induced by the fibration 
\[
X^{S^1} \times_{\overline{G}} B\overline{G} \to BG.
\]
Let $\tau$ denote the generator of $\Lambda_G(X)$. Then for $j \ge 0$ we have obvious identifications
\[
F_j/F_{j+1} \cong H^j_G[U] \tau.
\]
Now let $c$ be a non-zero element in $H^*_G$ of degree $|c| = deg(c)$. By the discussion above we know that $c \, U^k\tau$ is in the image of $\iota^*$ for some $k \ge 0$. Hence we may define:

\begin{definition}\label{def:dinv}
Let $c$ be a non-zero element in $H^*_G$ of degree $|c| = deg(c)$. We define $d_{G,c}(X) \in \mathbb{Z}$ by
\[
d_{G,c}(X) = {\rm min}\{ 2k+s | \; \exists x \in \widetilde{H}^{s+2k+|c|}_{\overline{G}}(X), \; \iota^*(x) \in F_{|c|}, \; \iota^*(x) = c \, U^k \, \tau \; ({\rm mod} \; F_{|c|+1} ) \}.
\]
For convenience we set $d_{G,0}(X) = -\infty$. Then if $c$ is an element of $H^*_G$, we write $c = c_0 + c_1 + \cdots + c_r$, where $c_i \in H^i_G$ and set
\[
d_{G,c}(X) = {\rm max}\{ d_{G,c_0}(X) , \dots , d_{G,c_r}(X) \}.
\]
\end{definition}

Note that $d_{G,ac}(X) = d_{G,c}(X)$ for any $a \in \mathbb{F}^*$.

In concrete terms, the condition that $\iota^*(x) \in F_{|c|}$ and $\iota^*(x) = c \, U^k \, \tau \; ({\rm mod} \; F_{|c|+1} )$ means that $\iota^*(x)$ is of the form
\[
\iota^*(x) = c \, U^k \, \tau + c_1 \, U^{k-1} \, \tau + \cdots + c_r \, U^{k-r} \, \tau
\]
for some $r \ge 0$ and some $c_1, \dots c_r \in H^{*\ge (|c|+1)}_G$.

\begin{remark}\label{rem:ext}
Let $X$ be a space of type $G$-SWF. The definition of $d_{G,c}(X)$ does not depend on a choice of splitting of $S^1 \to \widetilde{G} \to G$. Indeed, two splittings differ by a homomorphism $\phi \colon G \to S^1$. Let $\alpha = \phi^*(U) \in H^2_G$. The change of splitting acts on $H^*_G[U]$ by sending $U$ to $U + \alpha$. Then since $(U+\alpha)^k = U^k + \cdots $, where $\cdots$ denotes terms involving lower powers of $U$, it follows that $d_{G,c}(X)$ does not depend on the choice of splitting of $\widetilde{G}$. 
\end{remark}

\begin{proposition}\label{prop:din}
Let $X$ be a space of type $G$-SWF for the trivial extension. Then for all $c_1, c_2 \in H^*_G$, we have
\[
d_{G,c_1+c_2}(X) \le {\rm max}\{ d_{G,c_1}(X) , d_{G,c_2}(X) \}
\]
and
\[
d_{G,c_1c_2}(X) \le {\rm min}\{ d_{G,c_1}(X) , d_{G,c_2}(X) \}.
\]
\end{proposition}
\begin{proof}
Let $s$ be the level of $X$. First consider the case that $c_1,c_2$ are homogeneous, that is, $c_1 \in H^{|c_1|}_G, c_2 \in H^{|c_2|}_G$ for some $|c_1|, |c_2|$. Then by Definition \ref{def:dinv}, we have that there exist $x_1 \in \widetilde{H}^{d_{G,c_1}(X) + |c_1|}_{\overline{G}}(X)$ and $x_2 \in \widetilde{H}^{d_{G,c_2}(X) + |c_2|}_{\overline{G}}(X)$ such that
\[
\iota^*(x_1) = c_1 \, U^{k_1} \, \tau + \cdots, \quad \iota^*(x_2) = c_2 \, U^{k_2} \, \tau + \cdots
\]
where $+ \cdots $ denotes terms that are in the next stage of the filtration and $k_i = (d_{G,c_i}(X)-s)/2$ for $i=1,2$. Note that if $c_1$ or $c_2$ are zero then we take $x_1$ or $x_2$ to be zero.

If $|c_1| \neq |c_2|$, then by Definition \ref{def:dinv}, we have $d_{G,c_1+c_2}(X) = {\rm max}\{ d_{G,c_1}(X) , d_{G,c_2}(X) \}$. Now suppose that $|c_1| = |c_2|$. Let $k = {\rm max}\{k_1,k_2\}$ and set $x = U^{k-k_1}x_1 + U^{k-k_2}x_2 \in \widetilde{H}^{2k + s + |c_1|}_{\overline{G}}(X)$. Then
\[
\iota^*(x) = (c_1+c_2) \, U^k \, \tau + \cdots
\]
and hence, from the definition of $d_{G,c_1+c_2}(X)$, we have
\[
d_{G,c_1+c_2}(X) \le 2k + s = {\rm max}\{ 2k_1 + s , 2k_2+s\} = {\rm max}\{ d_{G,c_1}(X) , d_{G,c_2}(X) \}.
\]

Next we observe that $c_2 x_1 \in \widetilde{H}^{d_{G,c_1}(X) + |c_1|+|c_2|}_{\overline{G}}(X)$ and
\[
\iota^*( c_2 x_1 ) = (c_1c_2) \, U^{k_1} \, \tau
\]
and so it follows that $d_{G,c_1c_2}(X) \le d_{G,c_1}(X)$. Exchanging the roles of $x_1,x_2$ and $c_1,c_2$, we similarly find that $d_{G,c_1c_2}(X) \le d_{G,c_1}(X)$, hence
\[
d_{G,c_1c_2}(X) \le {\rm min}\{ d_{G,c_1}(X) , d_{G,c_2}(X) \}.
\]

Now suppose that $c_1,c_2$ are not necessarily homogeneous. We may write $c_1 = a_0 + a_1 + \cdots + a_r$, $c_2 = b_0 + b_1 + \cdots + b_r$, for some $r \ge 0$, where $a_i,b_i \in H^i_G$. By Definition \ref{def:dinv} we have
\[
d_{G,c_1}(X) = {\rm max}_i \{ d_{G,a_i}(X) \}, \quad d_{G , c_2}(X) = {\rm max}_i \{ d_{G,b_i}(X) \}.
\]
Then since $c_1+c_2 = (a_0+b_0) + (a_1+b_1) + \cdots + (a_r+b_r)$, we get
\begin{align*}
d_{G,c_1+c_2}(X) &= {\rm max}_i \{ d_{G, a_i+b_i}(X) \} \\
& \le {\rm max}_i \{ {\rm max}\{ d_{G,a_i}(X) , d_{G,b_i}(X) \} \\
&= {\rm max} \{ {\rm max}_i \{ d_{G,a_i}(X) \} , {\rm max}_i \{ d_{G,b_i}(X) \} \} \\
&= {\rm max} \{ d_{G,c_1}(X) , d_{G,c_2}(X) \}.
\end{align*}
Next, we have $c_1 c_2 = \sum_{i,j} a_i b_j$ and hence
\begin{align*}
d_{G , c_1c_2}(X) &\le {\rm max}_{i,j} \{ d_{G,a_i b_j}(X) \} \\
&\le {\rm max}_{i,j} \{ d_{G , a_i}(X) \} \\
&= {\rm max}_i \{ d_{G , a_i}(X) \} \\
&= d_{G , c_1}(X)
\end{align*}
where we used $d_{G,a_ib_j}(X) \le d_{G , a_i}(X)$. Similarly we get $d_{G , c_1c_2}(X) \le d_{G , c_2}(X)$ and hence 
\[
d_{G , c_1c_2}(X) \le {\rm min}\{ d_{G,c_1}(X) , d_{G,c_2}(X) \}.
\]
\end{proof}

Recall that the ordinary (non-equivariant) $d$-invariant of $X$, $d(X)$, is defined by
\[
d(X) = {\rm min}\{ j \; | \exists x \in \widetilde{H}^{j}_{S^1}(X) \, \iota^*(x) \neq 0 \}.
\]
It is not hard to see that $d(X) = d_{ \{e\},1}(X)$, where $\{e\}$ denotes the trivial group and $1$ is the generator of $H^0(pt)$.

\begin{proposition}\label{prop:Gand1}
Let $X$ be a space of type $G$-SWF for the trivial extension. Then
\[
d_{G,1}(X) \ge d(X).
\]
\end{proposition}
\begin{proof}
By the definition of $d_{G,1}(X)$, there exists $x \in \widetilde{H}^{d_{G,1}(X)}_{\overline{G}}(X)$ such that $\iota^*(x) = U^k \, \tau + \cdots$, where $k = (d_{G,1}(X)-s)/2$ and $s$ is the level of $X$. Let $y \in \widetilde{H}^{d_{G,1}(X)}_{S^1}(X)$ be the image of $x$ under the map induced by $S^1 \to \overline{G}$. Then it follows that $\iota^*(y) = U^k \, \tau \in \widetilde{H}^{d_{G,1}(X)}_{S^1}( X^{S^1})$. In particular, $\iota^*(y) \neq 0$ and hence $d_{G,1}(X) \ge d(X)$ by the definition of $d(X)$.
\end{proof}

Let $S^1$ act trivially on $\mathbb{R}$ and act by scalar multiplication on $\mathbb{C}$. Let $V$ be a real representation of $G$. Then $V_{\mathbb{R}} = \mathbb{R} \otimes_{\mathbb{R}} V$ and $V_{\mathbb{C}} = \mathbb{C} \otimes_\mathbb{R} V$ may be regarded as representations of $\overline{G} = S^1 \times G$, where $S^1$ acts on the first factor and $G$ on the second. 

\begin{proposition}\label{prop:susdinv}
Let $X$ be a space of type $G$-SWF for the trivial extension and let $V$ be a finite dimensional representation of $G$ of type (1) or (2), as in Section \ref{ss:equivswc}. Then for any $c \in H^*_G$, we have
\[
d_{G,c}( V^+ \wedge X ) = d_{G,c}(X) + dim_{\mathbb{R}}(V).
\]
\end{proposition}
\begin{proof}
This result follows easily from the Thom isomorphism, together with the fact that in the type (2) case, the $\overline{G}$-equivariant Euler class of $V$ has the form
\[
e_{\overline{G}}(V) = U^{dim(V)} + c_{G,1}(V) U^{dim(V)-1} + \cdots + c_{G,dim(V)}(V),
\]
where $c_{G,j}(V) \in H^{2j}_G$ denotes the $j$-th $G$-equivariant Chern class of $V$.
\end{proof}

If $Z = \Sigma^{-V} X \in \mathfrak{C}(\overline{G})$ is an equivariant spectrum of type $G$-SWF, we define the $d_{G,c}$-invariant $d_{G,c}(Z)$ of $Z$ to be
\[
d_{G,c}(Z) = d_{G,c}(X) - dim_{\mathbb{R}}(V).
\]
This definition is well-defined by Proposition \ref{prop:susdinv}. We also define a corresponding $\delta$-invariant by setting $\delta_{G,c}(Z) = d_{G,c}(Z)/2$.

\begin{definition}\label{def:deg}
Let $X,Y$ be spaces of type $G$-SWF for the trivial extension of $G$, where $X$ has level $s$ and $Y$ has level $t$. Let $f \colon X \to Y$ be an $S^1 \times G$-equivariant map. Consider the restriction 
\[
{f}^{S^1} \colon X^{S^1} \to Y^{S^1}
\]
of $f$ to the fixed point set. Note that $\widetilde{H}^*_G( X^{S^1} )$ is a free $H^*_G$-module starting in degree $s$. Let $\tau_{X^{S^1}}$ denote a generator. Then $\tau_{X^{S^1}}$ is unique up to an element of $\mathbb{F}^*$. Similarly $\widetilde{H}^*_G( Y^{S^1} )$ is a free $H^*_G$-module starting in degree $t$ and we let $\tau_{Y^{S^1}}$ denote a generator. Then there exists a uniquely determined $\mu \in H^{t-s}_G$ such that 
\[
{(f^{S^1})}^*( \tau_{Y^{S^1}} ) = \mu \, \tau_{X^{S^1}}.
\]
We call $\mu = deg(f^{S^1})$ the {\em degree} of ${f^{S^1}}$. If we choose different generators for $\widetilde{H}^*_G( X^{S^1})$ or $\widetilde{H}^*_G( Y^{S^1})$, then $deg( f^{S^1})$ changes by an element of $\mathbb{F}^*$, hence $deg( f^{S^1})$ is well-defined up to multiplication by elements of $\mathbb{F}^*$. If $t < s$, then $deg( f^{S^1}) = 0$.
\end{definition}

Note that suspension does not change the degree of $f^{S^1}$. Hence we can more generally speak of the degree of $f^{S^1}$ when $f$ is a stable map between spectra of type $G$-SWF.

\begin{proposition}\label{prop:ineq}
Let $f \colon X \to Y$ be a $\overline{G}$-equivariant map of spaces of type $G$-SWF for the trivial extension, where $X$ has level $s$ and $Y$ has level $t$. Let $\mu = deg( f^{S^1} ) \in H^{t-s}_G$ be the degree of $f^{S^1}$. Then for any non-zero $c \in H^*_G$, we have
\[
d_{G, c \mu }(X)-s \le d_{G , c}(Y)-t.
\]
\end{proposition}
\begin{proof}
We prove the result when $c \in H^{|c|}_G$ is homogeneous. The general case follows easily from this. The inclusion of the fixed points sets gives a commutative diagram.
\[
\xymatrix{
X \ar[r]^-f & Y  \\
X^{S^1} \ar[u]^-\iota \ar[r]^-{f^{S^1}} & Y^{S^1} \ar[u]^-\iota
}
\]
Consider the induced commutative diagram in equivariant cohomology:
\[
\xymatrix{
\widetilde{H}^*_{\overline{G}}( Y ) \ar[d]^-{\iota^*} \ar[r]^-{f^*} & \widetilde{H}^*_{\overline{G}}( X ) \ar[d]^-{\iota^*} \\
\widetilde{H}^*_{\overline{G}}( Y^{S^1} ) \ar[r]^-{(f^{S^1})^*} & \widetilde{H}^*_{\overline{G}}( X^{S^1} )
}
\]
From the definition of $d_{G , c}(Y)$, there exists some $x \in \widetilde{H}^{d_{G , c}(Y)+|c|}_{\overline{G}}( Y )$ such that 
\[
\iota^*(x) = c \, U^k \, \tau_{Y^{S^1}} + \cdots
\]
where $k = (d_{G,c}(Y) - t)/2$. Then by commutativity of the diagram, we have
\begin{align*}
\iota^*( f^*(x) ) &= (f^{S^1})^*(  \iota^*(x) )\\
&= (f^{S^1})^*( c \, U^k \, \tau_{Y^{S^1}} + \cdots) \\
&= c \mu \, U^k \, \tau_{X^{S^1}} + \cdots
\end{align*}
It follows that
\[
d_{G , c \mu}(X) \le d_{G,c}(Y) + |c| - |c\mu| = d_{G,c}(Y) - |\mu| = d_{G,c}(Y) - t + s.
\]
Hence
\[
d_{G , c\mu}(X) - s \le d_{G,c}(Y) - t.
\]
\end{proof}

\subsection{Alternative characterisation of $d_{G,c}$}

In this section we will give an alternative characterisation of $d_{G,c}$ which does not directly refer to $\iota^*$ and is sometimes more convenient for computations.

Let $X$ be a space of type $G$-SWF for the trivial extension $\overline{G}$. Set $\Lambda^*_G = \widetilde{H}^*_{\overline{G}}(X^{S^1})$. The inclusion of the fixed points $\iota \colon X^{S^1} \to X$ induces a map $\iota^* \colon \widetilde{H}^*_{\overline{G}}(X) \to \Lambda^*_G$. Recall that $\Lambda^*_G$ is a free $H^*_G[U]$ module of rank $1$. Let $\tau$ denote a generator of $\Lambda^*_G$, so $\Lambda^*_G \cong H^*_G[U]\tau$. Recall that we have a filtration $F_j$ on $\Lambda^*_G$ given by $F_j = H_G^{* \ge j} \Lambda_G^*$. Similarly there is a filtration on $\widetilde{H}^*_{\overline{G}}(X)$ which comes from the spectral sequence for equivariant cohomology. We will denote this filtration by $\mathcal{F}_j$. Then $\iota^*(\mathcal{F}_j) \subseteq F_j$ because the inclusion $\iota$ induces a map between spectral sequences.

Let $c \in H^*_G$ be a non-zero element of degree $|c|$. Recall that the invariant $d_{G,c}(X)$ is defined by
\[
d_{G,c}(X) = \min\{ i \; | \; \exists x \in \widetilde{H}^i_{\overline{G}}(X), \; \iota^*(x) = c U^k \tau \; ({\rm mod} \; F_{|c|+1}) \text{ for some } k \ge 0 \} - |c|.
\]

The localisation theorem in equivariant cohomology implies that upon localising with respect to $U$, $\iota^*$ becomes an isomorphism:
\[
\iota^* \colon U^{-1} \widetilde{H}^*_{\overline{G}}(X) \to U^{-1}\Lambda_G \cong H^*_G[U,U^{-1}]\tau.
\]
In particular, there exists an element $\theta \in \widetilde{H}^{2k+deg(\tau)}_{\overline{G}}(X)$ such that $\iota^*(\theta) = U^l \tau$ for some $l \ge 0$. Fix a choice of such a $\theta$. The localisation isomorphism implies that $\iota^*(x) = 0$ if and only if $U^k x = 0$ for some $k \ge 0$.

\begin{proposition}\label{prop:altd}
Let $c \in H^*_G$ be a non-zero element of degree $|c|$. Then
\[
d_{G,c}(X) = \min \{ i \; | \; \exists x \in \widetilde{H}^i_{\overline{G}}(X), \; U^n x = c U^k \theta \; ({\rm mod} \; \mathcal{F}_{|c|+1}) \text{ for some } n,k \ge 0 \} - |c|.
\]
\end{proposition}
\begin{proof}
Let
\[
a_{G,c}(X) = \min \{ i \; | \; \exists x \in \widetilde{H}^i_{\overline{G}}(X), \;  U^n x = c U^k \theta \; ({\rm mod} \; \mathcal{F}_{|c|+1}) \text{ for some } n,k \ge 0 \} - |c|.
\]
Then we need to show that $d_{G,c}(X) = a_{G,c}(X)$. Suppose $x \in \widetilde{H}^{a_{G,c}(X)+|c|}_{\overline{G}}(X)$ satisfies $U^n x = c U^k \theta \; ({\rm mod} \; \mathcal{F}_{|c|+1})$ for some $n,k \ge 0$. Then
\[
U^n \iota^*(x) = c U^k \iota^*(\theta) = c U^{k+l} \tau \; ({\rm mod} \; F_{|c|+1}).
\]
Since $U$ is injective on $\Lambda_G$ we must have $k+l \ge n$ and we can cancel $U^n$ from both sides to get
\[
\iota^*(x) = c U^{k+l-n} \tau \; ({\rm mod} \; F_{|c|+1}).
\]
Hence $d_{G,c}(X) \le deg(x) - |c| = a_{G,c}(X)$. Conversely, let $x \in \widetilde{H}^{d_{G,c}(X)+|c|}_{\overline{G}}(X)$ be such that $\iota^*(x) = c U^k \tau \; ({\rm mod} \; F_{|c|+1} )$ for some $k \ge 0$. Then
\[
\iota^*(x) = c U^k \tau + c_1 U^{k-1} \tau + c_2 U^{k-2} \tau + \cdots + c_k \tau
\]
where $c_i \in H_G^{|c|+2i}$. Since $\iota^*(\theta) = U^l \tau$, it follows that
\[
\iota^*( U^l x ) = \iota^*( c U^k \theta + c_1 U^{k-1} \theta + \cdots + c_k \theta ).
\]
Next recall that $\iota^*$ is an isomorphism after localising with respect to $U$. Hence if $\iota^*(y_1) = \iota^*(y_2)$, then $U^n y_1 = U^n y_2$ for some $n \ge 0$ and we have
\[
U^{n+l} x = c U^{n+k}\theta + c_1 U^{n+k-1}\theta + \cdots + c_k U^n \theta = c U^{n+k} \theta \; ({\rm mod} \; \mathcal{F}_{|c|+1} ).
\]
From the definition of $a_{G,c}(X)$, it follows that $a_{G,c}(X) \le deg(x) - |c| = d_{G,c}(X)$. We have shown $d_{G,c}(X) \le a_{G,c}(X)$ and $a_{G,c}(X) \le d_{G,c}(X)$, hence $d_{G,c}(X) = a_{G,c}(X)$.
\end{proof}

\subsection{Equivariant $d$-invariants for rational homology $3$-spheres}

We return to the setting that $Y$ is a rational homology $3$-sphere, $G$ is a finite group acting on $Y$ preserving the isomorphism class of a spin$^c$-structure $\mathfrak{s}$. Choose a $G$-invariant metric $g$ and let $G_{\mathfrak{s}}$ be the $S^1$-central extension of $G$ obtained by lifting $G$ to the spinor bundle corresponding to $\mathfrak{s}$. Now suppose that $G_{\mathfrak{s}}$ is a trivial extension, hence $G_{\mathfrak{s}} \cong \overline{G}$. From the construction of the Conley index, one finds that $SWF(Y , \mathfrak{s},g)$ is of type $G$-SWF at level $0$. 

\begin{definition}
Let $G$ act on $Y$ and let $\mathfrak{s}$ be a $G$-invariant spin$^c$-structure. Suppose that the corresponding $S^1$-extension $G_{\mathfrak{s}}$ is trivial and choose an isomorphism of extensions $G_{\mathfrak{s}} \cong \overline{G}$. For any $c \in H^*_G$ we  define the invariant $d_{G,c}(Y , \mathfrak{s})$ by
\[
d_{G,c}(Y , \mathfrak{s}) = d_{G,c}( SWF(Y , \mathfrak{s} , g) ) - 2n(Y,\mathfrak{s},g).
\]
We also set $\delta_{G,c}(Y,\mathfrak{s}) = d_{G,c}(Y , \mathfrak{s})/2$. 
\end{definition}

The definition of $d_{G,c}(Y,\mathfrak{s})$ does not depend on the choice of isomorphism $G_{\mathfrak{s}} \cong \overline{G}$ by Remark \ref{rem:ext}. The definition also does not depend on the choice of metric $g$ as a consequence of the relation $SWF(Y , \mathfrak{s} , g_1) = \Sigma^{SF_{G_{\mathfrak{s}}}(\{D_s \})} SWF(Y , \mathfrak{s} , g_0)$ and the Thom isomorphism.

We only define the invariants $d_{G,c}(Y , \mathfrak{s})$ in the case that $G_{\mathfrak{s}}$ is a trivial extension. This is because the definition of $d_{G,c}(Y , \mathfrak{s})$ uses localisation by $U$, but $U \in H^*_{S^1}$ does not necessarily extend to a class in $H^*_{G_{\mathfrak{s}}}$, unless $G_{\mathfrak{s}}$ is a trivial extension.

The inclusion $\iota \colon (V^0_\lambda(\mathbb{R}) )^+ \to I^{\mu}_{\lambda}$ of the $S^1$-fixed points of the Conley index desuspends to a map $\iota \colon \Sigma^{-V^{0}_{\lambda}(\mathbb{C})} S^0 \to SWF(Y , \mathfrak{s} , g)$, hence we get a homomorphism
\[
\iota^* \colon HSW^*_G(Y , \mathfrak{s}) \to \Lambda_G^*(Y,\mathfrak{s}),
\]
where we have set $\Lambda_G^*(Y,\mathfrak{s}) = \widetilde{H}^{* + 2n(Y , \mathfrak{s} , g)}_{\overline{G}}(  \Sigma^{-V^{0}_{\lambda}(\mathbb{C})} S^0 )$. This is a free $H^*_G[U]$-module and we let $\tau$ denote a generator. As in Section \ref{sec:gswf} we filter $\Lambda_G^*(Y,\mathfrak{s})$ by setting $F_j = H^{* \ge j}_G \Lambda_G^*(Y,\mathfrak{s})$. The construction of $\iota^*$ and $\Lambda_G^*(Y,\mathfrak{s})$ depend on the choice of metric $g$, but the construction for any two metrics are related by a canonical homomorphism. The $d$-invariants of $(Y,\mathfrak{s})$ are given by
\[
d_{G,c}(Y,\mathfrak{s}) = {\rm min}\{ 2k+j | \; \exists x \in {SWF}^{j+2k+|c|}_{G}(Y,\mathfrak{s}), \; \iota^*(x) \in F_{|c|}, \; \iota^*(x) = c \, U^k \, \tau \; ({\rm mod} \; F_{|c|+1} ) \}.
\]

Recall that $d(\overline{Y} , \mathfrak{s}) = -d(Y , \mathfrak{s})$. On the other hand, the behaviour of the invariants $d_{G,c}(Y,\mathfrak{s})$ under orientation reversal is not so straightforward. For example, it follows from Proposition \ref{prop:Gand1} that
\begin{equation}\label{equ:ineq}
-d_{G,1}(\overline{Y},\mathfrak{s}) \le d(Y , \mathfrak{s}) \le d_{G,1}(Y , \mathfrak{s}).
\end{equation}
In particular, $d_{G,1}(\overline{Y} , \mathfrak{s}) = -d_{G,1}(Y,\mathfrak{s})$ can only occur if $d_{G,1}(Y , \mathfrak{s}) = d(Y , \mathfrak{s})$ and $d_{G,1}(\overline{Y} , \mathfrak{s}) = d(\overline{Y},\mathfrak{s})$. From (\ref{equ:ineq}), we also get that
\[
d_{G,1}(Y , \mathfrak{s}) + d_{G,1}(\overline{Y} , \mathfrak{s}) \ge 0.
\]
We will show in Theorem \ref{thm:positive} that the invariants $d_{G,c}$ satisfy a stronger positivity condition.

\begin{proposition}\label{prop:lspaced}
Let $G$ act on $Y$ and let $\mathfrak{s}$ be a $G$-invariant spin$^c$-structure. Suppose that the corresponding extension $G_{\mathfrak{s}}$ is trivial. If $Y$ is an $L$-space (with respect to $\mathfrak{s}$ and $\mathbb{F})$, then for all non-zero $c \in H^*_G$  we have
\[
d_{G,c}(Y , \mathfrak{s}) = d(Y , \mathfrak{s}).
\]
\end{proposition}
\begin{proof}
If $Y$ is an $L$-space (with respect to $\mathfrak{s}$ and $\mathbb{F}$) then
\[
HSW^*(Y , \mathfrak{s} ) \cong \mathbb{F}[U]\theta,
\]
where $\theta$ has degree $d(Y , \mathfrak{s})$. From Theorem \ref{thm:degen}, there exists a class $\widehat{\theta} \in HSW^*_G(Y,\mathfrak{s})$ which maps to $\theta$ under the forgetful map $HSW^*_G(Y , \mathfrak{s}) \to HSW^*(Y , \mathfrak{s})$ and we have that $HSW^*_G(Y , \mathfrak{s})$ is a free $H^*_G[U]$-module generated by $\widehat{\theta}$. We must also have that $\iota^*(\widehat{\theta}) = U^k \tau \; ({\rm mod} \; F_1)$ for some $k \ge 0$, where $\tau$ is a generator of $\Lambda_G(Y , \mathfrak{s})$. This holds because $F_1$ is the kernel of the forgetful map $\Lambda_G(Y,\mathfrak{s}) \to \Lambda_{\{1\}}(Y,\mathfrak{s})$. So for any non-zero $c \in H^j_G$ we have that $\iota^*(c \widehat{\theta}) = c \iota^*(\widehat{\theta}) = c U^k \tau \; ({\rm mod} \; F_{1+|c|})$, hence $d_{G,c}(Y , \mathfrak{s}) \le deg(c \widehat{\theta}) - j = \deg(\widehat{\tau}) = d(Y , \mathfrak{s})$. That is, $d_{G,c}(Y,\mathfrak{s}) \le d(Y , \mathfrak{s})$ for all non-zero homogeneous $c$. On the other hand it is clear that there is no class of lower degree in $HSW^*_G(Y , \mathfrak{s})$ which maps under $\iota^*$ to a class of the form $c U^{k'} \tau \; ({\rm mod} \; F_{1+|c|})$. Hence $d_{G,c}(Y,\mathfrak{s}) = d(Y , \mathfrak{s})$ for all non-zero homogeneous $c$. Clearly the result extends to all non-zero $c$.
\end{proof}

\section{Behaviour under cobordisms}\label{sec:cobord}

We show that equivariant cobordisms of rational homology $3$-spheres induce maps on equivariant Seiberg--Witten--Floer cohomology. We follow the construction of Manolescu \cite{man}, incorporating the corrections due to Khandhawit \cite{kha}. Since our construction is a straightforward extension of that of Manolescu and Khandhawit, differing only in the replacement of $S^1$ by the larger group $G_{\mathfrak{s}}$, we will be brief.

\subsection{Finite dimensional approximation}\label{sec:cobordismfda}

Let $W$ be a compact, oriented smooth $4$-manifold with boundary $Y = \partial W$ a disjoint union of rational homology spheres $Y = \cup_j Y_j$. Assume further that $b_1(W) = 0$ and that $W$ is connected. If $\mathfrak{s}$ is a spin$^c$-structure on $W$, then the restriction of $\mathfrak{s}$ to $Y$ determines a spin$^c$-structure $\mathfrak{s}|_Y$ on $Y$. Since the boundary of $W$ is a union of rational homology $3$-spheres, we have $H^2(W , \partial W ; \mathbb{R}) \cong H^2(W ; \mathbb{R})$ and by Poincar\'e--Lefschetz duality we obtain a non-degenerate intersection form on $H^2(W ; \mathbb{R})$. Given a metric $g$ on $W$ which is isometric to a product metric in a collar neighbourhood of $\partial W$, we let $H^+(W)$ denote the space of self-dual $L^2$-harmonic $2$-forms on the cylindrical end manifold $\hat{W}$ obtained from $W$ by attaching half-infinite cylinders $[0,\infty) \times Y$ to $W$. It follows from \cite[Proposition 4.9]{aps1} that the natural map $H^+(W) \to H^2(W ; \mathbb{R})$ is injective and identifies $H^+(W)$ with a maximal positive definite subspace of $H^2(W ; \mathbb{R})$.

Suppose now that $G$ acts smoothly and orientation preservingly on $W$ and that this action sends each connected component of $\partial W$ to itself. Hence by restriction $G$ acts on each $Y_i$ by orientation preserving diffeomorphisms. Assume further that $G$ preserves the isomorphism class of a spin$^c$-structure $\mathfrak{s}$ on $W$. Set $\mathfrak{s}_i = \mathfrak{s}|_{Y_i}$. Then the action of $G$ on $Y_i$ preserves $\mathfrak{s}_i$. Similar to Section \ref{sec:lifting} we obtain an $S^1$-extension $G_{\mathfrak{s}}$ of $G$. Restricting to $Y_i$, we obtain an isomorphism of extensions $G_{\mathfrak{s}} \cong G_{\mathfrak{s}_i}$. Hence if $G_{\mathfrak{s}}$ is split, then it follows that each of the extensions $G_{\mathfrak{s}_i}$ is also split. Moreover a splitting of $G_{\mathfrak{s}}$ determines corresponding splittings of each $G_{\mathfrak{s}_i}$.

Choose a $G$-invariant metric $g$ on $W$ which is isometric to a product $( -\epsilon , 0] \times Y$ in some equivariant collar neighbourhood of $Y$ (see \cite[Theorem 3.5]{kan} for existence of equivariant collar neighbourhoods). To see that such a metric exists, first choose a $G$-invariant metric $g_Y$ on $Y$. Then choose an arbitrary metric $g'$ on $W$ which equals $(dt)^2 + g_Y$ in some equivariant collar neighbourhood $(-\epsilon , 0] \times Y$. Then let $g$ be obtained from $g'$ by averaging over $G$. Let $S^\pm$ denote the spinor bundles on $W$ corresponding to $\mathfrak{s}$. We note here that under these assumptions $G$ preserves the subspace $H^+(W)$ of $H^2(W ; \mathbb{R})$ defined by $g$. Let $\Omega^1_g(W)$ denote the space of $1$-forms on $W$ in double Coulomb gauge with respect to $Y$ \cite[Definition 1]{kha}. This space is easily seen to be preserved by the action of $G$ on $1$-forms. The double Coulomb gauge condition ensures that if $a \in \Omega^1_g(W)$ and $\phi \in \Gamma(S^+)$, then $(a , \phi)|_{Y_j}$ lies in the global Coulomb slice corresponding to $Y_j$. Let us temporarily assume that $Y = \partial W$ is connected. Let $\widehat{A}$ be a spin$^c$-connection on $W$ such that in a collar neighbourhood of $Y$ it equals the pullback of $A_0$. Using the same argument as in Section \ref{sec:lifting}, we can assume that $\widehat{A}$ is $G_{\mathfrak{s}}$-invariant. Then using $\widehat{A}$ as a reference connection, we obtain a map which may be thought of as the Seiberg--Witten equations on $W$ together with boundary conditions:
\begin{align*}
SW^\mu \colon & i \Omega^1_g(W) \oplus \Gamma(S^+) \to i \Omega^2_+(W) \oplus \Gamma(S^-) \oplus V^\mu_{-\infty}, \\
( a , \phi ) &\mapsto ( F^+_{\widehat{A}+a} - \sigma(\phi,\phi) , D_{\widehat{A}+a}(\phi) ,  p^\mu (a,\phi)|_{Y} )
\end{align*}
where $p^\mu$ is the orthogonal projection from $V$ to $V^\mu_{-\infty}$. Taking a finite dimensional approximation as described in \cite{man}, \cite{kha}, one obtains a map
\[
\Psi_{\mu,\lambda,U,U'} \colon (U')^+ \to (U^+) \wedge I^\mu_{\lambda},
\]
where $U' \subset i \Omega^1_g(W) \oplus \Gamma(S^+)$ and $U \subset i \Omega^2_+(W) \oplus \Gamma(S^-)$ are finite dimensional $G$-invariant subspaces which satisfy 
\begin{equation}\label{equ:uu'}
U \oplus V^0_\lambda \oplus Ker(L^0) \cong U' \oplus Coker(L^0)
\end{equation}
and $L^0$ is a Fredholm linear operator defined in \cite[Section 9]{man}. Since $SWF(Y , \mathfrak{s} , g) = \Sigma^{-V^0_\lambda} I^\mu_\lambda$, we can re-write the map $\Psi_{\mu , \lambda , U, U'}$ as
\[
\Psi_{\mu,\lambda,U,U'} \colon (U')^+ \to (U)^+ \wedge (V^0_\lambda)^+ \wedge SWF(Y , \mathfrak{s} , g).
\]
Taking the smash product with $Ker(L^0)$ and using (\ref{equ:uu'}), we see that $\Psi_{\mu,\lambda,U,U'}$ is stably equivalent to a map
\[
f \colon Ker(L^0)^+ \to Coker(L^0)^+ \wedge SWF(Y , \mathfrak{s} , g).
\]
The real part of $L^0$ has zero kernel and cokernel isomorphic to $H^+(W)$. The complex part of $L^0$ can be identified with the Dirac operator $D_{\widehat{A}}$ with Atiyah--Patodi--Singer (APS) boundary conditions. Thus
\[
Ker(L^0) \cong Ker_{APS}( D^+_{\widehat{A}} ), \quad Coker(L^0) \cong H^+(W) \oplus Coker_{APS}( D^+_{\widehat{A}} )
\]
where $Ker_{APS}( D^+_{\widehat{A}} )$, $Coker_{APS}( D^+_{\widehat{A}} )$ denote the kernel and cokernel of $D^+_{\widehat{A}}$ with APS boundary conditions. Hence we obtain a $G_{\mathfrak{s}}$-equivariant map
\[
f \colon Ker_{APS}( D^+_{\widehat{A}} )^+ \to (H^+(W))^+ \wedge Coker_{APS}( D^+_{\widehat{A}} )^+ \wedge SWF(Y , \mathfrak{s} , g).
\]
Note that $f$ is only a map in the stable sense. That is, $f$ is a morphism in the category $\mathfrak{C}(G_{\mathfrak{s}})$.

Recall that the $S^1$-fixed point set of $I^\mu_\lambda$ is $V^0_\lambda(\mathbb{R})^+$. The inclusion $(V^0_\lambda(\mathbb{R})) \to I^\mu_\lambda$ of the $S^1$-fixed points desuspends to a map $\iota \colon S^0 \to SWF(Y , \mathfrak{s} , g)$. By restricting to $S^1$-fixed points we obtain a commutative diagram
\[
\xymatrix{
Ker_{APS}( D^+_{\widehat{A}} )^+ \ar[r]^-{f} & (H^+(W))^+ \wedge Coker_{APS}( D^+_{\widehat{A}} )^+ \wedge SWF(Y , \mathfrak{s} , g) \\
S^0 \ar[u] \ar[r]^-{f^{S^1}} & (H^+(W))^+ \ar[u]
}
\]
Using that the Seiberg--Witten equations reduce to linear equations on the $S^1$-fixed point set, one finds that $f^{S^1} \colon S^0 \to (H^+(W))^+$ is the obvious map given by the one-point compactification of the inclusion $\{0\} \to H^+(W)$. Thus according to Definition \ref{def:deg}, $f^{S^1}$ has degree equal to $e(H^+(W))$, the image of the equivariant Euler class of $H^+(W)$ in $H^*_G(pt ; \mathbb{F})$. For instance if $\mathbb{F} = \mathbb{Z}/2\mathbb{Z}$, then $e(H^+(W))$ is the $b_+(W)$-th equivariant Stiefel--Whitney class. We will refer to $e(H^+(W))$ as the {\em $\mathbb{F}$-Euler class of $H^+(W)$}.

So far we have restricted to the case that the boundary $\partial W$ is connected. More generally if $\partial W = \cup_{j} Y_j$ is a union of rational homology $3$-spheres then much the same construction applies. The Conley index $I^\mu_\lambda$ is now given by the smash product of the Conley indices of each component, hence $f$ is now a map of the form
\[
f \colon Ker_{APS}( D^+_{\widehat{A}} )^+ \to (H^+(W))^+ \wedge Coker_{APS}( D^+_{\widehat{A}} )^+ \wedge \bigwedge_{j} SWF(Y_j , \mathfrak{s}_j , g_j).
\]
We still have that the degree of $f^{S^1}$ is $e(H^+(W))$.

\subsection{Equivariant Fr\o yshov inequality}

In this section we prove an equivariant generalisation of Fr\o yshov's inequality \cite{fro}.

\begin{theorem}\label{thm:don}
Let $W$ be a smooth, compact, oriented $4$-manifold with boundary and with $b_1(W) = 0$. Suppose that $G$ acts smoothly on $W$ preserving the orientation and a spin$^c$-structure $\mathfrak{s}$. Suppose that the extension $G_{\mathfrak{s}}$ is trivial. Suppose each component of $\partial W$ is a rational homology $3$-sphere and that $G$ sends each component of $\partial W$ to itself. Let $e \in H^{b_+(W)}_G$ be the $\mathbb{F}$-Euler class of any $G$-invariant maximal positive definite subspace of $H^2(W ; \mathbb{R})$. Let $c \in H^*_G$ and suppose that $c e \neq 0$.
\begin{itemize}
\item[(1)]{If $\partial W = Y$ is connected, then 
\[
\delta(W , \mathfrak{s}) \le \delta_{G,c}(Y , \mathfrak{s}|_Y) \; \; \text{ and } \; \; \delta_{G,c e}(\overline{Y} , \mathfrak{s}|_Y ) \le \delta( \overline{W} , \mathfrak{s} ),
\]
where we have defined
\[
\delta(W , \mathfrak{s}) = \frac{ c_1(\mathfrak{s})^2 - \sigma(W)}{8}.
\]
}
\item[(2)]{If $\partial W = \overline{Y}_1 \cup Y_2$ has two connected components, then 
\[
\delta_{G,c e}(Y_1 , \mathfrak{s}|_{Y_1}) + \delta(W , \mathfrak{s}) \le \delta_{G,c}(Y_2 , \mathfrak{s}|_{Y_2}).
\]
}
\end{itemize}

\end{theorem}
\begin{proof}
We will give the proof in the case that $W$ is connected. The general case follows easily from this by applying the theorem to each component of $W$. To simplify notation we will write $\mathfrak{s}$ instead of $\mathfrak{s}|_Y$ and write $g$ instead of $g|_Y$. In case (1), $\partial W = Y$ is connected.

As in Section \ref{sec:cobordismfda}, choosing suitable metrics and reference connections we obtain a stable map
\[
f \colon Ker_{APS}( D^+_{\widehat{A}} )^+ \to (H^+(W))^+ \wedge Coker_{APS}( D^+_{\widehat{A}} )^+ \wedge SWF(Y , \mathfrak{s} , g)
\]
such that the degree of $f^{S^1}$ is $e$. Applying Proposition \ref{prop:ineq} to $f$, we obtain
\[
2 dim_{\mathbb{C}}( Ker_{APS}( D^+_{\widehat{A}} )) + b_+(W) \le  b_+(W) + 2 dim_{\mathbb{C}}( Coker_{APS}( D^+_{\widehat{A}} )) + d_c(Y , \mathfrak{s}) + 2 n(Y , \mathfrak{s} , g)
\]
which simplifies to
\[
ind_{APS}( D^+_{\widehat{A}} ) \le d_c(Y , \mathfrak{s})/2 + n(Y , \mathfrak{s} , g) = \delta_c(Y , \mathfrak{s}) + n(Y , \mathfrak{s} , g).
\]
Combined with Equation (\ref{equ:aps}) we get $\delta(W , \mathfrak{s} ) \le \delta_c(Y , \mathfrak{s})$.

Next recall from Section \ref{sec:duality} the duality map
\[
\varepsilon \colon SWF(Y , \mathfrak{s} , g) \wedge SWF( \overline{Y} , \mathfrak{s} , g) \to S^{-k(D) \mathbb{C}}
\]
where $k(D) = dim_{\mathbb{C}}( Ker(D) )$. By the definition of equivariant duality, we have that
\[
\varepsilon^{S^1} \colon SWF(Y , \mathfrak{s} , g)^{S^1} \wedge SWF( \overline{Y} , \mathfrak{s} , g)^{S^1} \to S^0
\]
is a non-equivariant duality. It follows that $\varepsilon^{S^1}$ has degree $1$. Taking the map $f$, suspending by $SWF(\overline{Y} , \mathfrak{s} , g)$ and composing with $\varepsilon$, we obtain a stable map
\[
h \colon Ker_{APS}( D^+_{\widehat{A}} )^+ \wedge SWF( \overline{Y} , \mathfrak{s} , g) \to (H^+(W))^+ \wedge Coker_{APS}( D^+_{\widehat{A}} )^+ \wedge S^{-k(D) \mathbb{C}}
\]
such that the degree of $h^{S^1}$ is $e$. Applying Proposition \ref{prop:ineq} to $h$ we obtain
\begin{align*}
2 dim_{\mathbb{C}}( Ker_{APS}( D^+_{\widehat{A}} ) + & d_{G,ce}(\overline{Y} , \mathfrak{s} , g) + 2n(\overline{Y} , \mathfrak{s} , g) + b_+(W) \\
& \quad \le  b_+(W) + 2 dim_{\mathbb{C}}( Coker_{APS}( D^+_{\widehat{A}} ) - 2 k(D)
\end{align*}
which simplifies to
\[
ind_{APS}( D^+_{\widehat{A}} ) +\delta_{G,ce}(\overline{Y} , \mathfrak{s} , g) + n(\overline{Y} , \mathfrak{s} , g) \le - k(D).
\]
Using Equations (\ref{equ:aps}) and (\ref{equ:nnbar}), we obtain $\delta(W , \mathfrak{s}) \le -\delta_{G,ce}(\overline{Y} , \mathfrak{s})$, or equivalently $\delta_{G,c e}(\overline{Y} , \mathfrak{s}|_Y ) \le \delta( \overline{W} , \mathfrak{s} )$.

The proof of case (2) is similar. We start with the map
\[
f \colon Ker_{APS}( D^+_{\widehat{A}} )^+ \to (H^+(W))^+ \wedge Coker_{APS}( D^+_{\widehat{A}} )^+ \wedge SWF(\overline{Y_1} , \mathfrak{s} , g) \wedge SWF(Y_2 , \mathfrak{s} , g).
\]
Suspending by $SWF(Y_1 , \mathfrak{s} , g)$ and applying the duality map corresponding to $Y_1$ we obtain a map
\begin{align*}
h \colon & Ker_{APS}( D^+_{\widehat{A}} )^+ \wedge SWF( Y_1 , \mathfrak{s} , g) \\
& \quad \to (H^+(W))^+ \wedge Coker_{APS}( D^+_{\widehat{A}} )^+ \wedge SWF(Y_2 , \mathfrak{s} , g) \wedge S^{-k(D_1) \mathbb{C}}
\end{align*}
where $k(D_1)$ is the dimension of the kernel of the Dirac operator on $Y_1$. Applying Proposition \ref{prop:ineq} to this map and simplifying, we obtain the inequality $\delta_{G,c e}(Y_1 , \mathfrak{s}|_{Y_1}) + \delta(W , \mathfrak{s}) \le \delta_{G,c}(Y_2 , \mathfrak{s}|_{Y_2})$.

\end{proof}

\begin{definition}
Let $(Y_1,\mathfrak{s}_1),(Y_2,\mathfrak{s}_2)$ be rational homology $3$-spheres equipped with spin$^c$-structures. Suppose that $G$ acts orientation preservingly on $Y_1,Y_2$ and preserves the spin$^c$-structures $\mathfrak{s}_1,\mathfrak{s}_2$. A {\em $G$-equivariant rational homology cobordism from $(Y_1 , \mathfrak{s}_1)$ to $(Y_2 , \mathfrak{s}_2)$} is a rational homology cobordism $W$ from $Y_1$ to $Y_2$ such that the $G$-action and spin$^c$-structure $\mathfrak{s}_1 \cup \mathfrak{s}_2$ on $\partial W$ extend over $W$. We say that $(Y_1 , \mathfrak{s}_1), (Y_2 , \mathfrak{s}_2)$  are {\em $G$-equivariantly rational homology cobordant} if there exists a $G$-equivariant rational homology cobordism from $(Y_1 , \mathfrak{s}_1)$ to $(Y_2 , \mathfrak{s}_2)$.

Similarly we define the notion of a {\em $G$-equivariant integral homology cobordism} and say that two integral homology $3$-spheres $Y_1,Y_2$ on which $G$ acts are {\em $G$-equivariantly integral homology cobordant} if there is a $G$-equivariant integral homology cobordism from $Y_1$ to $Y_2$. Note that since $Y_1,Y_2$ are integral homology $3$-spheres, they have unique spin$^c$-structures which are automatically $G$-invariant and any $G$-equivariant integral homology cobordism from $Y_1$ to $Y_2$ has a unique spin$^c$-structure which restricts on the boundary to the unique spin$^c$-structures on $Y_1,Y_2$.
\end{definition}

\begin{corollary}\label{cor:inv}
The $G$-equivariant $\delta$-invariants $\delta_{G,c}(Y , \mathfrak{s})$ are invariant under $G$-equivariant rational homology cobordism, that is, if $(W,\mathfrak{s})$ is a $G$-equivariant rational homology cobordism from $(Y_1, \mathfrak{s}_1)$ to $(Y_2 , \mathfrak{s}_2)$ and if the extensions $G_{\mathfrak{s}_1}$, $G_{\mathfrak{s}_2}$ are trivial, then $\delta_{G,c}(Y_1 , \mathfrak{s}_1) = \delta_{G,c}(Y_2 , \mathfrak{s}_2)$ for all $c \in H^*_G$.
\end{corollary}
\begin{proof}
Since $W$ is a rational homology cobordism, we have $H^2(W ; \mathbb{R}) = 0$. So $\delta(W , \mathfrak{s}) = 0$ and $e = e(H^+(W)) = 1$. Therefore Theorem \ref{thm:don} gives
\[
\delta_{G,c}(Y_1,\mathfrak{s}_1) \le \delta_{G,c}(Y_2 , \mathfrak{s}_2).
\]
Similarly, viewing $\overline{W}$ as a $G$-equivariant rational homology cobordism from $Y_2$ to $Y_1$, we get
\[
\delta_{G,c}(Y_2,\mathfrak{s}_2) \le \delta_{G,c}(Y_1 , \mathfrak{s}_1).
\]
Hence $\delta_{G,c}(Y_1,\mathfrak{s}_1) = \delta_{G,c}(Y_2 , \mathfrak{s}_2)$.

\end{proof}

\begin{theorem}\label{thm:positive}
Let $Y$ be a rational homology $3$-sphere, $G$ a finite group acting on $Y$ preserving orientation and the isomorphism class of a spin$^c$-structure $\mathfrak{s}$ and suppose that $G_{\mathfrak{s}}$ is a trivial extension. Then for any $c_1,c_2 \in H^*_G$ with $c_1 c_2 \neq 0$ we have
\[
\delta_{c_1}(Y) + \delta_{c_2}(\overline{Y}) \ge 0.
\]
\end{theorem}
\begin{proof}
The proof is similar to that of Theorem \ref{thm:don}. Let $W = [0,1] \times Y$ be the trivial cobordism from $Y$ to itself. Choosing suitable metrics and reference connections we obtain a stable map 
\[
f \colon Ker_{APS}( D^+_{\widehat{A}} )^+ \to Coker_{APS}( D^+_{\widehat{A}} )^+ \wedge SWF(\overline{Y} , \mathfrak{s} , g) \wedge SWF(Y , \mathfrak{s} , g).
\]
Note that $H^+(W) = \{0\}$ and hence $e(H^+(W)) = 1$. Applying Proposition \ref{prop:ineq} to this map we see that for any $c_1,c_2 \in H^*_G$ with $c_1 c_2 \neq 0$, we have
\[
ind_{APS}( D^+_{\widehat{A}} ) \le \delta_{G,c_1c_2}( SWF(\overline{Y},\mathfrak{s},g) \wedge SWF(Y , \mathfrak{s} , g) ).
\]
From the definition of the $\delta$-invariant it is clear that 
\[
\delta_{G,c_1c_2}( SWF(\overline{Y},\mathfrak{s},g) \wedge SWF(Y , \mathfrak{s} , g) ) \le \delta_{G,c_1}( SWF(\overline{Y} , \mathfrak{s} , g)) + \delta_{G,c_2}(SWF( Y , \mathfrak{s} , g) )
\]
and hence
\[
ind_{APS}( D^+_{\widehat{A}} ) \le \delta_{G,c_1}(\overline{Y} , \mathfrak{s} ) + \delta_{G,c_2}(Y , \mathfrak{s} ) + n(\overline{Y},\mathfrak{s},g) + n(Y , \mathfrak{s} , g).
\]
On the other hand, for $W = [0,1] \times Y$, Equation (\ref{equ:aps}) reduces to
\[
ind_{APS}( D^+_{\widehat{A}}) = n(\overline{Y},\mathfrak{s},g) + n(Y , \mathfrak{s} , g).
\]
Hence we obtain $0 \le \delta_{G,c_1}(\overline{Y} , \mathfrak{s} ) + \delta_{G,c_2}(Y , \mathfrak{s} )$.
\end{proof}

\subsection{Induced cobordism maps}\label{sec:indcobord}

In this section we show that equivariant cobordisms induce maps on equivariant Seiberg--Witten--Floer cohomology.

\begin{theorem}
Let $W$ be a smooth, compact, oriented $4$-manifold with boundary and with $b_1(W) = 0$. Suppose that $G$ acts smoothly on $W$ preserving the orientation and a spin$^c$-structure $\mathfrak{s}$. Suppose that $\partial W = \overline{Y}_1 \cup Y_2$ where $Y_1,Y_2$ are rational homology $3$-spheres and set $\mathfrak{s}_i = \mathfrak{s}|_{Y_i}$. Suppose $G$ sends $Y_i$ to itself. Then there is a morphism of graded $H^*_{G_{\mathfrak{s}}}$-modules
\[
SW_G(W,\mathfrak{s}) \colon HSW^*_G(Y_2 , \mathfrak{s}_2) \to HSW^{*+b_+(W)-2\delta(W,\mathfrak{s})}_G(Y_1, \mathfrak{s}_1)
\]
such that the following diagram commutes
\[
\xymatrix{
HSW^*_G(Y_2 , \mathfrak{s}_2) \ar[rr]^-{SW_G(W,\mathfrak{s})} \ar[d] & & HSW^{*+b_+(W)-2\delta(W,\mathfrak{s})}_G(Y_1 , \mathfrak{s}_1) \ar[d] \\
HSW^*(Y_2 , \mathfrak{s}_2) \ar[rr]^-{SW(W,\mathfrak{s})} & & HSW^{*+b_+(W)-2\delta(W,\mathfrak{s})}(Y_1 , \mathfrak{s}_1)
}
\]
where the vertical arrows are the forgetful maps to non-equivariant Seiberg--Witten--Floer cohomology and $SW(W,\mathfrak{s})$ is the morphism of Seiberg--Witten--Floer cohomology groups induced by $(W,\mathfrak{s})$.

\end{theorem}
\begin{proof}
We give the proof in the case $W$ is connected. The general case follows by a similar argument. As in the proof of Theorem \ref{thm:don}, choosing suitable metrics and reference connections, we obtain a stable map
\[
h \colon S^{ind_{APS}(D^+_{\widehat{A}})} \wedge SWF( Y_1 , \mathfrak{s}_1 , g_1) \to (H^+(W))^+ \wedge SWF(Y_2 , \mathfrak{s}_2 , g_2) \wedge S^{-k(D_1) \mathbb{C}}
\]
where $k(D_1)$ is the dimension of the kernel of the Dirac operator on $Y_1$. The induced map in equivariant cohomology takes the form
\[
h^* \colon \widetilde{H}^j_{G_{\mathfrak{s}}}((H^+(W))^+ \wedge SWF(Y_2 , \mathfrak{s}_2 , g_2) \wedge S^{-k(D_1) \mathbb{C}}) \to \widetilde{H}^j_{G_{\mathfrak{s}}}( S^{ind_{APS}(D^+_{\widehat{A}})} \wedge SWF( Y_1 , \mathfrak{s}_1 , g_1) ).
\]
Using the Thom isomorphism, this is equivalent to
\[
h^* \colon \widetilde{H}^{j-b_+(W)+2k(D_1)}_{G_{\mathfrak{s}}}( SWF(Y_2 , \mathfrak{s}_2 , g_2) ) \to \widetilde{H}^{j-2 ind_{APS}(D^+_{\widehat{A}})}_{G_{\mathfrak{s}}}( SWF( Y_1 , \mathfrak{s}_1 , g_1) ).
\]
From Equation (\ref{equ:aps}), $ind_{APS}=\delta(W,\mathfrak{s})+n(Y_2,\mathfrak{s}_2,g_2)+n(\overline{Y}_1,\mathfrak{s}_1,g_1)$ and from Equation (\ref{equ:nnbar}) $n(Y_1,\mathfrak{s}_1,g_1) + n(\overline{Y}_1,\mathfrak{s}_1,g_1) = -k(D_1)$. Replacing $j$ by $j+b_+(W)-2k(D_1)+2n(Y_2,\mathfrak{s}_2,g_2)$, we see that $h^*$ takes the form
\[
h^* \colon \widetilde{H}^{j+2n(Y_2,\mathfrak{s}_2,g_2)}_{G_{\mathfrak{s}}}( SWF(Y_2 , \mathfrak{s}_2 , g_2) ) \to \widetilde{H}^{j+b_+(W)-2\delta(W,\mathfrak{s})+2n(Y_1,\mathfrak{s}_1,g_1) }_{G_{\mathfrak{s}}}( SWF( Y_1 , \mathfrak{s}_1 , g_1) ).
\]
Then since $HSW^*_G(Y_i , \mathfrak{s}_i) = \widetilde{H}^{*+2n(Y_i , \mathfrak{s}_i , g_i)}_{G_{\mathfrak{s}}}( SWF(Y_i , \mathfrak{s}_i , g_i) )$, we see that $h^*$ is equivalent to a map
\[
SW_G(W,\mathfrak{s}) \colon HSW^*_G(Y_2 , \mathfrak{s}_2) \to HSW^{*+b_+(W)-2\delta(W,\mathfrak{s})}_G(Y_1, \mathfrak{s}_1).
\]
Since this is a map of equivariant cohomologies induced by an equivariant map of spaces, it follows that $SW_G(W,\mathfrak{s})$ is a morphism of graded $H^*_{G_{\mathfrak{s}}}$-modules. Restricting to the subgroup $S^1 \to G_{\mathfrak{s}}$, we obtain the commutative diagram in the statement of the theorem.

\end{proof}

\section{The case $G = \mathbb{Z}_p$}\label{sec:zp}

In this section we specialise to the case $G = \mathbb{Z}_p$, $\mathbb{F} = \mathbb{Z}_p$, where $p$ is a prime number. Then for $p=2$ we have $H^*_G \cong \mathbb{F}[Q]$, where $deg(Q) = 1$ and if $p$ is odd we have $H^*_G \cong \mathbb{F}[R,S]/(R^2)$, where $deg(R)=1$, $deg(S)=2$. Suppose $G = \langle \tau \rangle$ acts smoothly and orientation preservingly on a rational homology $3$-sphere $Y$, preserving a spin$^c$-structure $\mathfrak{s}$. The action of $G$ is equivalent to giving an orientation preserving diffeomorphism $\tau \colon Y \to Y$ such that $\tau^p = id$ and $\tau^*(\mathfrak{s}) = \mathfrak{s}$. Choose a lift $\tau' \in G_{\mathfrak{s}}$ of $\tau$. Then $(\tau')^p = \zeta$ for some $\zeta \in S^1$. Replacing $\tau'$ by $\widetilde{\tau} = \zeta^{-1/p} \tau'$, where $\zeta^{1/p}$ is a $p$-th root of $\zeta$, we see that $\widetilde{\tau}^p = id$. Hence $G_{\mathfrak{s}}$ is a trivial extension. 

\subsection{$\delta$-invariants}

\begin{definition}
If $p=2$, then for any integer $j \ge 0$, we define $d_{j}(Y , \mathfrak{s} , \tau,2) = d_{\mathbb{Z}_2 , Q^j}(Y , \mathfrak{s})$. If $p$ is odd, then for any integer $j \ge 0$, we define $d_{j}(Y , \mathfrak{s} , \tau , p) = d_{\mathbb{Z}_p , S^j}(Y , \mathfrak{s})$. We also set $\delta_{j}(Y , \mathfrak{s} , \tau,p) = d_{j}(Y , \mathfrak{s} , \tau,p)/2$. When $p$ and $\tau$ are understood we will omit them from the notation and simply write $d_j(Y,\mathfrak{s})$ and $\delta_j(Y ,\mathfrak{s})$.
\end{definition}

In the case $p$ is odd, one may also consider the invariants $d_{\mathbb{Z}_p , RS^j}(Y , \mathfrak{s})$. For simplicity we will not consider these invariants.

\begin{theorem}\label{thm:deltaz2}
We have the following properties:
\begin{itemize}
\item[(1)]{$\delta_{0}(Y , \mathfrak{s}) \ge \delta(Y , \mathfrak{s})$.}
\item[(2)]{$\delta_{j+1}(Y , \mathfrak{s}) \le \delta_{j}(Y , \mathfrak{s})$ for all $j \ge 0$.}
\item[(3)]{The sequence $\{ \delta_{j}(Y , \mathfrak{s}) \}_{j \ge 0}$ is eventually constant.}
\item[(4)]{$\delta_j(Y , \mathfrak{s}) + \delta_j( \overline{Y} , \mathfrak{s}) \ge 0$ for all $j \ge 0$.}
\item[(5)]{If $Y$ is an $L$-space, then $\delta_j(Y , \mathfrak{s}) = \delta(Y , \mathfrak{s})$ for all $j \ge 0$.}
\end{itemize}
\end{theorem}
\begin{proof}
(1) is a restatement of Proposition \ref{prop:Gand1}. (2) follows from Proposition \ref{prop:din} taking $c_1 = Q^j$ and $c_2 = Q$ in the case $p=2$ and $c_1 = S^j$, $c_2 = S$ in the case $p$ is odd. (4) is a special case of Theorem \ref{thm:positive}. For (3), first note that the difference $\delta_j(Y , \mathfrak{s}) - \delta_{j+1}(Y , \mathfrak{s})$ is always an integer because $\delta_{G,c}(Y , \mathfrak{s}) + n(Y , \mathfrak{s} , g) \in \mathbb{Z}$ for any metric $g$. From (2) and (4) and the fact that $n(Y , \mathfrak{s} , g) + n(\overline{Y} , \mathfrak{s} , g) \in \mathbb{Z}$, it follows that $\delta_j(Y , \mathfrak{s}) + \delta_j( \overline{Y} , \mathfrak{s})$ is a non-negative, decreasing, integer-valued function. Hence the value of $\delta_j(Y , \mathfrak{s}) + \delta_j( \overline{Y} , \mathfrak{s})$ must eventually be constant. Using (2) again, it follows that $\delta_j(Y , \mathfrak{s})$ and $\delta_j(\overline{Y} , \mathfrak{s})$ are eventually constant. (5) is a restatement of Proposition \ref{prop:lspaced}.
\end{proof}

Next, we specialise Theorem \ref{thm:don} to the case $G = \mathbb{Z}_p$.

\begin{theorem}\label{thm:don2}
Let $W$ be a smooth, compact, oriented $4$-manifold with boundary and with $b_1(W) = 0$. Suppose that $\tau \colon W \to W$ is an orientation preserving diffeomorphism of order $p$ and $\mathfrak{s}$ a spin$^c$-structure preserved by $\tau$. Suppose each component of $\partial W$ is a rational homology $3$-sphere and that $\tau$ sends each component of $\partial W$ to itself. Suppose that the subspace of $H^2(W ; \mathbb{R})$ fixed by $\tau$ is negative definite. Then for all $j \ge 0$ we have
\begin{itemize}
\item[(1)]{If $\partial W = Y$ is connected, then 
\[
\delta(W , \mathfrak{s}) \le \delta_{j}(Y , \mathfrak{s}|_Y) \; \; \text{ and } \; \; \begin{cases} \delta_{j+b_+(W)}(\overline{Y} , \mathfrak{s}|_Y ) \le \delta( \overline{W} , \mathfrak{s} ) & (p = 2), \\ \delta_{j+b_+(W)/2}(\overline{Y} , \mathfrak{s}|_Y ) \le \delta( \overline{W} , \mathfrak{s} ) & (p \text{ odd}). \end{cases}
\]
}
\item[(2)]{If $\partial W = \overline{Y}_1 \cup Y_2$ has two connected components, then 
\[
\begin{cases} \delta_{j+b_+(W)}(Y_1 , \mathfrak{s}|_{Y_1}) + \delta(W , \mathfrak{s}) \le \delta_{j}(Y_2 , \mathfrak{s}|_{Y_2}) & (p=2), \\
\delta_{j+b_+(W)/2}(Y_1 , \mathfrak{s}|_{Y_1}) + \delta(W , \mathfrak{s}) \le \delta_{j}(Y_2 , \mathfrak{s}|_{Y_2}) & (p \text{ odd}).
\end{cases}
\]
}
\end{itemize}

\end{theorem}
\begin{proof}
Let $H^+(W)$ denote a $\tau$-invariant maximal positive definite subspace of $H^2(W ; \mathbb{R})$ (which always exits because $G = \langle \tau \rangle$ is finite) and let $e$ denote the image of the Euler class of $H^+(W)$ in $H^*_{\mathbb{Z}_p}$. To deduce the result from Theorem \ref{thm:don}, we just need to check that $e Q^j \neq 0$ for all $j \ge0$ if $p=2$ and $e S^j \neq 0$ for all $j \ge 0$ if $p$ is odd. 

In the case $p=2$, $e$ is the top Stiefel-Whitney class of $H^+(W)$, which is easily seen to be $Q^{b_+(W)}$ because our assumption that the subspace of $H^2(W ; \mathbb{R})$ fixed by $\tau$ is negative definite implies that $\tau$ acts as $-1$ on $H^+(W)$. Then clearly $eQ^j \neq 0$ for all $j \ge 0$.

Now suppose $p$ is odd. Let $L_i$ be the complex $1$-dimensional representation on which $\tau$ acts as multiplication by $\zeta^i$, $\zeta = e^{2\pi i/p}$. Any finite dimensional real representation of $G$ is the direct sum of a trivial representation and copies of the underlying real representations of the $L_i$ for $1 \le i \le p-1$. The hypothesis that the subspace of $H^2(W ; \mathbb{R})$ fixed by $\tau$ is negative definite means that as a representation of $G$, $H^+(W)$ contains no trivial summand. Hence $H^+(W)$ admits a complex structure such that $H^+(W) \cong \bigoplus_{i=1}^{p-1} L_i^{m_i}$ for some integers $m_i \ge 0$. The Euler class of $H^+(W)$ is equal to its top Chern class. Under the map $H^2(\mathbb{Z}_p ; \mathbb{Z}) \to H^2(\mathbb{Z}_p ; \mathbb{Z}_p)$ one finds that $c_1(L_i) $ gets sent to $i S$. Hence
\[
e = \prod_{i=1}^{p-1} (iS)^{m_i},
\]
from which it is clear that $e S^j \neq 0$ for all $j \ge 0$.
\end{proof}

\begin{remark}
Suppose that $p$ is odd. Then as in the proof of Theorem \ref{thm:don2}, $H^+(V)$ admits a complex structure. So if $p$ is odd and the assumptions of Theorem \ref{thm:don2} hold, then $b_+(W)$ must be even.
\end{remark}

To keep notation simple, we will henceforth set $b'_{\pm}(W) = b_{\pm}(W)$ if $p=2$ and $b'_{\pm}(W) = b_{\pm}(W)/2$ if $p$ is odd. Then (1) and (2) of Theorem \ref{thm:don2} can be written more uniformly as:
\[
\delta(W , \mathfrak{s}) \le \delta_{j}(Y , \mathfrak{s}|_Y) \; \; \text{ and } \delta_{j+b'_+(W)}(\overline{Y} , \mathfrak{s}|_Y ) \le \delta( \overline{W} , \mathfrak{s} )
\]
and
\[
\delta_{j+b'_+(W)}(Y_1 , \mathfrak{s}|_{Y_1}) + \delta(W , \mathfrak{s}) \le \delta_{j}(Y_2 , \mathfrak{s}|_{Y_2}).
\]

\begin{corollary}\label{cor:deltasigma}
Let $W$ be a smooth, compact, oriented $4$-manifold with boundary and with $b_1(W) = 0$. Suppose that $\tau \colon W \to W$ is an orientation preserving diffeomorphism of order $p$ and $\mathfrak{s}$ a spin$^c$-structure preserved by $\tau$. Suppose that $Y = \partial W$ is a rational homology $3$-sphere. Suppose that the subspace of $H^2(W ; \mathbb{R})$ fixed by $\tau$ is zero. Then
\begin{itemize}
\item[(1)]{$\delta_j(Y,\mathfrak{s}|_Y) \ge -\sigma(W)/8$ for all $j \ge 0$ and $\delta_j(Y,\mathfrak{s}|_Y) = -\sigma(W)/8$ for $j \ge b'_-(W)$.}
\item[(2)]{$\delta_j( \overline{Y},\mathfrak{s}|_Y) \ge \sigma(W)/8$ for all $j \ge 0$ and $\delta_j(\overline{Y},\mathfrak{s}|_Y) = \sigma(W)/8$ for $j \ge b'_+(W)$.}
\end{itemize}
\end{corollary}
\begin{proof}
It suffices to prove (1) since (2) follows by reversing orientation on $W$ and $Y$. Since $\mathbb{Z}_p$ preserves $\mathfrak{s}$, it follows that the image of $c_1(\mathfrak{s})$ in real cohomology lies in the subspace of $H^2(W ; \mathbb{R})$ fixed by $\mathbb{Z}_p$. By assumption this space is zero, hence $c_1(\mathfrak{s}) = 0$ in real cohomology and hence $c_1(\mathfrak{s})^2 = 0$. So $\delta(W , \mathfrak{s}) = -\sigma(W)/8$. Then from Theorem \ref{thm:don2} (1), we get $\delta_j(Y , \mathfrak{s}) \ge -\sigma(W)/8$ for all $j \ge 0$. Reversing orientation on $W$ and $Y$ an applying Theorem \ref{thm:don2} (1), we also get that $\delta_{j+b'_-(W)}(Y , \mathfrak{s}) \le -\sigma(W)/8$ for all $j \ge 0$ or equivalently, $\delta_j(Y ,\mathfrak{s}) \le -\sigma(W)/8$ for all $j \ge b'_-(W)$. Combining inequalities, we see that $\delta_j(Y , \mathfrak{s}) = -\sigma(W)/8$ for $j \ge b'_-(W)$.
\end{proof}

\subsection{Some algebraic results}

In this section we collect some algebraic results which will be useful for computing $\delta$ invariants. 

Let $Y$ be a rational homology $3$-sphere, $\tau \colon Y \to Y$ an orientation preserving diffeomorphism of prime order $p$ and $\mathfrak{s}$ a spin$^c$-structure preserved by $\tau$. Take $G = \mathbb{Z}_p = \langle \tau \rangle$ and $\mathbb{F} = \mathbb{Z}_p$. Let $\{ E_r^{p,q} , d_r \}_{r \ge 2}$ denote the spectral sequence relating equivariant and non-equivariant Seiberg--Witten--Floer cohomologies. So
\[
E_2^{p,q} = H^p( \mathbb{Z}_p , HSW^q(Y , \mathfrak{s}) )
\]
where $\mathbb{Z}_p$ acts on $HSW^q(Y , \mathfrak{s})$ via the action induced by $\tau$. To simplify notation we will write $H^q$ for $HSW^q(Y , \mathfrak{s})$ and $d$ for $d(Y , \mathfrak{s})$. So $E_2^{p,q} = H^p( \mathbb{Z}_p , H^q)$. For fixed $q$, $H^q$ is a finite dimensional representation of $\mathbb{Z}_p$ over $\mathbb{F}$. Moreover, for all sufficiently large $k$, we have
\begin{equation}\label{equ:largek}
H^{d+2k} = \mathbb{F}, \quad H^{d+2k+1} = 0.
\end{equation}

Recall that $H^*_G$ is isomorphic to $\mathbb{F}[Q]$ for $p=2$ and to $\mathbb{F}[R,S]/(R^2)$ for odd $p$. In the case $p=2$ we will set $S = Q^2$, so in all cases $S \in H^2_G$.

\begin{lemma}\label{lem:q}
If $V$ is a finite dimensional representation of $\mathbb{Z}_p$ over $\mathbb{F} = \mathbb{Z}_p$, then $S \colon H^i( \mathbb{Z}_p ; V) \to H^{i+1}( \mathbb{Z}_p ; V)$ is surjective for all $i \ge 0$ and an isomorphism for all $i \ge 1$. Furthermore, we have $dim_{\mathbb{F}}( H^i( \mathbb{Z}_p ; V )) \le dim_{\mathbb{F}}(V)$.
\end{lemma}
\begin{proof}
Since $\mathbb{Z}_p$ acts freely on $S^1$, it follows from \cite[page 114]{bro} that there is an element $\nu \in H^2(\mathbb{Z}_p ; \mathbb{Z})$ (independent of $V$) such that the cup product $\nu \colon H^i(\mathbb{Z}_p ; V) \to H^{i+2}(\mathbb{Z}_p ; V)$ is an isomorphism for $i>0$ and surjective for $i=0$. Since $V$ is a representation of $\mathbb{Z}_p$ over $\mathbb{F}$, the same statement holds if we replace $\nu$ by its image in $H^2(\mathbb{Z}_p ; \mathbb{F})$, which must have the form $a S$ for some $a \in\mathbb{F}$. Moreover $a \neq 0$ follows by considering the case that $V = \mathbb{Z}_p$ is the trivial representation. Hence the cup product $S \colon H^i(\mathbb{Z}_p ; V) \to H^{i+2}(\mathbb{Z}_p ; V)$ is an isomorphism for $i>0$ and surjective for $i=0$. We have by induction that $dim_{\mathbb{F}}( H^i(\mathbb{Z}_p ; V) ) \le dim_{\mathbb{F}}( H^0(\mathbb{Z}_p ; V) )$ if $i$ is even and $dim_{\mathbb{F}}( H^i(\mathbb{Z}_p ; V) ) \le dim_{\mathbb{F}}( H^1(\mathbb{Z}_p ; V) )$ if $i$ is odd. Then since $H^0(\mathbb{Z}_p ; V)$ and $H^1(\mathbb{Z}_p ; V)$ can both be expressed as certain subquotients of $V$, it follows that $dim_{\mathbb{F}}( H^i(\mathbb{Z}_p ; V) ) \le dim_{\mathbb{F}}(V)$ for all $i$.
\end{proof}

\begin{lemma}\label{lem:isoq}
For each $r \ge 2$, the map $S \colon E_r^{p,q} \to E_r^{p+2,q}$ is surjective for all $p \ge 0$ and an isomorphism for all $p \ge r-1$.
\end{lemma}
\begin{proof}
Recall that $E_2^{p,q} = H^p( \mathbb{Z}_p , H^q )$. Hence $S \colon E_2^{p,q} \to E_2^{p+2,q}$ is surjective for all $p$ and an isomorphism for all $p \ge 1$, by Lemma \ref{lem:q}. This proves the case $r=2$. Now we proceed by induction. Let $r > 2$ and suppose that $S \colon E_{r-1}^{p,q} \to E_{r-1}^{p+2,q}$ is surjective for all $p \ge 0$ and an isomorphism for all $p \ge r-1$. Let $x \in E_r^{p+2,q}$. Then $x = [y]$ for some $y \in E_{r-1}^{p+2,q}$ with $d_{r-1}(y) = 0$. By the inductive hypothesis $y = Sz$ for some $z \in E_{r-1}^{p,q}$. Then $S d_{r-1}(z) = d_{r-1}(Sz) = d_{r-1}(y) = 0$. That is, $S d_{r-1}(z) = 0$. However,  $d_{r-1}(z) \in E_{r-1}^{p+r-1,q+2-r}$ and $p+r-1 \ge r-2$, so $S \colon E_{r-1}^{p+r-1,q+2-r} \to E_{r-1}^{p+r+1,q+2-r}$ is an isomorphism by the inductive hypothesis. Hence $S d_{r-1}(z) = 0$ implies that $d_{r-1}(z) = 0$. So $z$ defines a class $w = [z] \in E_r^{p,q}$. Then $Sw = [Sz] = [y] = x$. Hence $S \colon E_r^{p,q} \to E_r^{p+2,q}$ is surjective for all $p \ge 0$. 

Now suppose that $p \ge r-1$ and consider $x \in E_r^{p,q}$ satisfying $Sx = 0$. Write $x = [y]$ for some $y \in E_{r-1}^{p,q}$ satisfying $d_{r-1}(y) = 0$. Then $0 = Sx = S[y] = [Sy]$. Hence $Sy = d_{r-1}(z)$ for some $z \in E_{r-1}^{p-r+3,q+r-2}$. By the inductive hypothesis and since $p-r+3 \ge (r-1)-r+3 = 2$, we have that $z = Sw$ for some $w \in E_{r-1}^{p-r+1,q+r-2}$. Hence $Sy = d_{r-1}(z) = d_{r-1}(Sw) = S d_{r-1}(w)$. By the inductive hypothesis $S \colon E_{r-1}^{p,q} \to E_{r-1}^{p+2,q}$ is injective, hence $y = d_{r-1}(w)$ and $x = [y] = [d_{r-1}(w)] = 0$. So $S \colon E_r^{p,q} \to E_r^{p+2,q}$ is injective for $p \ge r-1$.
\end{proof}

From the above lemma, we see that $E_r^{p,q}$ does not depend on $p$, provided $p \ge r-1$. Let $M_r^q$ denote $E_r^{p,q}$ for $p \ge r-1$. Moreover, since the differentials $\{ d_r\}$ for the spectral sequence $E_r^{p,q}$ commute with $S$, they induce differentials $d_r \colon M_r^q \to M_r^{q+1-r}$ for which $M_{r+1}$ is the cohomology of $d_r \colon M_r \to M_r$. Thus $M_{r+1}$ is a subquotient of $M_r$.

For any module $V$ over $\mathbb{F}[U]$, we define
\[
V_{red} = \{ x \in V \; | \; U^k x = 0 \text{ for some } k \ge 0 \} \; \text{ and } \; V^\infty = V/V_{red}.
\]

\begin{lemma}\label{lem:red}
For each $r \ge 2$, the image of the differential $d_r$ is contained in $(E_r^{*,*})_{red}$.
\end{lemma}
\begin{proof}
From (\ref{equ:largek}) we have that there exists a $k_0$ such that $H^{d+2k} = \mathbb{F}$ and $H^{d+2k+1} = 0$ for all $k \ge k_0$. Hence the action of $\tau$ is trivial in these degrees and   we have
\[
E_2^{p,d+2k} = \mathbb{F}, \quad E_2^{p,d+2k+1} = 0
\]
for all $k \ge k_0$. Since $SWF(Y,\mathfrak{s},g)$ is a space of type $\mathbb{Z}_p$-SWF, the localisation theorem in equivariant cohomology implies that there exists a $k_1 \ge k_0$ such that the generator $x \in E_2^{0,d+2k_1} = \mathbb{F}$ satisfies $d_r(x) = 0$ for all $r \ge 2$. Then if $y \in E_2^{p,q}$ with $q \ge d+2k_1$, it follows that $y$ is of the form $y = c U^a x$ for some $a \ge 0$, where $c \in H^p_G$. Hence $d_r(y) = 0$ for all $r \ge 2$. Now let $y \in E_r^{p,q}$ where $p,q$ are arbitrary. Then there exists some $a \ge 0$ such that $q+2a \ge d + 2k_1$, hence $U^a d_r(y) = d_r( U^a y ) = 0$. Therefore $d_r(y) \in (E_r^{*,*})_{red}$.
\end{proof}

Recall that $H^\infty$ is a free $\mathbb{F}[U]$-module of rank $1$ with generator in degree $d$. Hence we may write $H^\infty = \mathbb{F}[U]\theta$ where $deg(\theta) = d$. Next, observe that $E_2^{0,*}$ is the $\tau$-invariant part of $H^*$, hence may be regarded as an $\mathbb{F}[U]$-submodule of $H^*$. Similarly, since $E_{r+1}^{0,*}$ is the kernel of $d_r$ restricted to $E_r^{0,*}$, it follows that $E_{r+1}^{0,*}$ can be identified with an $\mathbb{F}[U]$-submodule of $E_{r}^{0,*}$. Hence $\{ E_r^{0,*} \}$ may be regarded as a decreasing sequence of $\mathbb{F}[U]$-submodules of $H^*$. Let $S_r$ denote the image of $E_r^{0,*}$ under the quotient map $H^* \to H^\infty = H^*/H_{red}$. The localisation theorem in equivariant cohomology implies that $S_r$ is non-zero and that the sequence $S_r$ eventually stabilises. Then since $S_r$ is a non-zero graded submodule of $H^\infty = \mathbb{F}[U]\theta$, it follows that $S_r = \mathbb{F}[U] U^{m_r}\theta$ for some non-negative integer $m_r$. Note also that the sequence $\{ m_r \}$ is increasing and is eventually constant.

\begin{lemma}\label{lem:mr}
For each $r \ge 2$ we have
\[
m_{r+1} - m_r \le dim_{\mathbb{F}}( (M_r)_{red} ) - dim_{\mathbb{F}}( (M_{r+1})_{red} ).
\]
\end{lemma}
\begin{proof}
The classes $U^{j+m_r} \theta$, $0 \le j < m_{r+1}-m_r$ form a basis for $S_r/S_{r+1}$. Choose a lift $x_r \in E_r^{0,d+2m_r}$ of $U^{m_r} \theta \in S_r$. Then $d_r( U^j x_r ) \neq 0$ for $0 \le j < m_{r+1}-m_r$, for if $d_r( U^j x_r) = 0$ for some $0 \le j < m_{r+1} - m_r$, then we would have $U^{j+m_r}\theta \in S_{r+1}$. Observe that $d_r( U^j x_r ) \in E_r^{r,*}$. By Lemma \ref{lem:isoq} and the definition of $M_r$, we see that $d_r(U^j x_r)$ can be identified with a non-zero element of $M_r$. Moreover, we have that $d_r( U^j x_r ) \in (M_r)_{red}$, by Lemma \ref{lem:red}. Now since the $d_r( U^j x_r)$ are non-zero and have distinct degrees, they span a subspace of $(M_r)_{red}$ of dimension $m_{r+1}-m_r$. Furthermore, this subspace lies in the image of $d_r$, hence $m_{r+1} - m_r \le dim_{\mathbb{F}}( (M_r)_{red} ) - dim_{\mathbb{F}}( (M_{r+1})_{red} )$.
\end{proof}

\begin{proposition}\label{prop:trivialtau}
Suppose that $\tau$ acts trivially on $HSW^*(Y , \mathfrak{s})$. Then
\[
\delta_1( Y , \mathfrak{s}) - \delta(Y , \mathfrak{s}) \le dim_{\mathbb{F}}( HSW_{red}(Y , \mathfrak{s}) ).
\]
\end{proposition}
\begin{proof}
Recall that $d = d(Y ,\mathfrak{s})$. Hence $\delta(Y , \mathfrak{s}) = d/2$. From the definition of the invariant $\delta_{1}(Y,\mathfrak{s})$, it follows that for all sufficiently large $r$, we have
\[
\delta_{1}(Y , \mathfrak{s} ) = m_r + \delta(Y , \mathfrak{s}).
\]
By Lemma \ref{lem:mr}, for each $r \ge 2$, we have
\[
m_{r+1} - m_r \le dim_{\mathbb{F}}( (M_r)_{red} ) - dim_{\mathbb{F}}( (M_{r+1})_{red})
\]
and summing from $2$ to $r-1$, we get
\[
m_r - m_2 \le dim_{\mathbb{F}}( (M_2)_{red} ).
\]
However since $\tau$ acts trivially on $HSW^*(Y , \mathfrak{s})$, we have that $E_2^{p,*} = HSW^*(Y , \mathfrak{s})$ for all $p \ge 0$. Hence $m_2 = 0$, $M_2 = HSW^*(Y , \mathfrak{s})$ and $(M_2)_{red} = HSW_{red}(Y , \mathfrak{s})$. Taking $r$ sufficiently large, we have
\[
\delta_1(Y , \mathfrak{s}) - \delta(Y , \mathfrak{s}) = m_r = m_r - m_2 \le dim_{\mathbb{F}}( HSW_{red} ).
\]
\end{proof}

\section{Branched double covers of knots}\label{sec:bdc}

\subsection{Concordance invariants}\label{sec:double}

Let $K \subset S^3$ be a knot in $S^3$. Let $Y = \Sigma_2(K)$ be the branched double cover of $S^3$, branched over $K$. Let $\pi \colon Y \to S^3$ denote the covering map. One finds that $b_1(Y) = 0$. Manolescu and Owens \cite{mo} define a knot invariant
\[
\delta(K) = 2d( \Sigma_2(K) , \mathfrak{t}_0) = 4 \delta( \Sigma_2(K) , \mathfrak{t}_0),
\]
where $\mathfrak{t}_0$ is the spin$^c$-structure induced from the unique spin structure on $\Sigma_2(K)$ (see \cite[Section 2]{mo} for an explanation of this). It is shown in \cite{mo} that $\delta(K)$ is always integer-valued and descends to a surjective group homomorphism $\delta \colon \mathcal{C} \to \mathbb{Z}$, where $\mathcal{C}$ is the smooth concordance group of knots in $S^3$.

The covering involution on $Y$ determines an action of $G = \mathbb{Z}_2$ on $Y$ preserving $\mathfrak{t}_0$ (by uniqueness of the underlying spin structure). Hence we may define the following knot invariants, for each $j\ge 0$:
\[
\delta_j(K) = 2d_j( \Sigma_2(K) , \mathfrak{t}_0) = 4 \delta_j( \Sigma_2(K) , \mathfrak{t}_0).
\]
Since $d_j( \Sigma_2(K) , \mathfrak{t}_0) - d( \Sigma_2(K) , \mathfrak{t}_0) \in 2 \mathbb{Z}$, it follows that $\delta_j(K) - \delta(K) \in 4 \mathbb{Z}$. Then, since $\delta(K)$ is integer-valued, it follows that the $\delta_j(K)$ are also integer-valued and moreover we have $\delta_j(K) = \delta(K) \; ({\rm mod} \; 4)$. 

\begin{proposition}\label{prop:conc}
For each $j \ge 0$, $\delta_j(K)$ depends only on the concordance class of $K$, hence $\delta_j$ descends to a concordance invariant $\delta_j \colon \mathcal{C} \to \mathbb{Z}$.
\end{proposition}
\begin{proof}
For an oriented knot $K$, recall that $-K$ denotes the knot obtained by reversing orientation on $S^3$ and $K$. It follows that $\Sigma_2(-K) = \overline{\Sigma_2(K)}$. A concordance of oriented knots $K_1,K_2$ is a smooth embedding of $\Sigma = [0,1] \times S^1$ in $[0,1] \times S^3$ having boundary $-K_1 \cup K_2$. Taking the double cover of $[0,1] \times S^3$ branched along $\Sigma$ gives a $\mathbb{Z}_2$-equivariant cobordism $W$ from $\Sigma_2(K_1)$ to $\Sigma(K_2)$. From the calculations in \cite[Section 3]{kata}, one sees that $W$ is a rational homology cobordism. We claim that $W$ is spin. To see this, choose a smoothly embedded surface $\Sigma$ in $D^4$ whose boundary is $K_1$. Let $W'$ be the double cover of $D^4 \cup [0,1] \times S^3 \cong D^4$ branched over $\Sigma \cup [0,1] \times S^1$. From \cite{kau} we see that $W'$ is spin. Since $W$ is embedded in $W'$, it follows that $W'$ is spin as well. Any spin structure $\mathfrak{t}$ on $W$ will restrict on each component of the boundary to the unique spin structure on the branched double cover $\Sigma_2(K_i)$. The result now follows by applying Corollary \ref{cor:inv} to $(W , \mathfrak{t})$.
\end{proof}

We note that the $\delta_j$ are {\em not} group homomorphisms.

Let $\sigma(K)$ denote the signature of $K$ and $g_4(K)$ the smooth $4$-genus. Set $\sigma'(K) = -\sigma(K)/2$. We also define $b_+(K) = g_4(K)-\sigma'(K)$, $b_-(K) = g_4(K) + \sigma'(K)$.

\begin{proposition}\label{prop:deltaknot}
The knot concordance invariants $\delta_j$ have the following properties:
\begin{itemize}
\item[(1)]{$\delta_0(K) \ge \delta(K)$.}
\item[(2)]{$\delta_{j+1}(K) \le \delta_j(K)$ for all $j \ge 0$.}
\item[(3)]{$\delta_j(K) \ge \sigma'(K)$ for all $j \ge 0$ and $\delta_j(K) = \sigma'(K)$ for $j \ge b_-(K)$.}
\item[(4)]{$\delta_j(-K) \ge -\sigma'(K)$ for all $j \ge 0$ and $\delta_j(-K) = -\sigma'(K)$ for $j \ge b_+(K)$.} 
\item[(5)]{If $\Sigma_2(K)$ is an $L$-space, then $\delta_j(K) = \delta(K)$ and $\delta_j(-K) = \delta(-K)$ for all $j \ge 0$.}
\end{itemize}

\end{proposition}
\begin{proof}
(1), (2) and (5) follow from (1), (2) and (5) of Theorem \ref{thm:deltaz2}. For (3) and (4), choose a smooth embedded surface $\Sigma \subset D^4$ in the $4$-ball of genus $g_4(K)$ which bounds $K$. Let $W$ be the double cover of $D^4$ branched along $\Sigma$. From \cite{kau} it follows that $W$ is spin. Let $\mathfrak{t}$ be any spin structure on $W$, then $\mathfrak{t}|_{\Sigma_2(K)} = \mathfrak{t}_0$ by uniqueness of $\mathfrak{t}_0$. Next, observe that $H^2(W ; \mathbb{R})^{\mathbb{Z}_2} = H^2(D^4 ; \mathbb{R}) = 0$. Then (3) and (4) follow by applying Corollary \ref{cor:deltasigma} to $(W , \mathfrak{t})$. 
\end{proof}

\begin{corollary}
If $K$ is a knot such that $\Sigma_2(K)$ is an $L$-space, then $\delta(K) = \sigma'(K)$.
\end{corollary}
\begin{proof}
This follows by (3) and (5) of Proposition \ref{prop:deltaknot}
\end{proof}

\begin{remark}
In particular, if $K$ is quasi-alternating, then $\Sigma_2(K)$ is an $L$-space \cite{os3}. This recovers the main result of \cite{liow} that $\delta(K) = \sigma'(K)$ for quasi-alternating knots.
\end{remark}

\begin{theorem}\label{thm:genus}
For a knot $K$, let $j_+(K)$ be the smallest positive integer such that $\delta_j(K) = \sigma'(K)$ and $j_-(K)$ the smallest positive integer such that $\delta_j(-K) = -\sigma'(K)$. Then 
\[
g_4(K) \ge \max\{ \sigma'(K)+j_-(K) , -\sigma'(K)+j_+(K)\}.
\]
\end{theorem}
\begin{remark}
Observe that the right hand side of this inequality is at least $|\sigma(K)|/2$. Hence we have obtained a strengthening of the well-known inequality $g_4(K) \ge | \sigma(K)|/2$ \cite{mur}.
\end{remark}
\begin{proof}
From Proposition \ref{prop:deltaknot} we have that $\delta_j(K) = \sigma'(K)$ for $j \ge g_4(K)+\sigma'(K)$ and $\delta_j(-K) = -\sigma'(K)$ for $j \ge g_4(K)-\sigma'(K)$. Hence $j_+(K) \le g_4(K)+\sigma'(K)$ and $j_-(K) \le g_4(K)-\sigma'(K)$.
\end{proof}

\begin{remark}\label{rem:higher}
In this section we have used branched double covers $\Sigma_2(K)$ of knots equipped with their natural $\mathbb{Z}_2$-action to obtain a sequence of concordance invariants. Similarly for any odd prime $p$ we may consider the cyclic branched cover $\Sigma_p(K)$ with its natural $\mathbb{Z}_p$-action. Once again there is a canonical spin$^c$-structure $\mathfrak{t}_0$ \cite{grs} and so we may define a sequence of invariants 
\[
\delta_{(p),j}(K) = 2 d_{\mathbb{Z}_p , S^j}( \Sigma_p(K) , \mathfrak{t}_0 ) 
\]
depending on a prime $p$ and an integer $j \ge 0$. By similar arguments to the $p=2$ case one finds that these are integer-valued knot concordance invariants of $K$.
\end{remark}

\section{Computations and applications}\label{sec:caa}

\subsection{Brieskorn homology spheres}\label{sec:brie}

Let $p,q,r$ be pairwise coprime positive integers and let $Y = \Sigma(p,q,r)$ be the corresponding Brieskorn integral homology $3$-sphere. Then $Y$ has a unique spin$^c$-structure and so when speaking of the Floer homology of $Y$ we omit the mention of the spin$^c$-structure.

Recall that $\Sigma(p,q,r)$ can be realised as the $p$-fold cyclic cover of $S^3$ branched along the torus knot $T_{q,r}$. In particular, this construction defines an action of $\mathbb{Z}_p$ on $Y$. Let $\tau : Y \to Y$ denote the generator of this action. Recall that $\Sigma(p,q,r)$ is obtained by taking the link of the singularity $\{ (x,y,z) \in \mathbb{C}^3 \; | \; x^p + y^q + z^r = 0 \}$. Then $\tau$ is given by $(x,y,z) \mapsto ( e^{2\pi i/p} x,y,z)$. This map is isotopic to the identity through the homotopy $(x,y,z) \mapsto ( e^{2\pi i qr t}x , e^{2\pi i pr t}y , e^{2\pi i pqt}z)$, $t \in [0, (qr)^*]$, where $0 < (qr)^* < p$ denotes the multiplicative inverse of $qr$ mod $p$. It follows that $\tau$ acts trivially on $HF^+(Y)$.

Henceforth we will assume that $p$ is a prime number. Set $\mathbb{F} = \mathbb{Z}_p$ and recall that $H_{\mathbb{Z}_p}^* \cong \mathbb{F}[Q]$ where $deg(Q) = 1$ if $p=2$ and $H_{\mathbb{Z}_p}^* \cong \mathbb{F}[R,S]/(R^2)$ where $deg(R)=1$, $deg(S)=2$ if $p$ is odd. Let $\mathfrak{s}$ denote the unique spin$^c$-structure on $Y$. As in Section \ref{sec:zp}, we let $\delta_j( Y , \mathfrak{s} ,\tau , p)$ denote $\delta_{\mathbb{Z}_p , Q^p}(Y , \mathfrak{s})$ for $p=2$ or $\delta_{\mathbb{Z}_p , S^p}(Y,\mathfrak{s})$ for odd $p$. We will further abbreviate this to $\delta_j(Y , \tau)$. When $p=2$, $\delta_j( Y , \tau ) = \delta_j ( T_{q,r})/4$, where $\delta_j( K )$ denotes the knot concordance invariant introduced in Section \ref{sec:double}. More generally, $\delta_j(Y , \tau) = \delta_{(p) , j}( T_{q,r} )/4$, where $\delta_{(p) , j}( K )$ is the knot concordance invariant defined in Remark \ref{rem:higher}.

\begin{example}\label{ex:235}
$(p,q,r) = (2,3,5)$. Then $Y = \Sigma(2,3,5)$ is the Poincar\'e homology $3$-sphere. Since $\Sigma(2,3,5)$ has spherical geometry it is an $L$-space \cite[Proposition 2.3]{os}. Therefore
\[
\delta_j( T_{3,5} ) = \delta( T_{3,5} ) = \sigma'(T_{3,5}) = 4, \text{ for all } j \ge 0.
\]
The property of being an $L$-space does not depend on the choice of orientation, so we also have
\[
\delta_j( -T_{3,5} ) = \delta( -T_{3,5} ) = -4, \text{ for all } j \ge 0.
\]
The same argument applied to $p=3$ or $5$ gives
\[
\delta_{(3),j}( T_{2,5} ) = \delta_{(5),j}( T_{2,3}) = 4, \text{ for all } j \ge 0
\]
and
\[
\delta_{(3),j}( -T_{2,5} ) = \delta_{(5),j}( -T_{2,3}) = -4, \text{ for all } j \ge 0.
\]
\end{example}

\begin{proposition}\label{prop:casson}
Let $p,q,r$ be positive, pairwise coprime integers and assume that $p$ is prime. Then $\delta_j( \Sigma(p,q,r) , \tau ) = -\lambda( \Sigma(p,q,r))$ for all $j \ge 0$, where $\lambda( \Sigma(p,q,r))$ is the Casson invariant of $\Sigma(p,q,r)$. Furthermore, we have that
\[
\delta_{(p) , j}( T_{q,r} ) = -\frac{1}{2} \sum_{j=1}^{p-1} \sigma_{j/p}( T_{q,r}), \text{ for all } j \ge 0
\]
where $\sigma_{\alpha}(K)$ is the Tristram--Levine signature of $K$.
\end{proposition}
\begin{proof}
Recall that $Y = \Sigma(p,q,r)$ is the boundary of a negative definite plumbing \cite{nr} whose plumbing graph has only one bad vertex in the terminology of \cite{os2}. Then it follows from \cite[Corollary 1.4]{os2} that $HF^+( \overline{Y})$ is concentrated in even degrees. Consequently $HF^+_{red}( Y)$ is concentrated in odd degrees. (Note that \cite{os2} uses $\mathbb{Z}$ coefficients. But it is shown there that $HF^+_{red}(Y ; \mathbb{Z})$ has no torsion and hence by the universal coefficient theorem, \cite[Corollary 1.4]{os2} also holds for $\mathbb{Z}_p$ coefficients). Therefore
\begin{equation}\label{equ:chi1}
\chi( HF^+_{red}( Y )) = dim_{\mathbb{F}}( HF^+_{red,even}(Y)) - dim_{\mathbb{F}}( HF^+_{red,odd}(Y)) = -dim_{\mathbb{F}}( HF^+_{red}(Y)).
\end{equation}
From \cite[Theorem 1.3]{os4} we have that $\chi( HF^+_{red}(Y))$ is related to the Casson invariant $\lambda( Y)$ via the formula
\begin{equation}\label{equ:chi2}
\chi( HF^+_{red}( Y )) = \lambda( Y ) + \delta( Y).
\end{equation}
Hence
\begin{equation}\label{equ:hfred}
dim_{\mathbb{F}}( HF^+_{red}(Y)) = -\lambda(Y) - \delta(Y).
\end{equation}

Moreover, from \cite{fs}, \cite{cs}, we have that
\[
\lambda(\Sigma(p,q,r)) = \frac{1}{8}\sum_{j=1}^{p-1} \sigma_{j/p}( T_{q,r}) = \frac{1}{8} \sigma( M(p,q,r))
\]
where $M(p,q,r)$ is the Milnor fibre 
\[
M(p,q,r) = \{ (x,y,z) \in \mathbb{C}^3 \; | \; x^p + y^q + z^r = \delta \} \cap D^6
\]
(where $\delta$ is a sufficiently small non-zero complex number). Recall that $M(p,q,r)$ is a compact smooth $4$-manifold with boundary diffeomorphic to $\Sigma(p,q,r)$. Moreover, $M(p,q,r)$ has the homotopy type of a wedge of $2$-spheres, so $b_1( M(p,q,r)) =0$. Further, $M(p,q,r)$ is a $p$-fold cyclic cover of $D^4$ branched along a surface bounding $T_{q,r}$. Hence the action of $\mathbb{Z}_p = \langle \tau \rangle$ on $Y$ extends to $M(p,q,r)$. From \cite[Lemma 2.1]{grs} it follows that there is a $\mathbb{Z}_p$ invariant spin structure $\mathfrak{t}_0$ on $M(p,q,r)$. Since $M(p,q,r)$ is a cyclic $p$-fold cover of $D^4$ it follows that the subspace of $H^2( M(p,q,r) ; \mathbb{R})$ fixed by $\tau$ is zero. Hence Corollary \ref{cor:deltasigma} may be applied, giving
\[
\delta_j( Y , \tau ) \ge -\frac{\sigma( M(p,q,r) )}{8} = -\lambda(Y), \text{ for all } j \ge 0.
\]

Since $\tau$ acts trivially on $HF^+(Y)$, Proposition \ref{prop:trivialtau} implies that
\[
\delta_0(Y , \tau) - \delta(Y) \le dim_{\mathbb{F}}( HF^+_{red}(Y) ) = -\lambda(Y) - \delta(Y).
\]
Hence $\delta_0(Y , \tau) \le -\lambda(Y)$. On the other hand, $\delta_0(Y , \tau) \ge \delta_j(Y , \tau) \ge -\lambda(Y)$ for any $j \ge 0$. Hence $\delta_j(Y,\tau) = -\lambda(Y)$ for all $j \ge 0$. Therefore we also have
\[
\delta_{(p) , j}( T_{q,r}) = 4 \delta_j(Y , \tau) = -4 \lambda(Y) = -\frac{1}{2} \sum_{j=1}^{p-1} \sigma_{j/p}( T_{q,r})
\]
for all $j \ge 0$.
\end{proof}

The above result shows that the values of $\delta_{(p), j}( T_{q,r})$ do not depend on $j$. In contrast, the values of $\delta_{(p),j}( -T_{q,r})$ usually do depend on $j$, as the following propositions illustrate.

\begin{proposition}\label{ex:36n-1}
Let $(a,b) = (3,6n-1)$ for $n \ge 1$. Then
\[
\delta( -T_{3,6n-1} ) = -4, \quad \sigma'( -T_{3,6n-1} ) = -4n
\]
and
\[
\delta_j( -T_{3,6n-1}  ) = \begin{cases} -4\left( \left\lfloor \frac{j}{2} \right\rfloor + 1\right) & 0 \le j \le 2n-3, \\ -4n & j \ge 2n-2. \end{cases}
\]
\end{proposition}
\begin{proof}The case $n=1$ is already covered in Example \ref{ex:235}, so we assume $n \ge 2$. Set $Y_{a,b} = \Sigma_2( T_{a,b} ) = \Sigma(2,a,b)$ and let $\tau$ be the covering involution. So $\delta_j( -T_{3,6n-1} ) = 4 \delta_j( \overline{Y_{3,6n-1}})$. From the computations in \cite[Section 8]{os4} we find that $d(\overline{Y_{3,6n-1}}) = -2$, $SWF^*_{red}(\overline{Y_{3,6n-1}}) = (\mathbb{F}_{-2})^{n-1}$, where the subscript indicates degree. To simplify notation we let $V = SWF^*_{red}(\overline{Y_{3,6n-1}}) = (\mathbb{F}_{-2})^{n-1}$.
Then 
\[
E_2^{*,*} \cong \mathbb{F}[U,Q] \theta \oplus V[Q]
\]
where the bi-degree is given as follows. $\theta$ and all elements of $V$ have bi-degree $(0,-2)$, $U$ has bidegree $(0,2)$ and $Q$ has bidegree $(1,0)$. Then $E_2^{p,q} = 0$ for $q < -2$. It follows that all the differentials in the spectral sequence are zero on $\theta$ and on $V$, since $d_r$ sends $E_r^{p,-2}$ to $E_r^{p+r, -1-r}$ and $-1-r < -2$ for $r \ge 2$. Hence $d_r$ is zero on all of $E_r$ and $E_\infty^{*,*} \cong E_2^{*,*}$. Let $\mathcal{F}_j$ denote the filtration on $HSW^*_{\mathbb{Z}_2}( \overline{Y_{3,6n-1}})$ corresponding to the spectral sequence, so $\mathcal{F}_j / \mathcal{F}_{j+1} \cong E_\infty^{j,*}$. In particular, $\mathcal{F}_1/\mathcal{F}_2 \cong \mathbb{F}[U]\theta \oplus V$. Choose lifts of $\theta$ and $V$ to $\mathcal{F}_1$. We lift $U^j \theta$ by taking the lift of $\theta$ and applying $U^j$. Hence we obtain a short exact sequence of $\mathbb{F}[U,Q]$-modules
\[
0 \to \mathbb{F}[U] \oplus V \to HSW^*_{\mathbb{Z}_2}( \overline{ Y_{3,6n-1} } ) \to \mathcal{F}_2 \to 0.
\]
Next, for each $j \ge 0$, $Q$ induces an isomorphism $Q \colon \mathcal{F}_j/\mathcal{F}_{j+1} \to \mathcal{F}_{j+1}/\mathcal{F}_{j+2}$ hence by applying $Q$ repeatedly to $\mathbb{F}[U]\theta \oplus V$, we obtain a splitting of the filtration $\{ \mathcal{F}_j \}$ as $\mathbb{F}[Q]$-modules. The splittings give an isomorphism of $\mathbb{F}[Q]$-modules
\[
HSW^*_{\mathbb{Z}_2}( \overline{ Y_{3,6n-1} } ) \cong \mathbb{F}[U,Q] \theta \oplus V[Q].
\]
However, this is not necessarily an isomorphism of $\mathbb{F}[U,Q]$-modules. Under this isomorphism, $U$ corresponds to an endomorphism of the form
\[
\hat{U} = U_{2} + QU_{1} + Q^2 U_{0} + Q^3 U_{-1} + \cdots 
\]
where $U_j \colon HSW^*( \overline{ Y_{3,6n-1} } ) \to HSW^{*+j}( \overline{ Y_{3,6n-1} } )$ and $U_2 = U$. Since $HSW^*( \overline{ Y_{3,6n-1} } )$ is concentrated in even degrees we have that $U_j = 0$ for odd $j$. Moreover, our construction is such that $U_{j}\theta = 0$ for $j \neq 2$. It follows that $U_j = 0$ for $j < 0$, as $V$ is concentrated in a single degree. So we get
\[
\hat{U} = U + Q^2 U_0
\]
for some $U_0 \colon V \to HSW^0( \overline{ Y_{3,6n-1} } )$.

To simplify notation set $d_j = d_j( \overline{ Y_{3,6n-1} })$. Using Proposition \ref{prop:altd} we obtain the following characterisation of $d_j$.
\[
d_j = \min\{ i \; | \; \exists x \in HSW^i_{\mathbb{Z}_2}( \overline{ Y_{3,6n-1} } ), \; \hat{U}^r x = U^m Q^j \theta \; ({\rm mod} \; Q^{j+1} ) \text{ for some } r,m \ge 0 \} - j.
\]
Recall that $\delta_j( -T_{3,6n-1} ) = \sigma'( -T_{3,6n-1} ) = -4n$ for sufficiently large $j$. Hence $d_j = -2n$ for sufficiently large $j$. Choose such a $j$. From the above characterisation of $j$ there exists $x \in HSW^{j-2n}_{\mathbb{Z}_2}(\overline{ Y_{3,6n-1} })$ such that $\hat{U}^r x = U^m Q^j \theta + \cdots$ where $\cdots$ denotes terms of higher order in $Q$. We have that $x = Q^a y$ for some $a \le j$. Then $\hat{U}^r Q^a y = U^m Q^j \theta + \cdots$. Since $Q$ is injective we may cancel, giving $\hat{U}^r y = U^m Q^{j-a} \theta + \cdots$. If $a = j$, then $\hat{U}^r y = U^m \theta + \cdots$. But $\hat{U} = U + Q^2U_0$, so $\hat{U}^r y = U^r y + \cdots$, hence $U^r y = U^m \theta + \cdots$. From the definition of the usual $d$-invariant we must have $deg(y) \ge d( \overline{Y_{3,6n-1} } ) = -2$. Hence $j-2n = deg(x) = a + deg(y) = j-2$, which is a contradiction since we have assumed that $n > 1$. It follows that $a < j$. We must have $y \in V$ for if $y = U^b \theta \; ({\rm mod} \; V)$, then we would have $\hat{U}^r x = U^{r+b} Q^a \theta + \cdots$, which contradicts $\hat{U}^r x = U^m Q^j \theta + \cdots$ as $a < j$. Therefore $y \in V$. In particular $deg(y) = -2$ and $j-2n = deg(x) = a-2$. 

Let $b$ be the smallest positive integer such that $U_0^b y \notin V$. Such a $b$ exists since $\hat{U}^r y = (U + Q^2 U_0)^r y = U^m Q^{j-a} \theta + \cdots$ and $U$ is zero on $V$. Then it follows that $r \ge b$ and $\hat{U}^r y = (U + Q^2 U_0)^r y = U^{r-b} Q^{2b} (U_0^b y) + \cdots = U^m Q^{j-a} \theta + \cdots$. Hence $2b = j-a$. So we have shown that $j = a+2b$ and $j-2n = a-2$. Hence $b = n-1$. But since $\dim_{\mathbb{F}}(V) = n-1 = b$, it follows that there exists a $v \in V$ such that $v, U_0v, U_0^2 v , \dots , U_0^{n-2}v$ is a basis for $V$ and $U_0^{n-1}v = \theta \; ({\rm mod} \; V)$. Now it is straightforward to see that the sequence $\{d_j\}$ must have the form $-2, -2, -4, -4, -6, -6, \dots , $ for $j \le 2n-3$ and $d_j = -2n$ for $j \ge 2n-2$.
\end{proof}

\begin{proposition}\label{ex:36n+1} Let
$(a,b) = (3,6n+1)$ for $n \ge 1$. Then
\[
\delta( -T_{3,6n+1} ) = 0, \quad \sigma'( -T_{3,6n+1} ) = -4n
\]
and
\[
\delta_j( -T_{3,6n+1}  ) = \begin{cases} -4 \left\lfloor \frac{j}{2} \right\rfloor & 0 \le j \le 2n-1, \\ -4n & j \ge 2n. \end{cases}
\]
\end{proposition}

\begin{proof} From \cite[Section 8]{os4} we find that $d(\overline{Y_{3,6n+1}}) = 0$, $SWF^*_{red}(\overline{Y_{3,6n+1}}) = (\mathbb{F}_{0})^{n}$ and $\sigma'( -T_{3,6n+1}) = 4n$. From here essentially the same argument as in Proposition \ref{ex:36n-1} gives the result.
\end{proof}

\begin{remark}
We can use Theorem \ref{thm:genus} and the computations in Propositions \ref{ex:36n-1}, \ref{ex:36n+1} to obtain a lower bound for the $4$-genus.
From Proposition \ref{ex:36n-1}, we see that $\sigma'(T_{3,6n-1}) = 4n$ and $j_-(T_{3,6n-1}) = 2n-2$, hence $g_4(T_{3,6n-1}) \ge 2n-2 + 4n = 6n-2$. On the other hand from the positive solution to the Milnor conjecture \cite{km2}, we know that $g_4( T_{a,b} ) = (a-1)(b-1)/2$. In particular $g_4( T_{3,6n-1} ) = 6n-2$. Hence the above estimate for $g_4( T_{3,6n-1} )$ is actually sharp.

Similarly, from Proposition \ref{ex:36n+1}, we see that $\sigma'(T_{3,6n+1}) = 4n$ and $j_-(T_{3,6n+1}) = 2n$. So we obtain an estimate $g_4( T_{3,6n+1} ) \ge 6n$. Once again, this estimate is sharp since $g_4( T_{3,6n+1}) = (3-1)(6n+1-1)/2 = 6n$.

\end{remark}

\subsection{Non-extendable actions}\label{sec:app1}

Let $Y$ be a rational homology $3$-sphere equipped with an orientation preserving action of $G$. Let $W$ be a smooth $4$-manifold with boundary $Y$. In this section we are concerned with the question of whether the $G$-action can be extended to $W$. In particular we are interested in finding obstructions to such an extension.

\begin{proposition}\label{prop:zero}
Let $Y$ be an integral homology $3$-sphere and $\mathfrak{s}$ the unique spin$^c$-structure on $Y$. Let $G$ act orientation preservingly on $Y$ and suppose that the extension $G_{\mathfrak{s}}$ is trivial. Suppose that $Y$ is the boundary of a contractible $4$-manifold $W$. If the action of $G$ extends over $W$ then $\delta_{G,c}(Y,\mathfrak{s}) = \delta_{G,c}(\overline{Y} , \mathfrak{s}) = 0$ for every non-zero $c \in H^*_G$.
\end{proposition}
\begin{proof}
Suppose that the $G$-action extends to $W$. Since $W$ is contractible, there is a unique spin$^c$-structure $\mathfrak{t}$ on $W$. By uniqueness it is $G$-invariant and $\mathfrak{t}|_Y = \mathfrak{s}$. Theorem \ref{thm:don} gives $\delta_{G,c}(Y,\mathfrak{s}) \le 0$ and $\delta_{G,c}(Y,\mathfrak{s}) \ge 0$, hence $\delta_{G,c}(Y,\mathfrak{s}) = 0$. Reversing orientations, we also find that $\delta_{G,c}(\overline{Y}\mathfrak{s})=0$.
\end{proof}

\begin{example}\label{ex:nonex}
Akbulut--Kirby constructed examples of contractible $4$-manifolds bounding integral homology spheres, in particular $\Sigma(2,5,7), \Sigma(3,4,5)$ and $\Sigma(2,3,13)$ bound contractible $4$-manifolds \cite[Theorem 2]{ak}. Further examples were given by Casson--Harer, in particular $\Sigma(2,2s-1,2s+1)$ for odd $s$ bounds a contractible $4$-manifold \cite{ch}.

Now let $Y = \Sigma(2,3,13)$ and let $\tau$ be the involution obtained by viewing $Y$ as the branched double cover $\Sigma_2( T_{3,13})$. Then $\delta_2(\overline{Y}) = -1$ by Proposition \ref{ex:36n+1}. Then it follows from Proposition \ref{prop:zero} that $\tau$ does not extend to an involution on any contractible $4$-manifold $W$ bounded by $Y$. On the other hand, since $\tau$ is isotopic to the identity, $\tau$ does extend to a diffeomorphism on $W$.

Similarly if we let $Y = \Sigma(2,2s-1,2s+1) = \Sigma_2( T_{2s-1,2s+1})$ where $s$ is odd and let $\tau$ be the covering involution, then $Y$ bounds a contractible $4$-manifold $W$ but $\tau$ does not extend to an involution on $W$ because $\delta_j(Y) = -\sigma( T_{2s-1,2s+1})/8 = (s^2-1)/2 \neq 0$, for all $j \ge 0$.

More generally, let $Y = \Sigma(p,q,r)$ where $p,q,r$ are pairwise coprime positive integers. Assume that $p$ is prime and let $\mathbb{Z}_p = \langle \tau \rangle$ act on $Y$ by realising $Y$ as the $p$-fold cyclic branched cover $\Sigma_p( T_{q,r})$. Then $\delta_{0}(Y , \tau) = -\lambda( \Sigma(p,q,r) )$. From \cite[Chapter 19]{sav}, it can be seen that $\lambda( \Sigma(p,q,r) ) < 0$ and hence the $\mathbb{Z}_p$-action on $Y = \Sigma(p,q,r)$ is non-extendable over contractible $4$-manifolds bounded by $Y$. We have thus recovered a special case of the non-extendability results of Anvari--Hambleton \cite{ah1}, \cite{ah2}.
\end{example}

If we relax the condition that $W$ is contractible to being a rational homology $4$-ball, then we get a similar result, except that we have to make an assumption on the order of $H^2(W ; \mathbb{Z})$.

\begin{proposition}\label{prop:zero2}
Let $Y$ be an integral homology $3$-sphere and $\mathfrak{s}$ the unique spin$^c$-structure on $Y$. Let $G = \mathbb{Z}_p$ for a prime $p$ act orientation preservingly on $Y$ and suppose that the extension $G_{\mathfrak{s}}$ is trivial. Suppose that $Y$ is the boundary of a compact, oriented, smooth rational homology $4$-ball $W$ and assume that $p$ does not divide the order of $H^2(W ; \mathbb{Z})$. If the action of $G$ extends over $W$ then $\delta_{G,c}(Y,\mathfrak{s}) = \delta_{G,c}(\overline{Y} , \mathfrak{s}) = 0$ for every non-zero $c \in H^*_G$.
\end{proposition}
\begin{proof}
The set of spin$^c$-structures on $W$ has cardinality $| H^2(W ; \mathbb{Z}) |$ and $G = \mathbb{Z}_p$ acts on this set. By assumption, $p$ does not divide this number and hence there must exist a spin$^c$-structure $\mathfrak{t}$ whose stabiliser group is not trivial. Since $p$ is prime, this means $\mathfrak{t}$ is fixed by all of $G$. From here, the rest of the proof is the same as for Proposition \ref{prop:zero}.
\end{proof}

\begin{example}\label{ex:nonex2}
Let $Y = \Sigma(p,q,r)$ where $p,q,r$ are relatively prime and assume that $p$ is prime. Let $\mathbb{Z}_p$ act on $Y$ as described in Section \ref{sec:brie}. Recall from Proposition \ref{prop:casson} that $\delta_{\mathbb{Z}_p , 1}(Y , \mathfrak{s}) = -\lambda( \Sigma(p,q,r))$. As in Example \ref{ex:nonex}, $\lambda( \Sigma(p,q,r) ) < 0$ and hence $\delta_{\mathbb{Z}_p , 1}(Y , \mathfrak{s}) > 0$.

Therefore by Proposition \ref{prop:zero2}, if $W$ is a compact, oriented, smooth rational homology $4$-ball bounded by $Y$ and if $p$ does not divide the order of $H^2(W ; \mathbb{Z})$, then the action of $G$ does not extend over $W$. Thus we have obtained a partial extension of the results of Anvari--Hambleton to the case of rational homology $4$-balls.

Fintushel--Stern showed that $\Sigma(2,3,7)$ bounds a rational homology $4$-ball, although it does not bound an integral $4$-ball \cite{fs0}. Akbulut--Larson showed that $\Sigma(2,4n+1,12n+5)$ and $\Sigma(3,3n+1,12n+5)$ for $n$ odd bound rational $4$-balls but not integral $4$-balls \cite{al}. More examples, $\Sigma(2,4n+3,12n+7)$ and $\Sigma(3,3n+2,12n+7)$ for even $n$ were constructed by \c{S}avk \cite{savk}. Taking $p=2$ or $3$, the above Brieskorn spheres admit $\mathbb{Z}_p$-actions with non-zero delta invariants, as in Example \ref{ex:nonex}. Hence the $\mathbb{Z}_p$-action does not extend to any oriented rational homology $4$-ball $W$ with boundary $Y$, provided the order of $H^2(W ; \mathbb{Z})$ is coprime to $p$. However, it does not seem straightforward to determine whether the above examples are bounded by rational $4$-balls satisfying this coprimality condition.
\end{example}

\begin{proposition}\label{prop:trivialaction}
Let $Y$ be an integral homology $3$-sphere and $\mathfrak{s}$ the unique spin$^c$-structure on $Y$. Let $G$ act orientation preservingly on $Y$ and suppose that the extension $G_{\mathfrak{s}}$ is trivial. Suppose that $Y$ is the boundary of a smooth, compact, oriented $4$-manifold with $b_1(W) = 0$ and suppose that $H^2(W ; \mathbb{Z})$ has no $2$-torsion. 

\begin{itemize}
\item[(1)]{If $H^2(W ; \mathbb{R})$ is positive definite and $\delta_{G,1}(Y , \mathfrak{s}) > 0$, then the $G$-action on $Y$ can not be extended to a smooth $G$-action on $W$ acting trivially on $H^2(W ; \mathbb{Z})$.}
\item[(2)]{If $H^2(W ; \mathbb{R})$ is negative definite and $\delta_{G,c}(Y , \mathfrak{s}) < 0$ for some $c \in H^*_G$, then the $G$-action on $Y$ can not be extended to a smooth $G$-action on $W$ acting trivially on $H^2(W ; \mathbb{Z})$.}
\end{itemize}

\end{proposition}
\begin{proof}
Suppose the $G$-action on $Y$ extends to a smooth $G$-action on $W$ acting trivially on $H^2(W ; \mathbb{Z})$. Since $H^2(W ; \mathbb{Z})$ has no $2$-torsion, a spin$^c$-structure $\mathfrak{t}$ on $W$ is determined uniquely by $c_1(\mathfrak{t})$. Since $G$ acts trivially on $H^2(W ; \mathbb{Z})$, it follows that $G$ preserves every spin$^c$-structure. Furthermore $\mathfrak{t}|_Y = \mathfrak{s}$ for any spin$^c$-structure on $W$ by uniqueness of $\mathfrak{t}$.

If $H^2(W ; \mathbb{R})$ is negative definite, then Theorem \ref{thm:don} may be applied to any spin$^c$-structure $\mathfrak{t}$ on $W$, giving
\[
\delta( W , \mathfrak{t} ) \le \delta_{G,c}(Y , \mathfrak{s})
\]
for all $\mathfrak{t}$ and all $c \in H^*_G$. Since $Y$ is an integral homology sphere, the intersection form on the $H^2(W ; \mathbb{Z})/{torsion}$ is unimodular. By the main theorem of \cite{elk}, there exists a spin$^c$-structure $\mathfrak{t}$ such that $\delta(W , \mathfrak{t}) \ge 0$. Hence $\delta_{G,c}(Y , \mathfrak{s}) \ge 0$. The proof in the case that $H^2(W ; \mathbb{R})$ is positive definite is similarly obtained.
\end{proof}

\begin{example}
Consider again $Y = \Sigma(p,q,r)$ with the same $\mathbb{Z}_p$-action. Recall from Proposition \ref{prop:casson} that $\delta_{\mathbb{Z}_p , 1}(Y , \mathfrak{s}) = -\lambda( \Sigma(p,q,r))$. As in Example \ref{ex:nonex2}, $\delta_{\mathbb{Z}_p , 1}(Y , \mathfrak{s}) > 0$. So by Proposition \ref{prop:trivialaction}, the action of $\mathbb{Z}_p$ on $Y$ can not be extended to any smooth, compact, oriented $4$-manifold $W$ such that $b_1(W) = 0$, $H^2(W ; \mathbb{Z})$ has no $2$-torsion and with $\mathbb{Z}_p$ acting trivially on $H^2(W ; \mathbb{Z})$.
\end{example}

\subsection{Realisation problems}\label{sec:app2}

In this section we are concerned with the following realisation problem. Let $W$ be a smooth $4$-manifold with boundary an integral homology sphere $Y$. Suppose that a finite group $G$ acts on $H^2(W ; \mathbb{Z})$ preserving the intersection form. We say that the action of $G$ on $H^2(W ; \mathbb{Z})$ can be realised by diffeomorphisms if there is a smooth orientation preserving action of $G$ on $W$ inducing the given action on $H^2(W ; \mathbb{Z})$. 

For simplicity we will assume that $G = \mathbb{Z}_p$ for a prime $p$ so that all extensions $G_{\mathfrak{s}}$ are trivial.

\begin{proposition}\label{prop:realise}
Let $W$ be a smooth, compact, oriented $4$-manifold with $b_1(W) = 0$ and with boundary $Y = \partial W$ an $L$-space integral homology sphere. Suppose that an action of $G = \mathbb{Z}_p$ on $H^2(W ; \mathbb{Z})$ is given and suppose that $H^2(W ; \mathbb{Z})$ has no $2$-torsion. Suppose that the subspace of $H^2(W ; \mathbb{R})$ fixed by $G$ is negative definite. If the action of $G$ on $H^2(W ; \mathbb{Z})$ can be realised by diffeomorphisms, then
\[
\delta(W , \mathfrak{s}) \le \delta( Y , \mathfrak{s}|_Y )
\]
for every spin$^c$-structure $\mathfrak{s}$ on $W$ for which $c_1(\mathfrak{s})$ is invariant.
\end{proposition}
\begin{proof}
This is essentially a special case of Theorem \ref{thm:don2}. Note that since $H^2(W ; \mathbb{Z})$ is assumed to have no $2$-torsion, any spin$^c$-structure $\mathfrak{s}$ for which $c_1(\mathfrak{s})$ is invariant is preserved by $G$. So if $G$ is realisable by diffeomorphisms, then Theorem \ref{thm:don2} gives $\delta(W , \mathfrak{s}) \le \delta_{G,1}( Y , \mathfrak{s}|_Y )$. But we have assumed that $Y$ is an $L$-space, so $\delta_{G,1}( Y , \mathfrak{s}|_Y ) = \delta( Y , \mathfrak{s}|_Y )$.
\end{proof}

\begin{example}\label{ex:nonr}
We consider a specialisation of Proposition \ref{prop:realise} as follows. Take $G = \mathbb{Z}_p$. Assume $Y$ is an $L$-space integral homology $3$-sphere and let $\mathfrak{s}$ be the unique spin$^c$-structure. Suppose that $W$ is a smooth, compact, oriented $4$-manifold with $b_1(W) = 0$ and with boundary $Y$. Suppose that the intersection form on  $H^2(W ; \mathbb{Z})$ is even and that $H^2(W ; \mathbb{Z})$ has no $2$-torsion. Then $W$ is spin and it has a unique spin structure $\mathfrak{t}$. By uniqueness, the restriction of $\mathfrak{t}$ to the boundary equals $\mathfrak{s}$. Suppose that an action of $G = \mathbb{Z}_p$ on $H^2(W ; \mathbb{Z})$ is given and that the subspace of $H^2(W ; \mathbb{R})$ fixed by $G$ is negative definite. Then applying Proposition \ref{prop:realise} to $(W , \mathfrak{t})$, we find that $\delta(W , \mathfrak{t}) = -\sigma(W)/8 \le \delta(Y , \mathfrak{s})$. Therefore, if $\sigma(W)/8 < -\delta(Y,\mathfrak{s})$ then the action of $\mathbb{Z}_p$ on $H^2(W ; \mathbb{Z})$ is not realisable by a smooth $\mathbb{Z}_p$-action on $W$.

For example if $W = K3 \# W_0$ is the connected sum of a $K3$ surface with $W_0$, the negative definite plumbing of the $E_8$ graph, then $\partial W = Y = \Sigma(2,3,5)$ is the Poincar\'e homology $3$-sphere which is an $L$-space. Then $W$ satisfies all the above conditions and $\sigma(W)/8 = -3 < \delta(Y,\mathfrak{s}) = -1$. Hence for any prime $p$, any $\mathbb{Z}_p$-action on $H^2(W ; \mathbb{Z})$ such that the invariant subspace of $H^2(W ; \mathbb{R})$ is negative definite can not be realised by a smooth $\mathbb{Z}_p$-action on $W$.
\end{example}

\begin{corollary}\label{cor:realise}
Let $W$ be a smooth, compact, oriented $4$-manifold with $b_1(W) = 0$ and with boundary $Y = \partial W$ an $L$-space integral homology sphere. Suppose that $W$ is spin and that $H^2(W ; \mathbb{Z})$ has no $2$-torsion. If there is a smooth involution on $W$ which acts as $-1$ on $H^2(W ; \mathbb{R})$, then $\delta(Y,\mathfrak{s}) = -\sigma(W)/8$.
\end{corollary}
\begin{proof}
Since $W$ is spin, there is a spin$^c$-structure $\mathfrak{s}$ for which $c_1(\mathfrak{s}) = 0$. Proposition \ref{prop:realise} then implies that $-\sigma(W)/8 \le \delta(Y,\mathfrak{s})$. The same argument applied to $\overline{W}$ gives $\sigma(W)/8 \le \delta(Y,\mathfrak{s})$.
\end{proof}

\subsection{Equivariant embeddings of $3$-manifolds in $4$-manifolds}\label{sec:app3}

Let $Y$ be a rational homology $3$-sphere equipped with an orientation preserving action of $G$. By an equivariant embedding of $Y$ into a $4$-manifold $X$, we mean an embedding $Y \to X$ such that the action of $G$ on $Y$ extends over $X$. We consider some existence and non-existence results for equivariant embeddings.

\begin{proposition}
Suppose that $Y$ is an integral homology $3$-sphere. Let $\mathfrak{s}$ be the unique spin$^c$-structure on $Y$ and assume that $G_{\mathfrak{s}}$ is a trivial extension. If $Y$ can be equivariantly embedded in $S^4$, then $\delta_{G,c}(Y,\mathfrak{s}) = \delta_{G,c}(\overline{Y},\mathfrak{s}) = 0$ for every non-zero $c \in H^*_G$.
\end{proposition}
\begin{proof}
If $Y$ embeds equivariantly in $S^4$, then we obtain an equivariant decomposition $S^4 = W_+ \cup_Y W_-$. Mayer--Vietoris and Poincar\'e--Lefschetz imply that $W_{\pm}$ are integral homology $4$-balls, hence are contractible by Whitehead's theorem. The result now follows from Proposition \ref{prop:zero}.
\end{proof}

\begin{example}
Let $Y = \Sigma(2,3,13)$, equipped with the involution $\tau$ obtained from viewing $Y$ as the branched double cover $\Sigma_2( T_{3,13})$. Then $Y$ embeds in $S^4$ \cite[Theorem 2.13]{bb}. On the other hand, $\delta_2(\overline{Y} , \mathfrak{s}) = -1$ by Proposition \ref{ex:36n+1}. Hence $Y$ can not be equivariantly embedded in $S^4$.
\end{example}

It is known that every $3$-manifold $Y$ embeds in the connected sum $\#^n (S^2 \times S^2)$ of $n$ copies of $S^2 \times S^2$ for some sufficiently large $n$ \cite[Theorem 2.1]{agl}. Aceto--Golla-Larson define the {\em embedding number} $\varepsilon(Y)$ of $Y$ to be the smallest $n$ for which $Y$ embeds in $\#^n( S^2 \times S^2)$. Here we consider an equivariant version of the embedding number. To obtain interesting results we need to make an assumption on the kinds of group actions allowed.

\begin{definition}
Let $G = \mathbb{Z}_p = \langle \tau \rangle$ where $p$ is a prime number. We say that a smooth, orientation preserving action of $G$ on $X = \#^n(S^2 \times S^2)$ is {\em admissible} if $H^2( X ; \mathbb{Z})^{\tau} = 0$, where $H^2(X ; \mathbb{Z})^{\tau} = \{ x \in H^2(X ; \mathbb{Z}) \; | \; \tau(x) = x \}$.
\end{definition}

One way of constructing admissible actions is as follows. Let $X$ be the $p$-fold cyclic cover of $S^4$, branched over an unknotted embedded surface $\Sigma \subset S^4$ of genus $g$. Then $X$ is diffeomorphic to $\#^{g(p-1)}(S^2 \times S^2)$ \cite[Corollary 4.3]{ak2} and the action of $\mathbb{Z}_p$ on $X$ as a cyclic branched cover is admissible (as can be seen from the proof of Theorem 9.3 in \cite{bar}).

Let $\mathbb{Z}_p = \langle \tau \rangle$ act on a rational homology $3$-sphere $Y$. We define the {\em equivariant embedding number $\varepsilon(Y,\tau)$} of $(Y,\tau)$ to be the smallest $n$ for which $Y$ embeds equivariantly in $\#^n (S^2 \times S^2)$ for some admissible $\mathbb{Z}_p$-action on $\#^n( S^2 \times S^2)$, if such an embedding exists. We set $\varepsilon(Y,\tau) = \infty$ if there is no such embedding.

Recall that the {\em double slice genus} \cite[\textsection 5]{lime} $g_{ds}(K)$ of a knot $K$ in $S^3$ is defined as the minimal genus of an unknotted compact oriented surface $S$ embedded in $S^4$ whose intersection with the equator $S^3$ is $K$. From the definition, it follows that $2g_4(K) \le g_{ds}(K) \le 2 g_3(K)$, where $g_3(K)$ is the $3$-genus of $K$.

\begin{proposition}\label{prop:emb1}
Let $Y = \Sigma_2(K)$ be the branched double cover of a knot $K$ and let $\tau$ be the covering involution on $Y$. Then $\varepsilon(Y , \tau) \le g_{ds}(K)$.
\end{proposition}
\begin{proof}
Let $S$ be an unknotted embedded surface in $S^4$ of genus $g_{ds}(K)$ intersecting the equator in $K$. Let $W$ be the double cover of $S^4$ branched along $S$. Then from \cite[Corollary 4.3]{ak2}, $W$ is diffeomorphic to $\#^{g_{ds}(K)} (S^2 \times S^2)$. The covering involution on $W$ is admissible since $W$ is a branched double cover of $S^4$. Clearly $Y = \Sigma_2(K) = \partial W$ embeds equivariantly in $W$ and so $\varepsilon(Y,\tau) \le g_{ds}(K)$.
\end{proof}

\begin{proposition}\label{prop:emb2}
Let $\mathbb{Z}_p = \langle \tau \rangle$ act orientation preservingly on an integral homology $3$-sphere $Y$. Let $j(Y,\tau)$ be the smallest positive integer such that $\delta_j(Y,\mathfrak{s},\tau,p) + \delta_j( \overline{Y},\mathfrak{s},\tau,p) = 0$, or $j(Y,\tau) = \infty$ if no such $j$ exists. Here $\mathfrak{s}$ is the unique spin$^c$-structure on $Y$. Then $\varepsilon(Y,\tau) \ge j(Y,\tau)$ if $p=2$ and $\varepsilon(Y,\tau) \ge 2j(Y,\tau)$ if $p$ is odd. 
\end{proposition}
\begin{proof}
To simplify notation we write $\delta_j(Y)$ for $\delta_j(Y,\mathfrak{s},\tau,p)$. Suppose that $Y$ embeds equivariantly in $X = \#^n( S^2 \times S^2)$, for an admissible action of $\tau$. Then we obtain an equivariant splitting $X = X_+ \cup_Y X_-$. Let $\mathfrak{t}$ be the unique spin structure on $X$. By uniqueness, $\mathfrak{t}$ is $\tau$-invariant and $\mathfrak{t}|_Y = \mathfrak{s}$. Corollary \ref{cor:deltasigma} applied to $X_{+}$ gives $\delta_j(Y) = -\sigma(X_+)/8$ for $j \ge b'_-(X_+)$ and $\delta_j(\overline{Y}) = \sigma(X_+)/8$ for $j \ge b'_+(X_+)$. Let $n' = n$ if $p=2$ or $n/2$ if $p$ is odd. Since $b'_{\pm}(X_+) \le b'_{\pm}(X) = n'$, we see that $\delta_j(Y) + \delta_j(\overline{Y}) = 0$ for $j \ge n'$. Hence $j(Y,\tau) \le n'$. Therefore $\varepsilon(Y,\tau) \ge j(Y,\tau)$ if $p=2$ and $\varepsilon(Y,\tau) \ge 2j(Y,\tau)$ if $p$ is odd. 
\end{proof}

\begin{example}
Let $Y = \Sigma(2,3,6n+1) = \Sigma_2(T_{3,6n+1})$ and equip $Y$ with the covering involution $\tau$. Then $g_3(T_{3,6n+1}) = 6n$, hence $\varepsilon(Y,\tau) \le 12n$, by Proposition \ref{prop:emb1}. From Proposition \ref{ex:36n+1}, we see that $j(Y,\tau) = 2n$ and so $\varepsilon(Y,\tau) \ge 2n$. So we have an estimate on the equivariant embedding number of the form
\[
2n \le \varepsilon( \Sigma(2,3,6n+1) , \tau ) \le 12n.
\]
Suppose that $n$ is odd. Then from \cite[Proposition 3.5]{agl}, the (non-equivariant) embedding number of $\Sigma(2,3,6n+1)$ is given by $\varepsilon( \Sigma(2,3,6n+1) )=10$. In particular, we see that $\varepsilon( \Sigma(2,3,6n+1) , \tau ) > \varepsilon( \Sigma(2,3,6n+1) )$ for all odd $n > 5$. Also, since we obviously have $\varepsilon( Y , \tau ) \ge \varepsilon(Y)$, we see that
\[
10 \le \varepsilon( \Sigma( 2 , 3 , 7) , \tau ) \le 12.
\]
In fact, we will now prove that $\varepsilon( \Sigma( 2 , 3 , 7) , \tau ) = 12$. Suppose that $Y = \Sigma(2,3,7)$ embeds equivariantly in $X = \#^n (S^2 \times S^2)$ for some admissible involution, where $n \le 12$. Then we obtain an equivariant splitting $X = X_+ \cup_{Y} X_-$. Since $Y$ is an integral homology sphere, the intersection forms on $X_{\pm}$ are unimodular. They are also even, since $X$ is spin. Moreover the Rochlin invariant of $Y$ is $1$. So the intersection forms of $X_\pm$ must contain at least one $E_8$ or $-E_8$ summand. Proposition \ref{prop:casson} implies that $\delta_j(Y) = 1$ for all $j \ge 0$ and Proposition \ref{ex:36n+1} implies that $\delta_j(\overline{Y})=0$ for $j=0,1$ and $\delta_j(\overline{Y})=-1$ for $j \ge 2$. Since $n \le 12$, Corollary \ref{cor:deltasigma} applied to $X_{\pm}$ then implies that the intersection form of $X_+$ must be of the form $\alpha H \oplus (-E_8)$ for some $\alpha \ge 2$ (where $H$ is the hyperbolic lattice) and similarly the intersection form of $X_-$ must be of the form $\alpha' H \oplus (E_8)$ for some $\alpha' \ge 2$. The intersection form of $X$ is then $(\alpha + \alpha' + 8)H$ and so $n = \alpha + \alpha' + 8 \ge 2+2+8 = 12$. This proves that
\[
\varepsilon( \Sigma( 2 , 3 , 7) , \tau ) = 12.
\]
\end{example}

\begin{example}
Let $Y = \Sigma(2,3,5) = \Sigma_2(T_{3,5})$ and equip $Y$ with the covering involution $\tau$. Then $g_3(T_{3,5}) = 4$, hence $\varepsilon(Y,\tau) \le 8$, by Proposition \ref{prop:emb1}. On the other hand, $\varepsilon(Y,\tau) \ge \varepsilon(Y)$ and $\varepsilon(Y) = 8$ by \cite[Proposition 3.4]{agl}, so $\varepsilon( \Sigma(2,3,5),\tau ) = 8$.
\end{example}

\bibliographystyle{amsplain}

\end{document}